\documentclass[reqno,12pt]{preprint}
\usepackage[marginparwidth=1in]{geometry}
\usepackage[full]{textcomp}
\usepackage[osf]{newtxtext}
\usepackage{cabin}
\usepackage{enumerate}
\usepackage{enumitem}
\usepackage{bbm}

\usepackage[varqu,varl]{inconsolata}
\usepackage[cal=boondoxo]{mathalfa}

\usepackage{comment}
\usepackage{hyperref}
\usepackage{breakurl}
\usepackage{mathrsfs}
\usepackage{booktabs}
\usepackage{caption}
\usepackage{stmaryrd}

\usepackage{tikz}
\usepackage{amsmath} 
\usepackage{mhequ} 
\usepackage{mhsymb} 
\usepackage{mhenvs} 
\usepackage{microtype}
\usepackage{amssymb} 
\usepackage{eufrak}

\usepackage{wasysym}
\usepackage{centernot}
\usepackage{wrapfig}

\DeclareMathAlphabet{\mathbb}{U}{jkpsyb}{m}{n}
\SetMathAlphabet{\mathbb}{bold}{U}{jkpsyb}{bx}{n}

\newcommand{\Qrat}{\mathbb{Q}}
\newcommand{\T}{\mathbb{T}}
\newcommand{\Mm}{\mathbb{M}}
\newcommand{\D}{\mathbb{S}}
\newcommand{\Zint}{\mathbb{Z}}

\newcommand{\Prob}{\mathbb{P}}

\newcommand{\Ch}{\mathbb{D}}

\newcommand{\cA}{\mathcal{A}}

\newcommand{\cD}{\mathcal{D}}

\newcommand{\cG}{\mathcal{G}}
\newcommand{\cH}{\mathcal{H}}

\newcommand{\cN}{\mathcal{N}}

\newcommand{\cS}{\mathcal{S}}

\newcommand{\cW}{\mathcal{W}}

\newcommand{\sK}{\mathfrak{K}}

\newcommand{\ft}{\mathfrak{t}}

\newcommand{\sw}{\mathfrak{w}}

\newcommand{\sz}{\mathfrak{z}}

\newcommand{\bw}{\mathrm{bw}}
\newcommand{\com}{\mathrm{c}}

\newcommand{\dd}{\mathrm{d}}      
\newcommand{\diam}{\mathrm{diam}}  
\newcommand{\dis}{\mathop{\mathrm{dis}}}
\newcommand{\Disc}{\mathop{\mathrm{Disc}}}

\newcommand{\M}{\mathrm{M1}}

\newcommand{\per}{\mathrm{per}}

\newcommand{\Sp}{\mathrm{sp}}

\newcommand{\uDelta}{\boldsymbol{\Delta}}

\renewcommand{\SS}{\mathscr{S}}
\newcommand{\ST}{\mathscr{T}}

\def\restr{\mathord{\upharpoonright}}

\renewcommand{\R}{\mathbb{R}}

\renewcommand{\N}{\mathbb{N}}
\renewcommand{\Q}{\mathbb{Q}}
\renewcommand{\T}{\mathbb{T}}
\renewcommand{\Z}{\mathbb{Z}}
\renewcommand{\E}{\mathbb{E}}
\renewcommand{\P}{\mathbb{P}}

\newcommand{\da}{\downarrow}
\newcommand{\ua}{\uparrow}
\newcommand{\uda}{\downarrow\mathrel{\mspace{-1mu}}\uparrow}
\newcommand{\dotp}{\bigcdot}
\newcommand{\llb}{\llbracket}
\newcommand{\rrb}{\rrbracket}

\newcommand{\pid}{\pi^\da}

\newcommand{\pidd}{\pi^{\da,\delta}}
\newcommand{\piud}{\pi^{\ua,\delta}}

\colorlet{darkblue}{blue!90!black}
\colorlet{darkred}{red!90!black}
\colorlet{dr}{red!90!black}

\tikzset{
        dot/.style={thin,circle,fill=gray,draw=black,inner sep=0pt,minimum size=2mm},
        vertex/.style={thin,circle,fill=black,draw=black,inner sep=0pt,minimum size=1mm},
	}

\makeatletter 
\newcommand*{\bigcdot}{}
\DeclareRobustCommand*{\bigcdot}{%
  \mathbin{\mathpalette\bigcdot@{}}%
}
\newcommand*{\bigcdot@scalefactor}{.5}
\newcommand*{\bigcdot@widthfactor}{1.15}
\newcommand*{\bigcdot@}[2]{%
  \sbox0{$#1\vcenter{}$}
  \sbox2{$#1\cdot\m@th$}%
  \hbox to \bigcdot@widthfactor\wd2{%
    \hfil
    \raise\ht0\hbox{%
      \scalebox{\bigcdot@scalefactor}{%
        \lower\ht0\hbox{$#1\bullet\m@th$}%
      }%
    }%
    \hfil
  }%
}
\makeatother

\makeatletter

\DeclareRobustCommand{\TitleEquation}[2]{\texorpdfstring{\StrLeft{\f@series}{1}[\@firstchar]$\if%
b\@firstchar\boldsymbol{#1}\else#1\fi$}{#2}}

\makeatother

\def\dash{\leavevmode\unskip\kern0.18em--\penalty\exhyphenpenalty\kern0.18em}
\def\slash{\leavevmode\unskip\kern0.15em/\penalty\exhyphenpenalty\kern0.15em}

\begin{document}

\title{The Brownian Web as a random \TitleEquation{\R}{R}-tree}
\author{G. Cannizzaro$^{1,2}$ and M. Hairer$^{1,3}$}

\institute{Imperial College London, SW7 2AZ, UK  \and University of Warwick, CV4 7AL, UK
\and EPFL, 1015 Lausanne, Switzerland\\
 \email{giuseppe.cannizzaro@warwick.ac.uk}\\
 \email{m.hairer@imperial.ac.uk}}

\maketitle

\begin{abstract}
Motivated by~\cite{CHbc}, we provide a construction of the Brownian Web~\cite{TW,FINR}, 
i.e.\ a family of coalescing Brownian motions starting from every point in $\R^2$ simultaneously,
as a random variable taking values in a space of (spatial) $\R$-trees. This gives a stronger topology
than the classical one {(i.e.\ Hausdorff convergence on closed sets of paths)}, 
thus providing us with more continuous functions of the Brownian Web
and ruling out a number of potential pathological behaviours.
Along the way, we introduce a modification of the 
topology of spatial $\R$-trees in~\cite{DL,BCK} which makes it a complete separable metric space and could be of 
independent interest. 
We determine some properties of the characterisation of the Brownian Web in this context (e.g.\ its box-counting dimension) 
and recover some which were determined in earlier works, such as duality, special points and convergence 
of the graphical representation of coalescing random walks. 
\end{abstract}

\setcounter{tocdepth}{2}       
\tableofcontents

\section{Introduction}

The Brownian Web is a random object that can be heuristically described as a collection of coalescing Brownian 
motions starting from every space-time point in $\R^2$, a typical realisation of which is displayed in Figure~\ref{fig:BW}. 
Its study originated in the PhD thesis of Arratia~\cite{A}, 
who was interested in the Voter model~\cite{Lig}, its dual, given by a family of (backward) coalescing random walks, 
and their diffusive scaling limit. 
It was rediscovered by T\'oth and Werner in~\cite{TW}, where they provided the first thorough construction, 
determined its main properties, and used it to introduce the so-called true self-repelling motion. 
A different characterisation was subsequently given in~\cite{FINR} where, by means of a new topology, 
a sufficient condition for the convergence of families of coalescing random walks was derived. 
Later on, further generalisations 
via alternative approaches appeared, 
e.g.\ in~\cite{NT} \dash motivated by the connection with Hastings--Levitov 
planar aggregation models \dash, in~\cite{BGS} \dash where the optimal convergence condition 
was obtained and a family of coalescing Brownian motions on the Sierpinski gasket was built \dash, 
and in~\cite{GSW} \dash where the Brownian Web was used to study the scaling limit of the genealogies 
of a population. 
For an account of further developments of the Brownian Web and 
the diverse contexts in which it emerged, we refer to the review paper~\cite{SSS}. 

\begin{figure}
\begin{center}
\includegraphics[width=14cm]{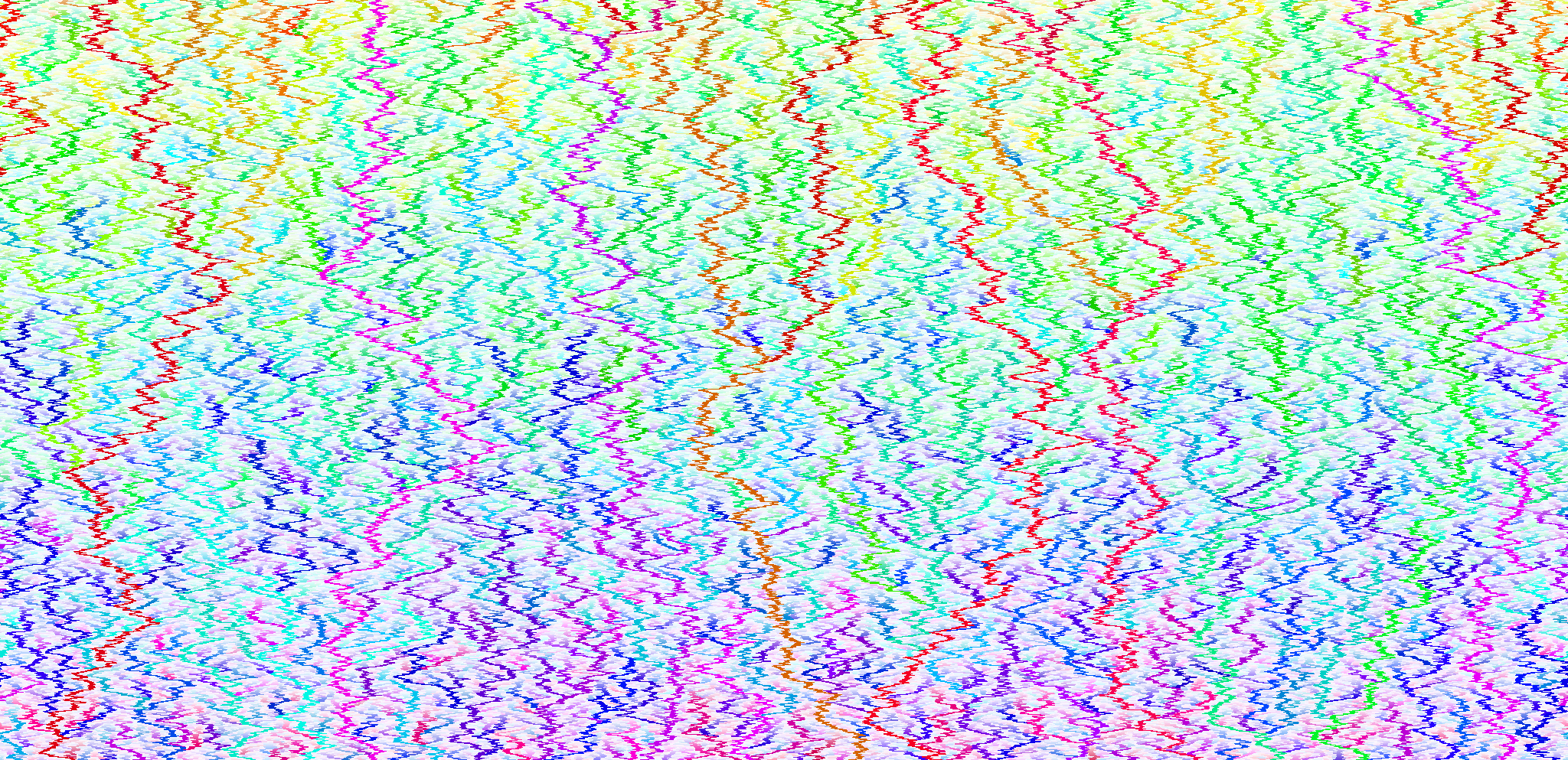}
\end{center}
\caption{A typical realisation of the Brownian web: coalescing Brownian trajectories emanate from
every point of the plane simultaneously. Trajectories are coloured according to their creation time\slash age.}
\label{fig:BW}
\end{figure}
In most (but not all, see e.g.~\cite{GSW})
 of these works, the Brownian Web is viewed as a random (compact) collection of paths $\cW$ 
in a suitable space of trajectories. Elements of $\cW$ are pairs $(t,\pi)$ with $t \in \R$ and
$\pi \colon \R \to \R$ such that furthermore $\pi(s) = \pi(t)$ for $s \ge t$.
{ In the present paper (similarly to~\cite{A3,GSW}), we focus instead on }
another of its characterising features, namely 
its {\it coalescence} or {\it tree structure}, clearly apparent in Figure~\ref{fig:BW}. 
The main motivation comes from the companion paper~\cite{CHbc} in which such a structure is 
used to construct  and study the Brownian Castle, 
a stochastic process whose value at a given point in $\R^2$ equals that of a Brownian motion indexed by the Brownian Web. 
Since the characteristics of the Brownian Castle are given by {\it backward} (coalescing) Brownian motions, 
in what follows we will (mainly) consider the case in which paths in $\cW$ run backward in time 
(the so-called backward Brownian Web~\cite{FINR}). 

To carry out this programme, we would like to view $\cW$ 
as a metric space with metric given by the {\it intrinsic distance}, namely the distance between 
two points $(t_i, \pi_i)$ is given by $t_1 + t_2 - 2\tau$, where $\tau$ is the largest time
such that $\pi_1(s) = \pi_2(s)$ for all $s \le\tau$.
More precisely, we view the Brownian Web as a (random) quadruplet 
$\zeta^\da_\bw\eqdef(\ST^\da_\bw,\ast^\da_\bw, d^\da_\bw,M^\da_\bw)$ such that
$(\ST^\da_\bw,\ast^\da_\bw, d^\da_\bw)$ is a pointed locally compact $\R$-tree, 
namely a connected locally compact metric space with no loops (see Definition~\ref{def:Rtree}) 
and such that $M^\da_\bw \colon \ST^\da_\bw \to \R^2$ is an embedding into $\R^2$.
(In the above identification of $\ST^\da_\bw$ with a set of elements of the type $(t,\pi)$, one 
would simply set $M^\da_\bw(t,\pi) = (t,\pi(t))$.)

{The goals of the present article are: identify a ``good'' space in which the 
quadruplet $\zeta^\da_\bw$ lives and in which we can uniquely characterise its law; 
determine a suitable topology under which such space is Polish 
and that allows for a manageable characterisation of its compact subsets; 
show that standard approximations to the Brownian Web converge in this (stronger) topology. 
Let us remark that the choice of the space and topology thereon is done in such a way that 
the Brownian Castle in~\cite{CHbc} is {\it continuous} (in a suitable sense) as a map 
from such space to the set of c\`adl\`ag functions.}

First, we introduce the space $\T^\alpha_\Sp$, $\alpha\in(0,1)$, whose 
elements are {\it spatial $\R$-trees}, i.e.\ quadruplets of the form $\zeta=(\ST,\ast,d,M)$ 
in which 
\begin{enumerate}[noitemsep]
\item $(\ST,\ast,d)$ is a pointed locally compact $\R$-tree,
\item $M$, the {\it evaluation map}, is a {locally little $\alpha$-H\"older continuous} map from $\ST$ to $\R^2$, i.e.
for all $K\subset\ST$ compact
\[
\lim_{\eps \to 0} \sup_{\sz \in K} \sup_{d(\sz,\sz') \le \eps} \|M(\sz)-M(\sz')\| / d(\sz,\sz')^\alpha = 0\,,
\]
\item $M$ is proper, i.e.\ the preimage of compact subsets is compact. 
\end{enumerate}
{In $\T^\alpha_\Sp$, we identify elements $\zeta=(\ST,\ast,d,M)$ and $\zeta'=(\ST',\ast',d',M')$ 
if there exists a bijective isometry $\phi\colon\ST\to\ST'$ such that $\phi(\ast)=\ast'$ and $M\circ\phi=M'$. 

Before commenting on the reason why we require the previous properties and introducing a metric on 
$\T^\alpha_\Sp$ under which it is complete and separable, we state and prove (one of) the main 
result of the present paper, namely the characterisation of the law of the Brownian Web (Tree).

Given any finite tuple $(z_1,\ldots,z_m) \subset \R^2$ and writing $z_i = (t_i, x_i)$, 
let $X$ be the unique (in law) $\R^m$-valued continuous 
martingale (but with time running backwards, i.e.\ $\E (X_s \,|\,\CF_t) = X_t$ for $s < t$ where $\CF_t = \sigma\{X_r\,:\, r\ge t\}$) such that, setting $\tau_{ij} = \sup\{t \le t_i \wedge t_j\,:\, 
X_i(t) = X_j(t)\}$ (in particular $\tau_{ii} = t_i$), one has
\begin{equ}
\lim_{t \to \infty} X_i(t) = x_i\;,\qquad
\scal{X_i, X_j}(t) = 
\left\{\begin{array}{cl}
	0 & \text{if $t \ge \tau_{ij}$,} \\
	\tau_{ij}-t & \text{otherwise,}
\end{array}\right.
\end{equ}
so that in particular $X_i(t)=x_i$ for all $t\geq t_i$. 
In other words, the $X_i$'s are Brownian motions starting at $x_i$ at time $t_i$ that are independent until the first 
time they meet, at which point they coalesce.
We then set $d_{ij} = 2\tau_{ij}-t_i-t_j$.

\begin{theorem}\label{thm:Main}
There exists a unique (in law) random element $\zeta=(\ST,\ast,d,M)$ of $\T^\alpha_\Sp$ with the following properties
\begin{enumerate}
\item $M(\ast) = (0,0)$,
\item for any fixed $z \in \R^2$, there almost surely exists a unique $\sz_z \in \ST$ such that 
$M(\sz_z) = z$,
\item for any finite tuple $(z_1,\ldots,z_m) \subset \R^2$, the law of $\{d(\sz_{z_i},\sz_{z_j})\}_{i,j \le m}$
is the same as the law of $\{d_{ij}\}_{i,j\le m}$ constructed just above,
\item for any fixed countable dense subset $\CD \subset \R^2$, the set $\{\sz_z\,:\, z \in \CD\}$ is almost surely 
dense in $\ST$.
\end{enumerate}
\end{theorem}

Thanks to this theorem, the \textit{backward Brownian Web Tree} (which will be defined in Definition~\ref{def:BW}) 
is the random variable $\zeta^\da_\bw\eqdef(\ST^\da_\bw,\ast^\da_\bw, d^\da_\bw,M^\da_\bw)$, 
whose law is uniquely characterised by points 1-4 above. 

\begin{proof}
The existence of a random variable satisfying these properties follows by the existence part of Theorem~\ref{thm:BW} below.
Regarding uniqueness, it suffices to note that properties 2 and 3 guarantee that any
 two candidates $\zeta = (\ST,\ast,d,M)$ and $\bar \zeta = (\bar\ST,\bar\ast,\bar d,\bar M)$ can be
coupled in such a way that $d(\sz_{z},\sz_{z'}) = \bar d(\bar\sz_{z},\bar\sz_{z'})$ almost surely, for all $z,z' \in \CD$. 
It follows from property 4 that, setting $\phi (\sz_z) = \bar\sz_z$, this extends to an isometry 
$\phi \colon \ST \to \bar \ST$. Since $M(\sz_z) = \bar M(\bar \sz_z) = z$ by assumption, the continuity of $M$ and $\bar M$
immediately implies that $M = \bar M\circ \phi$ so that, since one also has $\phi(\ast) = \bar \ast$ by property 1, 
$\zeta = \bar\zeta$ in $\T^\alpha_\Sp$.
\end{proof}

We will give the precise definition of the metric we endow $\T^\alpha_\Sp$ with in Section~\ref{sec:SpTrees}.  
Heuristically, under this metric,} 
a sequence $\{\zeta_n=(\ST_n,\ast_n,d_n,M_n)\}_n$ converges to $\zeta=(\ST,\ast,d,M)$ provided that 
\begin{enumerate}[noitemsep]
\item the pointed $\R$-trees converge in a local version of the Gromov--Hausdorff topology\footnote{For an 
introduction in the case of general metric 
and length spaces we refer to the monograph~\cite{BBI}, 
and to~\cite{E08} for the specific case of $\R$-trees.},
\item the evaluation maps $M_n$ converge to $M$ locally uniformly {and in $\alpha$-H\"older sense},
\item the size of the preimage of compact balls via the $M_n$ converges.  
\end{enumerate}
Let us comment on what these conditions entail and why we require them. 
The first condition takes into account the metric structure of the $\R$-trees and morally says that 
the metrics $d_n$ converge to $d$. 
In the present context, this ensures that couples of distinct 
paths which are close also coalesce approximately at the same time. 
The locally uniform convergence of the evaluation maps can be thought of as a control 
over the sup-norm distance of paths and 
is somewhat similar in spirit to that of~\cite{FINR}. The control on the $\alpha$-H\"older distance 
was added in order to make the space $\T^\alpha_\Sp$ complete. 
Indeed, as pointed out in~\cite[Remark 3.2]{BCK}, 
if we remove Hölder continuity of $M$ from the definition of $\T^\alpha_\Sp$ and from
its metric, then it would no longer be complete (see Remark~\ref{rem:Holder} below). 
We require the evaluation maps to be proper as we want to prevent the existence of { infinite 
sequences of points that are all pairwise of order $1$ distance apart in $\ST$ 
but whose image under $M$ comes infinitely often arbitrarily close to a given point in $\R^2$} 
(see Remark~\ref{rem:properness} for more details). Since we need this property to be preserved 
when taking limits, we included the third condition above. 

We stress once again that the definition of our topology and the related definition of the 
Brownian Web given in the present paper, is motivated by the 
construction of the Brownian Castle in~\cite{CHbc}.  
In particular, while if we only wanted to provide an $\R$-tree characterisation of the Brownian Web 
the first two points above would have been sufficient, 
the control on the H\"older regularity and the properness will play a crucial role in~\cite{CHbc}. 
Indeed, the covariance of the Brownian motion indexed by the Brownian Web, $\ST^\da_\bw$, 
is fully determined by the metric structure of the $\R$-tree. 
Such a structure is related to the Euclidean distance on $\R^2$ via the evaluation map, which, if H\"older continuous,
provides a mean of comparison between the two metrics. 
On the other hand, the Brownian Castle is a random map on $\R^2$, so that we need a right inverse 
to the evaluation map $M$, called {\it tree map} (see~\cite[Definition 2.12]{CHbc}), 
which assigns to every point in $\R^2$ its ``rightmost'' preimage in the tree. 
In order to retain any control on the regularity properties of the tree map, properness turns out to be essential. {More precisely, the ``Brownian Castle'' studied in \cite{CHbc}
is obtained by taking a Brownian motion indexed by a Brownian Web Tree and then 
view it as a map $\R^2 \to \R$ by precomposing it with the tree map. 
Properness is essential to guarantee that this procedure is stable under approximations.}

\subsection{Alternative topologies and relation to previous characterisations}

Over the years, a variety of topologies on spaces of $\R$-trees  have been considered. 
The one outlined above is similar to those in~\cite{DL,BCK}, with the additional condition about the 
H\"older continuity and properness of the evaluation map. In~\cite{DGP, KL, ALW, GSW}, the authors 
introduce so-called marked metric measure spaces, formed of triplets $(\ST,d,\mu)$ 
in which $(\ST,d)$ is a metric space and $\mu$ is a locally finite measure on $\ST\times I$, for $I$ 
a complete and separable metric space. 
The measure $\mu$ should be thought of as equal to $\nu(\dd x)\otimes\delta_{\kappa(x)}(\dd u)$ 
for $\nu$ a locally finite measure on $\ST$ and $\kappa\colon \ST \to I$ a ``mark function''. 
Upon taking $I=\R^2$, the mark function plays a similar role to the evaluation map $M$ above. 
In our context a natural choice of measure $\nu$ would be the length measure, but this 
is only {\it $\sigma$-finite} and not {\it locally finite} as the above references require. 
In principle, we could artificially add a locally finite measure (necessarily with full support), 
but this would cause additional complications for no benefit in our setting. 

With respect to the classical construction of the Brownian Web, in terms of a family of paths, 
notice that it is not always possible to view a family of paths as an $\R$-tree 
(trivially, the paths might not be coalescing) and, conversely, there is 
no {\it canonical} way to associate a collection of paths to a generic spatial $\R$-tree (think 
of the case in which segments in the tree backtrack so that they cannot be viewed as 
functions of one coordinate). In Definition~\ref{def:CharTree} below,
we define a subset  $\Ch^\alpha_\Sp \subset \T^\alpha_\Sp$ of ``directed trees'' for which 
the association is meaningful and 
prove that, as suggested by the heuristic description above, 
our topology is strictly finer than that in~\cite{FINR} (see Proposition~\ref{p:MapTopo}). 
While this ensures that many of the results obtained for the Brownian Web (existence of a dual, 
its properties, special points) can be translated to the present setting (see Section~\ref{sec:DBW}), 
convergence statements in $\T^\alpha_\Sp$ do not follow from those previously 
established.  
This is remedied in Section~\ref{sec:ConvBW}, where a convergence criterion to $\zeta^\da_\bw$ 
is derived.

As shown in~\cite{CHbc}, there are two main advantages of the characterisation of the 
Brownian Web outlined above. First, it allows to preserve information on the {\it intrinsic metric} 
on the set of trajectories, which in turn is at the basis of the properties and the proof of 
the universality statement for the Brownian Castle in~\cite[Theorem 1.4]{CHbc}. 
Moreover, the $\R$-tree structure (together with local compactness) automatically 
endows $\ST^\da_\bw$ with a $\sigma$-finite 
length measure (see~\cite[Section 4.5.3]{E08}) that can be useful in many contexts and, for example, 
could provide a more direct construction of the {\it marked Brownian Web} of~\cite{FINRb}\footnote{The marked Brownian Web is built from a Poisson marking of the double points of the Brownian Web, i.e. points from which two 
trajectories originate. As these points correspond to points in the skeleton of the dual web, the marking is obtained by considering the Poisson random measure with intensity given by the length measure on
the dual web.}.
Moreover, this is the measure that gives the white noise arising in the construction of the Brownian Castle, which
is also one reason why we do not attempt to distort the tree in a way that could potentially lead to better compactness
properties.

At last, {we mention that in \cite[Sec.~4.2]{A3}, Aldous sketched the construction of a random 
$\R$-tree corresponding to a mesh of the Brownian web as a limit of specific approximations
with nice exchangeability properties. His work builds on a very general and rather
``soft'' construction, but provides relatively little information about the random tree obtained in this way. 
(His random trees are tree-like closed subsets of $\ell^1$, so for example even local compactness is not
guaranteed.)
On the other hand, the present paper provides a global construction of the Brownian Web 
(in space-time, as opposed to the one at fixed times of~\cite{GSW}) as a random $\R$-tree 
satisfying a number of useful properties.}

Many fascinating random $\R$-trees have been studied, such as 
Aldous's CRT in~\cite{A1,A2,A3}, the L\'evy and 
Stable trees of Le Gall and Duquesne and their connection to superprocesses~\cite{DL}, 
and display interesting relations to important
statistical mechanics models, e.g.\
the Brownian Map and random plane quadrangulations~\cite{LG,Mi}, 
the scaling limit of the Uniform Spanning Tree and SLE~\cite{Schramm,BCK}, just to mention a few. 
As expected, the law of the Brownian Web as a random $\R$-tree 
is different from those alluded to above (see Corollary~\ref{cor:BCdim} and Remark~\ref{rem:BCdim}) 
but it would be interesting to explore further this new interpretation in light of the aforementioned works 
to see if extra properties of the Brownian Web itself or the Brownian Castle of~\cite{CHbc} can be derived.

\subsection{Outline of the paper}

In Section \ref{sec:Trees}, we collect all the preliminary results and constructions concerning $\R$-trees 
which will be needed throughout the paper. 
After recalling their basic definitions and geometric features, we introduce, for 
$\alpha\in(0,1)$, the spaces $\T^\alpha_\Sp$, of spatial $\R$-trees, and their ``directed'' subset $\Ch^\alpha_\Sp$. 
We define a metric which makes them complete and separable, and identify a necessary and sufficient condition for 
a subset to be compact. 
In Section~\ref{s:Topo}, we compare the metric above and that of~\cite{FINR}, 
and show that the former is stronger than the latter. 

Section~\ref{sec:BW-BCM} is devoted to the Brownian Web and its periodic version~\cite{CMT}. 
At first (Section~\ref{sec:BW}), we provide another characterisation of its law on $\Ch^\alpha_\Sp$ 
and determine some of its properties as an $\R$-tree, such as box covering dimension and relation to~\cite{FINR}. 
Then, we state and prove a convergence criterion (Section~\ref{sec:ConvBW}) 
and, in Section~\ref{sec:DBW}, we describe its dual and the so-called ``special points''. 

At last, in Section~\ref{sec:DWT} we first show how to make sense of the graphical construction 
of a system of coalescing backward random walks (and its dual) in the present context and 
conclude by deriving its scaling limit.

\subsection*{Notation} 

We will denote by $|\cdot|_e$ the usual Euclidean norm on $\R^d$, $d\geq 1$, and 
adopt the short-hand notation $|x|\eqdef|x|_e$ and $\|x\|\eqdef|x|_e$ for $x\in\R$ and $\R^2$ respectively. 
Let $(\ST,d)$ be a metric space. We define the Hausdorff distance $d_H$ between two non-empty subsets $A,\,B$ of $\ST$ as 
\begin{equ}
d_H(A,B)\eqdef\inf\{\eps\colon A^\eps\subset B\text{ and } B^\eps\subset A\}
\end{equ}
where $A^\eps$ is the $\eps$-fattening of $A$, i.e. $A^\eps=\{\sz\in\ST\colon \exists\, \sw\in A\text{ s.t. } d(\sz,\sw)< \eps\}$.

Let $(\ST,d,\ast)$ be a pointed metric space, i.e. $(\ST,d)$ is as above and $\ast \in \ST$, 
and let $M\colon \ST\to\R^d$ be a map. For $r>0$ and $\alpha\in(0,1)$, we define the $\sup$-norm 
and $\alpha$-H\"older norm of $M$ restricted to a ball of radius $r$ as
\begin{equ}
\|M\|^{(r)}_\infty\eqdef\sup_{\sz\in B_d(\ast,r]}|M(\sz)|_e\,,\qquad \|M\|^{(r)}_\alpha\eqdef\sup_{\substack{\sz,\sw\in B_d(\ast,r]\\d(\sz,\sw)\leq 1}}\frac{|M(\sz)-M(\sw)|_e}{d(\sz,\sw)^\alpha}\,.
\end{equ}
where $B_d(\ast,r]\subset\ST$ is the closed ball of radius $r$ centred at $\ast$, and, for $\delta>0$, 
the modulus of continuity as
\begin{equation}\label{e:MC}
\omega^{(r)}(M,\delta)\eqdef \sup_{\substack{\sz,\sw\in B_d(\ast,r]\\d(\sz,\sw)\leq \delta}} |M(\sz)-M(\sw)|_e\,.
\end{equation}
In case $\ST$ is compact, in all the quantities above, the suprema are taken over the whole space $\ST$ and 
the dependence on $r$ of the notation will be suppressed. 
Moreover, we say that a function $M$ is (locally) {\it little $\alpha$-H\"older continuous} if for all $r>0$, 
$\lim_{\delta\to 0} \delta^{-\alpha} \omega^{(r)}(M,\delta)=0$. 

Let $I\subseteq \R$ be a compact interval and $D(I,\R_+)$ be the space of c\`ad\`ag functions on $I$ with values in 
$\R_+\eqdef[0,\infty)$, endowed with the M1 topology that we now introduce. 
For $f\in D(I,\R_+)$, denote by $\Disc(f)$ the set of discontinuities of $f$ and by 
$\Gamma_f$ its completed graph, i.e. the graph of $f$ to which all the vertical segments joining the points 
of discontinuity are added. Order $\Gamma_f$ by saying that $(x_1,t_1)\leq (x_2,t_2)$ if either $t_1<t_2$ 
or $t_1=t_2$ and $|f(t_1^-)-x_1|\leq |f(t_1^-)-x_2|$.
Let $P_f$ be the set of all parametric representations of $\Gamma_f$, which is the set of all non-decreasing 
(with respect to the order on $\Gamma_f$) functions $\sigma_f\colon I\to \Gamma_f$. 
Then, if $I$ is bounded, we set 
\begin{equ}
\hat d^\com_{\M}(f,g)\eqdef 1\wedge\inf_{\sigma_{f},\sigma_g} \|\sigma_{f}-\sigma_{g}\|_\infty
\end{equ}
(with $\|\cdot\|_\infty$ denoting the supremum norm)
and $d^\com_{\M}(f,g)$ to be a topologically equivalent metric with respect to which 
$D(I,\R_+)$ is complete (see~\cite[Sec.~12.8]{Whitt} for an example of metric which makes the space complete). 
{If instead $I=[-1,\infty)$, we denote by $f^{(t)}$ the restriction of $f$ to $[-1, t]$ and define
\begin{equ}[def:M1metric]
d_{\M}(f,g)\eqdef \int_0^\infty e^{-t} \big( 1\wedge d_\M^{\com}(f^{(t)},g^{(t)})\big)\,\dd t\,.
\end{equ}
which is well defined in view of Theorem 12.9.2 and eq. (9.1) in~\cite{Whitt}.}

At last, we will write $a\lesssim b$ if there exists a constant $C>0$ such that $a\leq C b$  
and $a\approx b$ if $a\lesssim b$ and $b\lesssim a$.

\subsection*{Acknowledgements}

{\small
The authors would like to thank the referees for their very detailed reports which helped to improve the presentation of the paper.
In particular, we are grateful to the referee who suggested the formulation of Theorem~\ref{thm:Main} as stated.
GC would like to thank the Hausdorff Institute in Bonn for the kind hospitality during the programme 
``Randomness, PDEs and Nonlinear Fluctuations'', where part of this work was carried out. 
GC gratefully acknowledges financial support via the EPSRC grant EP/S012524/1 and the UKRI FL Fellowship 
MR/W008246/1.
MH gratefully acknowledges financial support from the Leverhulme trust via a Leadership Award, 
the ERC via the consolidator grant 615897:CRITICAL, and the Royal Society via a research professorship. 
}

\section{Preliminaries}\label{sec:Trees}

In this section, we gather all the results on $\R$-trees which will be necessary in the sequel. 
At first, we summarise some of their geometric properties.  

\subsection{\TitleEquation{\R}{R}-trees in a nutshell}\label{sec:RTrees}

Let us begin by recalling the definition of $\R$-tree given in~\cite[Definition 2.1]{DL}.

%

\begin{definition}\label{def:Rtree}
A metric space $(\ST,d)$ is an $\R$-tree if for every $\sz_1,\sz_2\in\ST$ 
\begin{enumerate}[noitemsep]
\item there is a unique isometric map $f_{\sz_1,\sz_2} :[0, d(\sz_1,\sz_2)]\to\ST$ such that $f_{\sz_1,\sz_2}(0)=\sz_1$ and 
	$f_{\sz_1,\sz_2}(d(\sz_1,\sz_2))=\sz_2$,
\item for every continuous injective map $q\colon [0,1] \to \ST$ such that $q(0)=\sz_1$ and $q(1)=\sz_2$, one has 
\begin{equ}
q([0,1])=f_{\sz_1,\sz_2}([0,d(\sz_1,\sz_2)])\,.
\end{equ}
\end{enumerate}
A {\it pointed $\R$-tree} is a triple $(\ST,\ast,d)$ such that $(\ST,d)$ is an $\R$-tree and $\ast \in \ST$. 
\end{definition}

\begin{remark}
We do not call such spaces {\it rooted} because {in the sequel we will also be considering general 
metric spaces for which there is no notion of root.}
\end{remark}

For an $\R$-tree $(\ST,d)$ and any two points $\sz_1,\sz_2\in\ST$, we define the {\it segment joining $\sz_1$ and $\sz_2$} 
as the range of the map $f_{\sz_1,\sz_2}$ and denote it by $\llb\sz_1,\sz_2\rrb$. 
{For every three points $\sz_1,\sz_2,\sz_3\in\ST$ there exists a unique point $\sw\in\ST$ such that 
$\llb \sz_1,\sz_3\rrb\cap\llb \sz_2,\sz_3\rrb\cap\llb \sz_1,\sz_3\rrb=\{\sw\}$. 
We call $\sw$, the {\it projection of $\sz_i$ onto $\llb\sz_j,\sz_k\rrb$}, for $i,k,j$ distinct elements of $\{1,2,3\}$ 
(see for example~\cite[Chapter 2.1]{Ch} for basic properties of $\R$-trees)}. 

\begin{definition}{\cite[Definition 2]{CMS}}
Let $(\ST,d)$ be an $\R$-tree and $r>0$. A segment $\llb\sz_1,\sz_2\rrb\subset\ST$ has {\it $r$-finite branching} if 
the set of points $\sw\in \llb\sz_1,\sz_2\rrb$ which are the projection of some point $\sz\in\ST$ onto $\llb\sz_1,\sz_2\rrb$
with $d(\sz,\sw)\geq r$ is finite. An $\R$-tree $\ST$ is said to have $r$-finite branching if every segment of $\ST$ does. 
\end{definition}

Given $\sz\in\ST$, the number of connected components of $\ST\setminus\{\sz\}$ is the {\it degree of $\sz$}, $\deg(\sz)$ in short. 
A point of degree $1$ is an {\it endpoint}, of degree $2$, 
an {\it edge point} and  if the degree is $3$ or higher, a {\it branch point}. 
The following lemma is taken from~\cite[Lemma 3]{CMS}. 

\begin{lemma}\label{l:Rtrees}
Let $(\ST,d)$ be an $\R$-tree, $\sz_0\in\ST$ and let $\SS$ be a dense subset of $\ST$. The following statements hold:
\begin{enumerate}[noitemsep]
\item  If $\sz\in\ST$ is not an endpoint for $\ST$, then there exists $\sw\in\SS$ such that $\sz\in\llb\sz_0,\sw\rrb$.
\item If $\SS$ is a sub-$\R$-tree of $\ST$, i.e. $(\SS,d)$ is itself an $\R$-tree,  then every point of $\ST\setminus\SS$ is an endpoint for $\ST$. 
\end{enumerate}
\end{lemma}

Notice that the connected components of $\ST\setminus\{\sz\}$ are themselves $\R$-trees, i.e.\ subtrees of $\ST$, 
and they are called {\it directions at $\sz$}. 

\begin{definition}{\cite[Definition 1]{CMS}}
Let $(\ST,d)$ be an $\R$-tree, $\sz\in\ST$ and $\{\ST_i\,:\,i\in I\}$, where $I$ is an index set,  the set of directions at $\sz$.
For $r>0$, we say that $\ST_i$ {\it has length $\geq r$} if there exists $\sw\in\ST_i$ such that $d(\sz,\sw)\geq r$. 
The $\R$-tree $\ST$ is {\it $r$-locally finite at $\sz$} if the set of all directions at $\sz$ of length $\geq r$ is finite, and it is 
{\it $r$-locally finite} if it is  $r$-locally finite at $\sz$ for every $\sz\in\ST$. 
\end{definition}

An important notion for us in the context of $\R$-trees, is that of {\it end}. To introduce it, we follow~\cite[Chapter 2.3]{Ch} 
{ (see also~\cite[Section 2]{EV})}.
A subset $L$ of an $\R$-tree $\ST$ is linear if it is isometric to an interval of $\R$, which could be either bounded 
 or unbounded. 
For $\sz\in\ST$, we write $L_\sz$ for an arbitrary linear subset of $\ST$ having $\sz$ as an endpoint
and we say that $L_\sz$ is a {\it $\ST$-ray from $\sz$} if it is maximal for inclusion. 
We also say that rays $L_\sz$ and $L_{\sz'}$ are equivalent if there exists 
$\sw\in\ST$ such that $L_\sz\cap L_{\sz'}$ is a ray from $\sw$. 
The equivalence classes of $\ST$-rays are the {\it ends} of $\ST$. Clearly, every endpoint determines an end for $\ST$ 
and we will refer to them as {\it closed ends}, while the remaining ends will be called {\it open}. 
By~\cite[Lemma 3.5]{Ch}, for every $\sz\in\ST$ and every open end $\dagger$ of $\ST$, there exists a unique
$\ST$-ray from $\sz$ representing $\dagger$ which we will denote by $\llb\sz,\dagger\rangle$. 
Moreover we say that $\dagger$ is an {\it open end with (un-)bounded rays} if for every $\sz\in\ST$, the map 
$\bar\iota_\sz: \llbracket\sz,\dagger\rangle\to\R_+$ given by 
\begin{equation}\label{e:iota}
\bar\iota_\sz(\sw)=d(\sz,\sw)\,,\qquad\sw\in\llb\sz,\dagger\rangle
\end{equation}
is (un-)bounded.

We conclude this section by showing how the geometric structure of an $\R$-tree is intertwined with its 
metric properties. The following statements summarise (or are easy consequences of) results in~\cite[Theorem 4.14]{Ch},
~\cite[Theorem 2.5.28]{BBI} and~\cite[Theorem 2, Proposition 5]{CMS}. 

\begin{theorem}\label{thm:Rtrees}
The completion of an $\R$-tree is an $\R$-tree
and an $\R$-tree is complete if and only if every open end has unbounded rays.
Let $(\ST,d)$ be a locally compact complete $\R$-tree, then
\begin{enumerate}[label=(\alph*),noitemsep]
\item $\ST$ is proper, i.e. every closed bounded subset is compact,
\item $\ST$ is $r$-locally finite and has $r$-finite branching for every $r>0$,
\item $\ST$ has countably many branch points and every point has at most countable degree. 
\end{enumerate}
\end{theorem}

\subsection{Spatial \TitleEquation{\R}{R}-trees}\label{sec:SpTrees}

Now that we discussed geometric features of $\R$-trees 
we are ready to study the metric properties of the space of $\R$-trees. 
{We will focus on the so-called $\alpha$-spatial $\R$-trees, which is a subset of the 
space of $\alpha$-spatial metric spaces that we now introduce. }

\begin{definition}\label{def:SpTree}
Let $\alpha\in(0,1)$. The space of {\it pointed $\alpha$-spatial metric space} $\Mm^\alpha_\Sp$ is the set 
of equivalence 
classes\footnote{The collection of all quadruplets as described here is not a set, but since every {metric space} $\ST$ in which bounded closed subsets are compact has the cardinality of the continuum, one can see that the collection of equivalence classes is indeed set-sized.} of quadruplets $\zeta=(\ST,\ast, d, M)$ where
\begin{itemize}[noitemsep, label=-]
\item $(\ST,\ast,d)$ is a complete, separable 
pointed metric space {such that every bounded closed subset is compact}, 
\item $M$, the {\it evaluation map}, is a locally little $\alpha$-H\"older continuous proper\footnote{Namely such that $\lim_{\eps \to 0} \sup_{\sz \in K} \sup_{d(\sz,\sz') \le \eps} \|M(\sz)-M(\sz')\| / d(\sz,\sz')^\alpha = 0$ for every compact $K$ and the preimage of every compact set is compact. } map 
from $\ST$ to $\R^2$. For any point $\sz\in\ST$, we define the projections $M_t(\sz),\,M_x(\sz)\in\R$ as 
$M(\sz)= (M_t(\sz),M_x(\sz))\in\R^2$. 
\end{itemize}
We identify $\zeta$ and $\zeta'$ 
if there exists a bijective isometry $\phi:\ST\to\ST'$ such that $\phi(\ast)=\ast'$ and $M'\circ\phi\equiv M$, in short 
(with a slight abuse of notation) $\varphi(\zeta)=\zeta'$.
We denote by $\T^\alpha_\Sp$ the subset of $\Mm^\alpha_\Sp$ whose elements $\zeta$ are 
such that $(\ST,\ast,d)$ is an $\R$-tree\footnote{Note that by Theorem~\ref{thm:Rtrees}(a) in any complete locally compact 
$\R$-tree, closed bounded subsets are compact.}.  
\end{definition}

\begin{remark}\label{rem:Periodic}
We will also consider situations in which the map $M$ is $\R\times\T$-valued, where 
$\T\eqdef\R/\Zint$ is the torus of size $1$ endowed with the usual periodic metric
$d(x,y)= \inf_{k \in \Z} |x-y+k|$. 
We denote by $\Mm^\alpha_{\Sp,\per}$ the space of periodic pointed $\alpha$-spatial metric spaces. 

In what follows, we will denote subsets of $\Mm^\alpha_\Sp$ and $\Mm^\alpha_{\Sp,\per}$ with 
the same notation except for the addition of a subscript ``$\per$'', standing for {\it periodic}, in the latter case.
It will always be immediate to see how the definitions, statements and proofs need to be adapted in order to hold 
not only for the space $\cS$ we are considering but also for its periodic counterpart $\cS_\per$. 
\end{remark}

For any spatial metric space $\zeta=(\ST,\ast,d,M)$,
%
we introduce the {\it properness map} 
$b_\zeta\colon \R\to \R_+$, a map whose role is to ``quantify'' the properness of $M$. 
For $r<0$, $b_\zeta(r) = 0$, while for $r\ge 0$ we set
\begin{equation}\label{def:PropMap}
b_\zeta(r)\eqdef \sup_{\sz\,:\,M(\sz)\in\Lambda_r} d(\ast,\sz)\;,
\end{equation}
where $\Lambda_r\eqdef [-r,r]^2\subset\R^2$ and in the periodic case $\Lambda_r=\Lambda_r^\per\eqdef [-r,r]\times\T$. 

\begin{lemma}\label{l:CadlagPropMap}
Let $\alpha\in(0,1)$ and $\zeta=(\ST,\ast,d,M)\in\Mm^\alpha_\Sp$. 
Then, the properness map $b_\zeta$ in~\eqref{def:PropMap} is non-decreasing and c\`adl\`ag. 
\end{lemma}
\begin{proof}
The function $b_\zeta$ is non-decreasing by construction, so that at every point $r>0$ it admits left and right limits. 
To show it is c\`adl\`ag, it suffices to prove that $\lim_{s\da r} b_\zeta(s)=b_\zeta(r)$. 

Notice that, for every $s>0$, since $M$ is proper and 
$\Lambda_s$ is closed, there exists $\sz_s\in M^{-1}(\Lambda_s)$ such that $b_\zeta(s)=d(\ast,\sz_s)$. 
Let $s_n$ be a sequence decreasing to $r$ and, without loss of generality, assume 
$M(\sz_{s_n})\in\Lambda_{s_n}\setminus\Lambda_r$. Using the properness of $M$ once again, 
$M^{-1}(\Lambda_{s_0})$ is compact so that $\{\sz_{s_n}\}_n\subset M^{-1}(\Lambda_{s_0})$ 
admits a converging subsequence. 
Let $\bar\sz$ be a limit point. By construction, $d(\ast,\sz_{s_n})\geq d(\ast,\sz_r)$ for all $n$, therefore 
$d(\ast,\bar\sz)\geq d(\ast,\sz_r)$. But $M(\bar\sz)\in\Lambda_r$ since $M$ is continuous, so 
$d(\ast,\bar\sz)\leq  d(\ast,\sz_r)$. {It follows that 
$b_{\zeta_n}(s_n)=d(\ast,\sz_{s_n})\to d(\ast,\bar\sz)=d(\ast,\sz_r)= b_{\zeta_n}(r)$ which completes 
the proof of the statement. }
\end{proof}

To turn $\Mm^\alpha_\Sp$ into a Polish space, we proceed similarly to~\cite{BCK}, 
but we introduce two conditions taking into account the H\"older regularity and the properness 
of $M$ respectively. 
Recall first that a correspondence $\CC$ between two metric spaces $(\ST,d)$, $(\ST',d')$ is a subset of 
$\ST\times\ST'$ such that for all $\sz\in\ST$ there exists at least one $\sz'\in\ST'$ for which $(\sz,\sz')\in\CC$ and vice versa. 
The {\it distortion} of a correspondence $\CC$ is given by 
\begin{equ}[e:dis]
\dis \CC\eqdef \sup\{|d(\sz,\sw)-d'(\sz',\sw')|\,:\, (\sz,\sz'),\,(\sw,\sw')\in\CC\}\;,
\end{equ}
and allows to give an alternative characterisation  of the 
Gromov-Hausdorff metric. 

{Let $\alpha\in(0,1)$ and $\zeta=(\ST,\ast,d,M)$, $\zeta'=(\ST',\ast',d',M')\in\Mm^\alpha_\Sp$. 
Let $\CC$ be a correspondence between $\ST$ and $\ST'$. We set
\begin{equation}\label{e:MetC}
\begin{split}
\uDelta^{\com,\CC}_\Sp(\zeta,\zeta') &\eqdef \frac{1}{2}\dis \CC+\sup_{(\sz,\sz')\in\CC}\|M(\sz)-M'(\sz')\|\\
& +\sup_{n\in\N}\,\,2^{n\alpha}\sup_{(\sz,\sz'),(\sw,\sw')\in\CC}
\|\psi_n(d(\sz,\sw))\,\delta_{\sz,\sw}M-\psi_n(d'(\sz',\sw'))\,\delta_{\sz',\sw'}M'\|
\end{split}
\end{equation}
where for every $n$, $\psi_n(x)\eqdef \psi(2^{n}x)$ and $\psi$ is a smooth function bounded above by $1$, 
which is $1$ on $[1,2]$ and $0$ outside $[2^{-1},4]$
(so that in particular its first derivative $\partial\psi_n$ is uniformly bounded, modulo 
absolute constants, by $2^{n}$). }
%
%
%
%
We can now define 
\begin{equation}\label{e:MetricCompact}
\Delta^\com_\Sp(\zeta,\zeta')\eqdef \uDelta^\com_\Sp(\zeta,\zeta')+d_\M(b_\zeta,b_{\zeta'})
\end{equation}
where $d_\M$ is the metric on the space of c\`adl\`ag functions given in~\eqref{def:M1metric} and
\begin{equation}\label{e:MetricC}
\uDelta^\com_\Sp(\zeta,\zeta')\eqdef\inf_{\CC\,:\,(\ast,\ast')\in\CC} \uDelta^{\com,\CC}_\Sp(\zeta,\zeta')\;.
\end{equation}
In view of Lemma~\ref{l:CadlagPropMap}, the metric above is well-defined.

{Before proceeding let us comment on the summands at the right hand side of~\eqref{e:MetC}. 
As we pointed out above, once we take the infimum over all the correspondences, the first term 
gives us the Gromov-Hausdorff distance between $(\ST,\ast,d)$ and $(\ST',\ast',d')$ by~\cite[Theorem 4.11]{E08}, 
which is a natural way to compare different metric spaces. 
The second term just measures the sup-norm distance between the maps $M$ and $M'$, while the third is 
a generalisation of the usual $\alpha$-H\"older metric. To make a comparison, if $\ST=\ST'$, 
the latter can be easily seen to be equivalent to the more familiar
$\sup_{\sz,\sw\in\ST} d(\sz,\sw)^{-\alpha}\|\delta_{\sz,\sw}M-\,\delta_{\sz,\sw}M'\|$. Now, in the present setting, 
we need to be able to measure the H\"older distance between functions which are not defined on the 
same space. The natural way to go beyond the same metric space case is to use the only 
way we have to ``connect'' $\ST$ and $\ST'$, i.e. the correspondence $\CC$. 
Hence, we replace the supremum over $\sz,\sw\in\ST$ 
with that over couples $(\sz,\sz'),(\sw,\sw')$ in $\CC$. On the other hand, we need to make 
sure we are comparing the increments of $M$ and $M'$ over points whose distance has the same order. 
The functions $\psi_m$'s play exactly this role - they guarantee not only that this is the case 
but further that $d(\sz,\sw),d'(\sz',\sw')\sim 2^{-m}$, 
so that we can substitute the $d(\sz,\sw)^{-\alpha}$ appearing in the classical definition, 
with $2^{m\alpha}\sim d(\sz,\sw)^{-\alpha},d'(\sz',\sw')^{-\alpha}$. 

}


\begin{proposition}\label{p:MetricC}
For $\alpha\in(0,1)$, let $\Mm^\alpha_\com$ (resp. $\T^\alpha_\com$) be the subset of 
$\Mm^\alpha_\Sp$ (resp. $\T^\alpha_\Sp$) consisting of compact metric spaces (resp. $\R$-trees). 
Then, $(\Mm_\com^\alpha,\Delta^\com_\Sp)$ is a complete separable metric space 
{and $\T^\alpha_\com$ is closed in $\Mm^\alpha_\com$.}
\end{proposition}
\begin{proof}
{We begin by verifying that $\Delta^\com_\Sp$ is a metric. 
By definition, it is clearly non-negative and symmetric. 
Concerning the triangle inequality, it holds for the second summand in~\eqref{e:MetricCompact}  
by~\cite[Theorem 12.3.1]{Whitt}, while for the first we argue as follows. 
Let $\zeta_1,\,\zeta_2,\,\zeta_3\in\Mm_\com^\alpha$, and 
$\CC_{1,2}$, $\CC_{2,3}$ be two correspondences between $\ST_1$, $\ST_2$ and $\ST_2$, $\ST_3$ respectively.
Then, upon choosing 
\begin{equ}
\CC_{1,3}\eqdef\{(\sz_1,\sz_3)\in\ST_1\times\ST_3\,:\,\exists\text{ $\sz_2\in\ST_2$ such that }
(\sz_1,\sz_2)\in\CC_{1,2},\,(\sz_2,\sz_3)\in\CC_{2,3}\}
\end{equ}
it is immediate to see that $\uDelta^{\com,\CC_{1,3}}_\Sp(\zeta_1,\zeta_3)\leq \uDelta^{\com,\CC_{1,2}}_\Sp(\zeta_1,\zeta_2)+\uDelta^{\com,\CC_{2,3}}_\Sp(\zeta_2,\zeta_3)$, which easily implies the triangle inequality for $\uDelta^\com_\Sp$. 

In order to show that the metric is positive definite, notice first that 
$\Delta^\com_\Sp(\zeta,\zeta')=0$ implies $\uDelta^\com_\Sp(\zeta,\zeta')=0=d_\M(b_\zeta,b_{\zeta'})$. 
Invoking~\cite[Theorem 12.3.1]{Whitt} once more, we immediately have that $b_\zeta\equiv b_{\zeta'}$. 
Hence, we are left to show that there exists a bijective isometry such that $\phi(\zeta)=\zeta'$, 
for which we argue similarly to~\cite[Lemma 2.1]{CHK}. Since $\uDelta^\com_\Sp(\zeta,\zeta')=0$, 
for every $\eps>0$ there exists a correspondence $\CC^\eps$ such that 
$\uDelta^{\com,\CC^\eps}_\Sp(\zeta,\zeta')<\eps$. Let $T$ be a countable dense 
subset of $\ST$ and let $\phi^\eps\colon T \to \ST'$ be such that $(\sz,\phi^\eps(\sz))\in\CC^\eps$. 
By construction, 
\begin{equ}[e:IsoCond]
|d(\sz_i, \sz_j)-d'(\phi^\eps(\sz_i), \phi^\eps(\sz_j))|<\eps\,,\qquad\text{and}\qquad \|M(\sz)-M'(\phi^\eps(\sz_i))\|<\eps\,.
\end{equ}
Since $\ST'$ is compact, we can find a subsequence in $\eps$ such that for all $\sz\in T$, 
$\phi^\eps(\sz)$ converges to some element $\phi(\sz)\in\ST'$. By~\eqref{e:IsoCond}, we immediately 
deduce that $\phi$ is a distance-preserving map on $T$ such that, for all $\sz\in T$, $M(\sz)=M'(\phi(\sz))$. 
Further, by reversing the roles of $\zeta$ and $\zeta'$ we can find a distance-preserving map $\psi$ from $\ST'$ 
to $\ST$. Since $\phi\circ\psi$ is an isometry from $\ST'$ to itself and $\ST'$ is compact, $\phi\circ\psi$ must be bijective 
(see~\cite[Theorem 1.6.14]{BBI}), which then implies that $\phi$ is itself bijective and satisfies $\phi(\zeta)=\zeta'$.

We now show completeness.
Let $\{\zeta_n\}_n$ be a Cauchy sequence in $(\Mm^\alpha_\com, \Delta^\com_\Sp)$. 
As an immediate consequence, the sequence $\{(\ST_n,\ast_n,d_n)\}_n$ is totally bounded in the space 
of (pointed) compact metric spaces. Moreover, it is not difficult to see that
\begin{equation}\label{e:SupBound1}
\sup_n\big(\|M_n\|_\infty+\sup_m 2^{-m\alpha}\omega(M_n, 2^{-m})\big)=C < \infty
\end{equation}
In order to construct the limit, we proceed as in the proof of~\cite[Theorem 7.4.15]{BBI}. 
Notice that, since the sequence $\{(\ST_n,\ast_n,d_n)\}_n$ is totally bounded, 
by~\cite[Theorem 7.4.15]{BBI},
for any $\eps>0$ there exists $N(\eps)$ such that, for all $n$, $\ST_n$
admits a finite $\eps$-net of at most $N(\eps)$ elements. 
Now, we recursively set $N_1=N(1)$ and $N_k=N_{k-1}+N(1/k)$, 
and for each $n$ we let $S_n=\{\sz_i^n\}_i$ be a countable dense set of 
$\ST_n$ such that the first $N_k$ elements form a $1/k$-net for $\ST_n$ and $\sz_0^n\eqdef\ast_n$. 
Then, passing at most to a subsequence, for every $i,j$, the limits 
$\lim_{n\to\infty}d_n(\sz_i^n,\sz_j^n)$, $\lim_{n\to\infty}M_n(\sz_i^n)$ can be shown to exist via a diagonal argument. 
Let $\tilde \ST=\{\sz_i\}_i$ be an abstract countable set and define a semimetric $d$ and a map $\tilde M$ on it, by imposing
\begin{equation}\label{e:MetricMap}
d(\sz_i,\sz_j)\eqdef\lim_{n\to\infty}d_n(\sz_i^n,\sz_j^n)\qquad \text{and}\qquad \tilde M(\sz_i)\eqdef \lim_{n\to\infty}M_n(\sz_i^n)\;.
\end{equation}
We then set $\ST$ to be the metric space obtained by taking the completion of $\bar \ST$, $\bar \ST$ being the quotient space
on $\tilde \ST$ in which points at distance $0$ are identified. $\ST$ is a compact metric space and is the 
Gromov-Hausdorff limit of $\ST_n$'s. 
It is also easy to see that $\tilde M$ is itself 
little $\alpha$-H\"older continuous 
and we let $M$ be the unique 
little $\alpha$-H\"older extension of $\tilde M$ to the whole of $\ST$. Further, 
the sequence $\{b_{\zeta_n}\}_n$ is Cauchy 
in the space of c\`adl\`ag functions endowed with the $M1$-topology, 
and therefore converges to a c\`adl\`ag function $b$ by~\cite[Theorem 12.8.1 and Theorem 12.9.2]{Whitt}. 

It remains to show that $\{\zeta_n\}_n$ converges to 
$\zeta\eqdef(\ST,\ast,d,M)$, where $\ast\eqdef \sz_0$, and that $b=b_\zeta$. 
Let $k\in\N$, $N_{k}$ be as above and $\eps\eqdef1/k$. 
Set $S_n^\eps\eqdef \{\sz_i^n\,:\,i\leq N_k\}$, $S^\eps\eqdef \{\sz_i\,:\,i\leq N_k\}$ and notice that, as shown in the proof 
of~\cite[Theorem 7.4.15]{BBI}, $S^\eps$ is an $\eps$-net for $\ST$. Define 
$\zeta_n^\eps\eqdef (S_n^\eps,\ast_n,d_n,M_n)$ and $\zeta^\eps\eqdef (S^\eps,\ast,d,M)$. 
By the triangle inequality, we have 
\begin{equation}\label{e:Conv1}
\uDelta^\com_\Sp(\zeta,\zeta_n)\leq \uDelta^\com_\Sp(\zeta,\zeta^\eps)+\uDelta^\com_\Sp(\zeta^\eps,\zeta_n^\eps)+\uDelta^\com_\Sp(\zeta_n^\eps,\zeta_n)=:A_1+A_2+A_3\,.
\end{equation}
Thanks to Lemma~\ref{l:Approx} below, $A_1$ and $A_3$ are converging to $0$ so that we only need to control $A_2$. 
For the latter, let $\CC_n\eqdef\{(\sz_i,\sz_i^n)\,:\,i\leq N_k\}$. 
Then,~\eqref{e:SupBound1} and~\eqref{e:MetricMap} ensure that
the assumptions of Lemma~\ref{l:Holder} are satisfied, so that also $A_2$ converges to $0$. 
At last, since $\uDelta^\com_\Sp(\zeta_n,\zeta)$ converges to $0$, 
Lemma~\ref{l:ContProp} immediately implies that $b=b_\zeta$.}

For separability, note that according to Lemmas~\ref{l:Approx} and~\ref{l:ContProp} below, 
any element $\zeta=(\ST,\ast,d,M)\in\Mm^\alpha_\com$ 
can be approximated in $\Mm^\alpha_\com$ by $\zeta^\eps=(\ST^\eps,\ast,d,M)$, where $\ST^\eps\subset\ST$ is a finite 
$\eps$-net of $\ST$. 
{Hence a countable dense set of $\Mm^\alpha_\com$ can be obtained 
by considering the set of metric spaces with finitely many points whose respective distances are rationals, 
endowed with maps $M$
which are $\Qrat^2$-valued.}

{To complete the proof of the statement, note that, as an immediate consequence of~\cite[Lemma 4.22]{E08}, 
$\T^\alpha_\com$ is a closed subset of $\Mm^\alpha_\com$. }
\end{proof}

\begin{remark}\label{rem:Holder}
As pointed out in~\cite[Remark 3.2]{BCK}, without the H\"older condition in the definition of $\Delta^\com_\Sp$, the space of
spatial metric spaces ($\R$-trees in their setting) with the metric comprising only the first two summands 
at the right hand side of~\eqref{e:MetC} 
is not complete while, if we did not assume the function $M$ to be {\it little} H\"older 
continuous it would lack separability. { Indeed, note that the space of $\alpha$-H\"older continuous functions 
on, say, $\R^d$ endowed with the usual H\"older metric is {\it not} separable. On the other hand, 
the subset of {\it little} $\alpha$-H\"older continuous functions (corresponding to the assumption made 
in Definition~\ref{def:SpTree}) is the closure of smooth functions with respect to the 
usual H\"older metric and thus in particular it is separable. }
\end{remark}

\begin{lemma}\label{l:Approx}
Let $\alpha\in(0,1)$, $\zeta=(\ST,\ast,d,M)\in\Mm_\com^{\alpha}$. 
Let $\delta>0$, $T \subset \ST$ be such that $\ast \in T$ 
and the Hausdorff distance between $T$ and $\ST$ is bounded above by $\delta\in(0,1)$ 
and define $\bar\zeta=(T,\ast,d,M \restr T)$. Then 
\begin{equ}[e:Approx]
\uDelta_\Sp^\com(\zeta,\bar\zeta)\lesssim \delta^{-\alpha/2}\omega(M,\sqrt{\delta})
\end{equ}
\end{lemma}
\begin{proof}
{The proof is provided in Appendix~\ref{a:Proofs}. }
%
%
\end{proof}

{ Let us now turn to the non-compact case. 
As we will only ever work with $\R$-trees, from 
now on, apart when explicitly stated, we will formulate and prove the results directly for $\T_\Sp^\alpha$. 
This in particular will allow us to use a number of statements from the literature which have been proved only 
for length spaces. 

We begin by introducing a suitable metric on $\T_\Sp^\alpha$.}
For $\zeta=(\ST,\ast,d,M)\in\T_\Sp^\alpha$ and any $r>0$, let 
\begin{equation}\label{e:rRestr}
\zeta^{(r)}\eqdef (\ST^{(r)}, \ast, d, M^{(r)})
\end{equation} 
where $\ST^{(r)}\eqdef B_d(\ast,r]$ is the closed ball of radius $r$ in $\ST$ 
and $M^{(r)}$ is the restriction of $M$ to $\ST^{(r)}$. We define $\Delta_\Sp$ 
as the function on $\T_\Sp^\alpha\times\T_\Sp^\alpha$ given by 
\begin{equation}\label{e:Metric}
\begin{split}
\Delta_\Sp(\zeta,\zeta')&\eqdef\int_0^{+\infty} e^{-r}\, \Big[1\wedge \uDelta^{\com}_\Sp(\zeta^{(r)},\zeta'^{\,(r)})\Big]\,\dd r+ d_\M(b_\zeta,b_{\zeta'})\\
&=:\uDelta_\Sp(\zeta,\zeta')+d_\M(b_\zeta,b_{\zeta'}).
\end{split}
\end{equation} 
for all $\zeta,\,\zeta'\in\T_\Sp^\alpha$. 

\begin{remark}\label{rem:properness}
The presence of the term $d_\M(b_\zeta,b_{\zeta'})$ in particular rules out the following 
type of example. Consider the $\R$-tree given by one infinite branch $e$ embedded into
 $\R^2$ as $[0,\infty)\times \{0\}$, as well as branches $e_n$ for $n \ge 1$ that are embedded as
 $[0,n]\times \{0\}$ and merge with $e$ at $(n,0)$. This tree lies in the completion of $\T_\Sp^\alpha$
under $\uDelta_\Sp$, but does not lie in $\T_\Sp^\alpha$.
 
In general, properness guarantees that we cannot have a tree $\zeta = (\ST, \ast, d, M) \in \T_\Sp^\alpha$
 admitting a sequence of points $\sz_n \in \ST$ such that $|M(\sz_n)| \le 1$ and $d(\sz_n,\sz_m) \ge 1$
 for all $n,m$. Indeed, since bounded subsets of $\ST$ are precompact (by the definition of local compactness),
 having a sequence such that $d(\sz_n,\sz_m) \ge 1$ for all 
 $(m,n)$ guarantees that $\lim_{n\to \infty} d(\sz_n,\ast) = +\infty$, so $\lim_{n\to \infty} |M(\sz_n)| = \infty$
 by properness.
\end{remark}


\begin{theorem}\label{thm:Metric}
For any $\alpha\in(0,1)$,
\begin{enumerate}[noitemsep, label=(\roman*)]
\item $\Delta_\Sp$ {in~\eqref{e:Metric} is well-defined on $\T^\alpha_\Sp\times\T^\alpha_\Sp$ and} 
is a metric,
\item the space $(\T^\alpha_\Sp,\Delta_\Sp)$ is separable and complete. 
\end{enumerate} 
\end{theorem}

We will first show point (i) and separability, then state and prove two 
lemmas, one concerning the properness map while the other the relation between $\Delta^\com_\Sp$ and $\Delta_\Sp$, 
and a characterisation of the compact subsets of $\T^\alpha_\Sp$. At last, we will see how to exploit them in order 
to show completeness. 

\begin{proof}[of Theorem~\ref{thm:Metric}(i)]
{ We begin by proving that $\Delta_\Sp$ is well-defined on $\T^\alpha_\Sp\times\T^\alpha_\Sp$. 
Since $\ST^{(r)}$ and $\ST'^{\,(r)}$ 
are compact as both $\ST$ and $\ST'$ are locally compact by assumption, 
$\zeta^{(r)},\,\zeta'^{(r)}\in\T^\alpha_\com$ so that $\Delta^\com_\Sp(\zeta^{(r)},\zeta'^{(r)})$ makes sense. 
To see that the first summand in~\eqref{e:Metric} is well-defined, we note that the map 
$r\mapsto \uDelta^{\com}_\Sp(\zeta^{(r)},\zeta'^{\,(r)})$ is c\`adl\`ag. 
Indeed, by the triangle inequality we have 
\begin{equs}
\uDelta^{\com}_\Sp(\zeta^{(r+\eps)},\zeta'^{\,(r+\eps)})\leq \uDelta^{\com}_\Sp(\zeta^{(r+\eps)},\zeta^{\,(r)})+\uDelta^{\com}_\Sp(\zeta^{(r)},\zeta'^{\,(r)})+\uDelta^{\com}_\Sp(\zeta'^{\,(r)},\zeta'^{\,(r+\eps)})
\end{equs}
which, by switching the roles of $r$ and $r+\eps$, immediately implies that 
\begin{equ}
|\uDelta^{\com}_\Sp(\zeta^{(r+\eps)},\zeta'^{\,(r+\eps)})-\uDelta^{\com}_\Sp(\zeta^{(r)},\zeta'^{\,(r)})|\leq \uDelta^{\com}_\Sp(\zeta^{(r+\eps)},\zeta^{\,(r)})+\uDelta^{\com}_\Sp(\zeta'^{\,(r)},\zeta'^{\,(r+\eps)})\,.
\end{equ}
Since $\ST^{(r)}$ and $\ST'^{\,(r)}$ are closed by definition, it is not difficult to see 
that the Hausdorff distance between $\ST^{(r+\eps)}$ and $\ST^{(r)}$ 
as well as that between $\ST'^{\,(r)}$ and $\ST'^{\,(r+\eps)}$ 
is going to $0$ as $\eps\to0$. Hence, Lemma~\ref{l:Approx} implies that the right hand side is converging to $0$. 
For the existence of left limits instead, 
let $\tilde \zeta^{(r)}\eqdef (\tilde\ST^{(r)}, \ast,d,\tilde M^{(r)})$ 
be such that $\tilde\ST^{(r)}$ is the closure $\cup_{\eps'>0}\ST^{(r-\eps')}$ and $\tilde M^{(r)}$ 
is the restriction of $M$ to $\tilde\ST^{(r)}$, and define 
$\tilde \zeta'^{\,(r)}\eqdef (\tilde\ST'^{\,(r)}, \ast',d',\tilde M'^{\,(r)})$ similarly. 
Then, replacing $\zeta^{(r)}$ and $\zeta'^{\,(r)}$ 
with $\tilde\zeta^{(r)}$ and $\tilde\zeta'^{\,(r)}$ in the argument 
for the right continuity, we can show that $\uDelta^{\com}_\Sp(\zeta^{(r-\eps)},\zeta'^{\,(r-\eps)})$ 
converges to $\uDelta^{\com}_\Sp(\tilde\zeta^{(r)},\tilde\zeta'^{\,(r)})$. 

We now prove that $\Delta_\Sp$ is indeed a metric.} 
As in the proof of Proposition~\ref{p:MetricC}, we only need to focus on the first summand in~\eqref{e:Metric} 
and show it satisfies the axioms of a metric. 
Positivity and symmetry clearly hold, while the triangle inequality follows by the fact that 
it holds for $\uDelta^\com_\Sp$. {For positive definiteness, we argue as in~\cite[Proposition 5.3]{ADH}. 
Assume $\uDelta^\com_\Sp(\zeta,\zeta')=0$. 
Then, since $r\mapsto \uDelta^{\com}_\Sp(\zeta^{(r)},\zeta'^{\,(r)})$ is c\`adl\`ag, 
$\uDelta^{\com}_\Sp(\zeta^{(r)},\zeta'^{\,(r)})=0$ for every $r\geq 0$ and consequently 
there exists an isometry $\phi_r$ from $\ST^{(r)}$ to $\ST'^{\,(r)}$ for which $\phi_r(\zeta)=\zeta'$. 
For every $n,k$ positive integers, let $\{\sz_i^{n,k}\}_{i\leq N_{n,k}}$ be a $1/k$-net of $\ST^{(n)}$ 
with $\sz_0^{n,k}\eqdef \ast$ and $N_{n,k}<\infty$ be its cardinality. 
Notice that, since for every $m\geq n$, $\phi_m$ is distance preserving, the sequence 
$\{\phi_m(\sz_i^{n,k})\}_{m\geq n}$ is bounded for every $n,k$ and $i\leq N_{n,k}$. 
Via a diagonal argument, it is then possible to find a subsequence such that 
$\phi(\sz_i^{n,k})\eqdef\lim_{m\to\infty} \phi_m(\sz_i^{n,k})$ exists for every $n,k$ and $i\leq N_{n,k}$.
Clearly, $\phi$ is distance preserving on $\{\sz^{n,k}_i\}_{n,k,i}$ and, for any $n,k$ given 
$\phi(\{\sz^{n,k}_i\}_{n,k,i})$ is a $2/k$-net for $\ST'^{\,(n)}$. Hence, $\phi(\{\sz^{n,k}_i\}_{n,k,i})$ 
is dense in $\ST'$ and since $\{\sz^{n,k}_i\}_{n,k,i}$ is dense in $\ST$, 
$\phi$ can be uniquely extended to a bijective isometry on $\ST$. At last, 
the continuity of $M$ and $M'$ imply $M'\circ\phi\equiv M$ which ensures that $\phi(\zeta)=\zeta'$. 
}

To show separability, given $\zeta\in\T^\alpha_\Sp$ and $r>0$, let $R\eqdef \diam(M(\ST^{(r)}))$. 
Then, the definition of the metric implies
$\Delta_\Sp(\zeta, \zeta^{(r)})\lesssim e^{-r}\vee e^{-R}$, so that any element of $\T^\alpha_\Sp$ can be approximated 
arbitrarily well by elements in $\T^\alpha_\com$. Since, in view of Proposition~\ref{p:MetricC}, 
the latter space is separable, 
and thanks to Lemma~\ref{l:CompactvsNonCompact} convergence in 
$\Delta^\com_\Sp$ implies convergence in $\Delta_\Sp$, 
separability on $\T^\alpha_\Sp$ follows.
\end{proof}

\begin{lemma}\label{l:ContProp}
Let $\alpha\in(0,1)$, $\{\zeta_n=(\ST_n,\ast_n,d_n,M_n)\}_{n\in\N}$  
{be in $\T^\alpha_\Sp$ (resp. $\Mm^\alpha_\com$)  and let $\zeta=(\ST,\ast,d,M)$ be such that 
$\uDelta_\Sp(\zeta_n,\zeta)$ (resp. $\uDelta_\com(\zeta_n,\zeta)$), converge to $0$ as $n\to\infty$. 
Assume further that} for every $r>0$
there exists a finite constant $C'=C'(r)>0$ such that
\begin{equation}\label{e:Comb}
b_{\zeta_n}(r)\leq C'\,,
\end{equation}
uniformly over $n\in\N$. Then, { $M$ is proper, so that in particular} $\zeta\in\T^\alpha_\Sp$ 
{(resp. $\zeta\in\Mm^\alpha_\com$)}, 
and $d_\M(b_{\zeta_n},b_\zeta)$ converges to $0$.
\end{lemma}
%
\begin{proof}
{ We only show the statement for $\{\zeta_n=(\ST_n,\ast_n,d_n,M_n)\}_{n\in\N}\subset \T^\alpha_\Sp$, 
as the proof does not rely on the $\R$-tree structure of the metric spaces and the case of 
$\Mm^\alpha_\com$ is the same but simpler. }

We begin by proving that $M$ is proper {(which in the compact case is obvious)}. 
Let {$r>0$ be fixed and} $\sz\in\ST$ be such that 
$M(\sz)\in\Lambda_r$. {Then, by~\eqref{e:Comb} and since $\uDelta_\Sp(\zeta_n,\zeta)\to 0$, 
there exists $R>0$
such that $\sz\in B_d(\ast,R]$ and for every $n$ we have both that
$M_n^{-1}(\Lambda_{r+1})\subset B_{d_n}(\ast_n,R]$ and that there exists 
a correspondence $\CC_n^R$ between $\ST^{(R)}$ and $\ST^{(R)}_n$ for which
$\eps_n\eqdef\uDelta^{\com,\CC_n^R}_\Sp(\zeta^{(R)}_n,\zeta^{(R)})\to 0$}.  
Without loss of generality, we can take $R>C'(r+1)+2$, so that, in view of~\eqref{e:Comb}, 
for every $n$, {$M_n^{-1}(\Lambda_{r+1})\subset B_{d_n}(\ast_n,R]$}. 
Now, let $\CC_n^R$ be a correspondence between $\ST^{(R)}$ and $\ST^{(R)}_n$ such that 
$\eps_n\eqdef\uDelta^{\com,\CC_n^R}_\Sp(\zeta^{(R)}_n,\zeta^{(R)})\to 0$. 
Let $\sz_n\in \ST^{(R)}_n$ be such that $(\sz,\sz_n)\in\CC^R_n$. 
Then, $|M_n(\sz_n)|\leq r+\eps_n$ so that, thanks to~\eqref{e:Comb},
\begin{equ}
d(\sz,\ast)\leq b_{\zeta_n}(r+\eps_n) +2\eps_n\leq C'(r+\eps_n)+2\eps_n\,,
\end{equ}
which implies that $M$ is proper. { Hence, $\zeta\in\T^\alpha_\Sp$ and, by Lemma~\ref{l:CadlagPropMap}, 
$b_\zeta$ is c\`adl\`ag. }

It remains to prove that $b_{\zeta_n}$ converges to $b_\zeta$.~\cite[Theorem 12.9.3 and Corollary 12.5.1]{Whitt} 
ensure that it suffices to show that $b_{\zeta_n}(r)\to b_\zeta(r)$ 
for every $r$ at which $b_\zeta$ is continuous. 
Let $r\in\Disc(b_\zeta)^c$, $R>b_\zeta(r)\vee C'(r)$ and $\CC_n^R$ and $\eps_n$ be as above. 
Notice that 
\begin{equs}
|b_\zeta(r)-b_{\zeta_n}(r)|&=\Big|b_\zeta(r)-\sup_{\substack{\sz_n\,:M_n(\sz_n)\in\Lambda_r\\(\sz,\sz_n)\in\CC^R_n}}d_n(\ast_n,\sz_n)\Big|\leq \Big|b_\zeta(r)-\sup_{\substack{\sz_n\,:M_n(\sz_n)\in\Lambda_r\\(\sz,\sz_n)\in\CC^R_n}}d(\ast,\sz)\Big|+\eps_n\,.
\end{equs}
Now, for $(\sz,\sz_n)\in\CC_n$, if $M(\sz)\in\Lambda_{r-\eps_n}$ then $M_n(\sz_n)\in\Lambda_r$, while if 
$M_n(\sz_n)\in\Lambda_r$, then $M(\sz)\in\Lambda_{r+\eps_n}$ which implies that 
\begin{equs}
b_\zeta(r-\eps_n)-b_\zeta(r)\leq \sup_{\substack{\sz_n\,:M_n(\sz_n)\in\Lambda_r\\(\sz,\sz_n)\in\CC_n}}d(\ast,\sz)-b_\zeta(r)\leq b_\zeta(r+\eps_n)-b_\zeta(r)
\end{equs}
from which the conclusion follows. 
\end{proof}

\begin{lemma}\label{l:CompactvsNonCompact}
For any $\alpha\in(0,1)$, the identity map from $(\T^\alpha_\com,\Delta_\Sp^\com)$ to $(\T^\alpha_\Sp,\Delta_\Sp)$ is continuous. 
\end{lemma}
\begin{proof}
Let $\{\zeta_n\}_n\,,\zeta\subset \T^\alpha_\com$ be such that $\Delta^\com_\Sp(\zeta_n,\zeta)$ 
converges to $0$. In particular, $d_\M(b_{\zeta_n},b_\zeta)\to 0$ as $n\to\infty$ so that
we are left to show that $\lim_n\uDelta_\Sp(\zeta_n,\zeta)=0$. 

{Let $\eps>0$ and $\bar m\in\N$ be such that $2^{\bar m\alpha}\omega(M,2^{-\bar m})\leq \eps$. 
Since $\Delta^\com_\Sp(\zeta_n,\zeta)\to 0$, for $n$ big enough, there exists a correspondence 
$\CC_n$ between $\ST$ and $\ST_n$ such that  
\begin{equ}[e:boundDisMapCont]
\sup_{(\sz,\sz_n), (\sw,\sw_n)\in\CC_n}|d(\sz,\sw)-d_n(\sz_n,\sw_n)|<\eps\,,\qquad\sup_{(\sz,\sz_n)\in\CC_n}\|M(\sz)-M(\sz_n)\|<\eps\,.
\end{equ}
Let $r\geq 0$ be fixed. Our goal is to show that there exists a correspondence $\CC_n^r$ such that 
the assumptions of Lemma~\ref{l:Holder} are satisfied. Note that clearly, since $\Delta^\com_\Sp(\zeta_n,\zeta)\to 0$ 
we have uniform bound on the modulus of continuity of $M^{(r)}_n$ and $M^{(r)}$ so that~\eqref{e:UnifModCont} 
clearly holds. 
For~\eqref{e:dis+M}, we proceed as in~\cite[Proposition 3.4]{BCK}. 
Namely, we define $\CC_n^r$ as the correspondence which contains the pair $(\sz,\sz_n)$ provided 
that either $\sz\in\ST^{(r)},\sz_n\in\ST_n^{(r)}$ and $(\sz,\sz_n)\in\CC_n$, or $\sz_n\in\ST_n^{(r)}$ 
(resp. $\sz\in\ST^{(r)}$) and $\sz$ (resp. $\sz_n$) is the closest point in $\ST^{(r)}$ (resp. $\ST_n^{(r)}$) 
to $\sz'$ (resp. $\sz_n'$) such that $(\sz',\sz_n)\in\CC_n$ (resp. $(\sz,\sz_n')\in\CC_n$). 
To be more precise, since $\ST$ and $\ST_n$ are $\R$-trees and in particular length spaces, 
we will take $\sz$ to be the point on the segment connecting $\ast$ to $\sz'$ such that $d(\ast,\sz)=r$ 
and similarly for $\sz_n$. 

Assume $(\sz,\sz_n)\in\CC^r_n\setminus\CC_n$ is such that $\sz_n\in\ST_n^{(r)}$ and $\sz'$ 
is as above. Then, by the first bound in~\eqref{e:boundDisMapCont}
\begin{equs}
r+d(\sz,\sz')=d(\ast,\sz)+d(\sz,\sz')=d(\ast,\sz')&< d(\ast_n, \sz_n)+\eps\leq r +\eps\,,
\end{equs}
which implies
\begin{equ}[e:newPoint]
d(\sz,\sz')< \eps\,.
\end{equ}
For $(\sz,\sz_n),(\sw,\sw_n)\in\CC^r_n\cap\CC_n$, thanks to~\eqref{e:boundDisMapCont} there is nothing to argue. 
If instead, say, $(\sz,\sz_n)\in\CC^r_n\setminus\CC_n$ is such that $\sz_n\in\ST_n^{(r)}$ and $\sz'$ 
is as above, by~\eqref{e:newPoint}, we obtain 
\begin{equ}
|d(\sz,\sw)-d_n(\sz_n,\sw_n)|\leq d(\sz,\sz')+|d(\sz',\sw)-d_n(\sz_n,\sw_n)|< 2\eps
\end{equ}
and we can argue similarly if also $(\sw,\sw_n)\in\CC^r_n\setminus\CC_n$. The bound on the difference 
of the evaluation maps instead reads
\begin{equs}
\|M(\sz)-M_n(\sz_n)\|&\leq \|\delta_{\sz,\sz'} M\|+\|M(\sz')-M_n(\sz_n)\|\leq\|M\|_\alpha d(\sz,\sz')^\alpha +\eps\lesssim \eps^\alpha
\end{equs}
where we used the $\alpha$-H\"older continuity of $M$. 
Collecting the previous bounds, we immediately see that~\eqref{e:boundDisMap1} holds, 
so that by Lemma~\ref{l:Holder} $\uDelta^\com_\Sp(\zeta_n^{(r)}, \zeta^{(r)})$ 
converges to $0$ and the statement follows at once. }
\end{proof}

\begin{proposition}\label{p:Compactness}
Let $\alpha\in(0,1)$ and $A$ be an index set. A subset $\CA=\{\zeta_a=(\ST_a,\ast_a,d_a,M_a)\,:\,a\in A\}$ of $\T^\alpha_\Sp$ 
is relatively compact if and only if for every $r>0$ and $\eps>0$ there exist
\begin{enumerate}[noitemsep]
\item a finite integer $N(r;\eps)$ such that uniformly over all $a\in A$, 
\begin{equ}[e:EpsNet]
\cN_{d_a}(\ST_a^{(r)},\eps)\leq N(r;\eps)
\end{equ}
where $\cN_{d_a}(\ST_a^{(r)},\eps)$ is the cardinality of the minimal $\eps$-net in $\ST_a^{(r)}$ 
with respect to the metric $d_a$, 
\item a finite constant $C=C(r)>0$ and $\delta=\delta(r,\eps)>0$ such that
\begin{equation}\label{e:Equicont}
\sup_{a\in A}\|M_a\|_{\infty}^{(r)}\leq C\qquad\text{and}\qquad\sup_{a\in A}\delta^{-\alpha}\omega^{(r)}(M_a,\delta)<\eps\,,
\end{equation}
\item a finite constant $C'=C'(r)>0$ such that~\eqref{e:Comb} holds uniformly over $a\in A$. 
\end{enumerate}
\end{proposition}

\begin{proof}
``$\Longleftarrow$'' Let $\{\zeta_n=(\ST_n,\ast_n,d_n,M_n)\}_n\subset \CA$ be a sequence satisfying the three properties above.

We want to extract a converging subsequence for $\{\zeta_n\}_n$ and construct the corresponding limit point. 
{The limit space is built as in~\cite[Sections 5.2.1 and 5.2.2]{ADH} so we briefly recall the construction 
and try to follow the notations therein as closely as possible.
For $r>0$, $\ell,\,k\in\N$ and any $n$, 
let $A_{r,k}(\ST_n)\eqdef \ST_n^{(r)}\setminus\ST_n^{(r-2^{-k})}$ 
and $\ell_k=\ell 2^{-k}$. By~\cite[Lemma 5.4]{ADH} and~\eqref{e:EpsNet}, 
there exists a $2^{-k-1}$-net of $A_{\ell_k,k}(\ST_n)$ with at most $N(\ell_k;2^{-k-2})$ elements, 
let it be $\cG^n_{\ell_k,k}$. 
Let $S^n_{\ell_k,k}$ be the union of $A_{\ell_k,k}(\ST_n)\cap \cG^n_{\ell_k,k}$ for $0\leq k'\leq k$, 
which is a $2^{-k-1}$-net for $A_{\ell_k,k}(\ST_n)$ whose cardinality is bounded above by 
$\tilde N(\ell_k;2^{-k-2})\eqdef\sum_{k'}N(\lceil\ell_k2^{-k'}\rceil2^{k'};2^{-k'-2})$. 
As in~\cite[eq. (5.3)]{ADH}, we write $S^n_{\ell_k,k}$ as
\begin{equ}
S^n_{\ell_k,k}\cup\{\ast_n\}\eqdef\{\sz^n_u\colon u=(k,\ell,i)\in U_{\ell_k,k}\}
\end{equ}
where $U_{\ell_k,k}\eqdef\{(k,\ell,i)\colon 0\leq i\leq  \tilde N(\ell_k;2^{-k-2})\}$ and 
we set $U$ to be the union of all the $U_{\ell_k,k}$ for $k,\ell\in\N$. 

Now, notice that, for every $u,u'\in U$, both the sequences $\{d^n(\sz^n_u,\sz^n_{u'})\}_n$ and $\{M_n(\sz^n_u)\}_n$ 
are bounded - the second claim following by the first condition in~\eqref{e:Equicont}. Hence, 
via a diagonal argument, upon passing to a subsequence, we can ensure that for every $u,u'\in U$ they 
converge.  
Let $\tilde \ST=\{\sz_u\}_{u\in U}$ be an abstract countable set 
and define a semimetric $d$ and a map $\tilde M$ on it, by imposing
\begin{equation}\label{e:MetricMapCom}
d(\sz_u,\sz_{u'})\eqdef\lim_{n\to\infty}d_n(\sz_u^n,\sz_{u'}^n)\qquad \text{and}\qquad \tilde M(\sz_u)\eqdef \lim_{n\to\infty}M_n(\sz_u^n)\;.
\end{equation}
We then set $\ST$ to be the metric space obtained by taking the completion of $\bar \ST$, 
$\bar \ST$ being the quotient space on $\tilde \ST$ in which points at distance $0$ are identified.
~\cite[Lemma 5.7]{ADH} ensures that $\ST$ is a length space, while we can see it is an $\R$-tree 
as the four point condition (see~\cite[Definition 3.9 and Theorem 3.40]{E08}) can be immediately 
shown to hold by~\eqref{e:MetricMapCom} and the fact it holds for each of the $\ST_n$'s. 
As in the first display in~\cite[Section 5.2.2]{ADH}, we set $U_{\ell_k,k}^+$ 
to be the union of $U_{j2^{-k},k}$ for $0\leq j\leq \ell$, and 
\begin{equ}[e:theNets]
S^{+}_{\ell_k,k}\eqdef\{\sz_u\colon u\in U_{\ell_k,k}^+\}\qquad \text{and}\qquad S^{n,+}_{\ell_k,k}\eqdef\{\sz_u^n\colon u\in U_{\ell_k,k}^+\}\,.
\end{equ}
\cite[Lemma 5.6]{ADH} ensures that, for every $\ell,k$, $S^{+}_{\ell_k,k}$ is a $2^{-k}$-net for $\ST^{(\ell_k)}$, 
which in particular implies that $\ST$ is locally compact.  
Moreover, by condition 2. and the second formula in~\eqref{e:MetricMapCom}, we have that 
$\tilde M$ is locally little $\alpha$-H\"older continuous so 
that we can set $M$ to be the 
unique locally little $\alpha$-H\"older continuous extension of $\tilde M$ to $\ST$. 

By condition 3. and Lemma~\ref{l:ContProp}, once we prove that $\uDelta_\Sp(\zeta_n,\zeta)$ converges to $0$, 
where $\zeta\eqdef(\ST,\ast,d,M)$ and $\ast\eqdef \sz_{(0,k,\ell)}$, we are done. 

To do so, let $r>0$ and $k\in\N$ be fixed and define $\ell\eqdef\lceil2^k r\rceil$ and $\eps\eqdef 2^{-k}$. 
Set $\zeta_n^{\ell,k}\eqdef (S^{n,+}_{\ell,k},\ast_n,d_n,M_n)$ and 
$\zeta^{\ell,k}\eqdef (S^+_{\ell,k},\ast,d,M)$. 
By the triangle inequality we have 
\begin{equs}[e:Conv]
\uDelta^\com_\Sp(\zeta^{(r)},\zeta^{(r)}_n)\leq& \uDelta^\com_\Sp(\zeta^{(r)},\zeta^{(\ell_k)})+\uDelta^\com_\Sp(\zeta^{(\ell_k)},\zeta^{\ell,k})+\uDelta^\com_\Sp(\zeta^{\ell,k},\zeta_n^{\ell,k})\\
&\quad+\uDelta^\com_\Sp(\zeta_n^{\ell,k},\zeta^{(\ell_k)}_n)+\uDelta^\com_\Sp(\zeta_n^{(\ell_k)},\zeta_n^{(r)})=:\sum_{i=1}^5 A_i\,.
\end{equs}
As we pointed out above $S^+_{\ell,k}$ and $S^{n,+}_{\ell,k}$ are $\eps$-nets for 
$\ST^{(\ell_k)}$ and $\ST^{(\ell_k)}_n$, respectively. Hence, by Lemma~\ref{l:Approx} and~\eqref{e:Equicont}, 
all the $A_i$'s, for $i\neq 3$, can be controlled by quantities 
which are vanishing as $k\to\infty$, so that we only need to focus on $A_3$. 
For this in turn, the second condition in~\eqref{e:Equicont} implies~\eqref{e:UnifModCont} 
while, upon choosing $\CC^r_n\eqdef\{(\sz^n_u,\sz_u)\colon u\in U_{\ell_k,k}^+\}$, 
we see that~\eqref{e:MetricMapCom} gives~\eqref{e:dis+M}. Hence, 
the assumptions of Lemma~\ref{l:Holder} hold, so that also $A_3\to0$. 

Since for every $r>0$, $\uDelta^\com_\Sp(\zeta^{(r)},\zeta^{(r)}_n)\to 0$, it follows that 
so does $\uDelta_\Sp(\zeta,\zeta_n)$ and the proof of ``$\Longleftarrow$'' is concluded.}

``$\Longrightarrow$'' Let $\cA$ be relatively compact in $\T^\alpha_\Sp$. Then, property~1.\ holds 
by{~\cite[Theorem 7.4.15]{BBI}}, while property~3.\ by~\cite[Theorem 12.12.2]{Whitt} 
{ on the necessary condition for a set to be compact in the strong $M1$ topology on the 
space of c\`adl\`ag functions}. For 
the second property, 
notice that since $\cA$ is totally bounded, 
for any $\eps>0$ and $r>0$ there exist $n \in \N$ and
$\{\zeta_k\,:\,k=1,\dots n\}$ such that $\cA$ is contained in the union of the balls of radius 
$e^{-r}\eps/4$ centred at $\zeta_k$. 
Hence, if $\zeta\in B(\zeta_k,e^{-r}\eps/4)$, then we have
\begin{equation}\label{e:BoundEps}
\uDelta^\com_\Sp(\zeta^{(r)},\zeta_k^{(r)}) < \tfrac{\eps}{4}
\end{equation} 
which implies that there exists a correspondence $\CC$ between $\ST^{(r)}$ and $\ST_k^{(r)}$ such that 
$\Delta^{\com,\CC}_{\Sp}(\zeta^{(r)},\zeta_k^{(r)})<\eps/4$. 
Since $\|M_\zeta\|^{(r)}_\infty\leq \eps/2 +\| M_{\zeta_k}\|^{(r)}_\infty$ by the triangle inequality, 
\begin{equation}\label{e:SupBound}
\sup_{\zeta\in\cA}\|M_\zeta\|^{(r)}_\infty\leq \tfrac{\eps}{4}+ \max_{k=1,\dots,n}\| M_{\zeta_k}\|^{(r)}_\infty\,,
\end{equation}
and the first bound in~\eqref{e:Equicont} follows. 
For the others, let $\delta>0$ and $\bar n\in\N$ the largest integer such that $2^{-\bar n}\leq\delta$. Then, 
\begin{equation}\label{e:HolderBound}
\begin{split}
&\sup_{n>\bar n}\,\,2^{n\alpha}\sup_{(\sz,\sz_k),(\sw,\sw_k)\in\CC}\|\psi_n(d(\sz,\sw))\delta_{\sz,\sw}M_\zeta-\psi_n(d_k(\sz_k,\sw_k))\delta_{\sz_k,\sw_k}M_{\zeta_k}\|\\
&\leq \sup_{n\in\N}\,\,2^{n\alpha}\sup_{(\sz,\sz_k),(\sw,\sw_k)\in\CC}\|\psi_n(d(\sz,\sw))\delta_{\sz,\sw}M_\zeta-\psi_n(d_k(\sz_k,\sw_k))\delta_{\sz_k,\sw_k}M_{\zeta_k}\|<\tfrac{\eps}{4}
\end{split}
\end{equation}
so that, once again, the second bound in~\eqref{e:Equicont} can be obtained by applying the triangle inequality 
and choosing the minimum $\delta$ for which 
$\sup_{k\leq n}\delta^{-\alpha}\omega^{(r)}(M_{\zeta_k},\delta)<\eps/2$\,. 
\end{proof}

\begin{proof}[of Theorem~\ref{thm:Metric}(ii)]
To prove completeness, it suffices to show that, if $\{\zeta_n\}_n$ is a Cauchy sequence in $\T^\alpha_\Sp$ 
then the conditions of Proposition~\ref{p:Compactness} are satisfied. 
Now, if $\{\zeta_n\}_n$ is Cauchy, then for every $r>0$, $\{\zeta_n^{(r)}\}_n$ is Cauchy with respect to 
$\Delta_\Sp^\com$, which implies that the sequence converges so that~1.\ holds in view of~\cite[Proposition 7.4.12]{BBI},~2.\ can be seen to be satisfied by arguing as in~\eqref{e:SupBound} and~\eqref{e:HolderBound}, and~3.\ follows by the fact that $D([-1,\infty),\R_+)$ is complete with respect to $d_\M$. 
\end{proof}

We conclude this section with a lemma that will be useful in the construction and characterisation of the Brownian Web. 
It guarantees that, 
under certain conditions, we can build an $\alpha$-spatial $\R$-tree inductively, by ``patching together'' pieces of branches.

\begin{lemma}\label{l:Emb}
Let $\alpha\in(0,1)$ and $\zeta_n=(\ST_n,\ast_n,d_n,M_n)$ be a relatively compact sequence in $\T^\alpha_\Sp$. 
Assume that for every $n < m\in\N$ there exists an isometric embedding 
$\iota_{n,m}$ of $\ST_n$ into $\ST_{m}$ such that $\iota_{n,m}(\ast_n)=\ast_{m}$, $\iota_{n,k} = \iota_{m,k} \circ \iota_{n,m}$ 
for $n<k<m$ and $M_{m}\circ\iota_{n,m}\equiv M_n$. 
Then, the sequence $\zeta_n$ converges to $\zeta=(\ST,\ast,d,M)$ and 
for every $n\in\N$ there exists an isometric embedding $\iota_n$ of $\ST_n$ into $\ST$ such that $\iota_{n}(\ast_n)=\ast$,
$\iota_n = \iota_{m}\circ \iota_{n,m}$ for $m>n$ and $M\circ\iota_n\equiv M_n$.
Moreover, $\tilde\ST\eqdef\bigcup_n\iota_n(\ST_n)$ is dense in $\ST$ and 
$M$ is the unique continuous extension of $\tilde M$ on $\tilde\ST$, 
the latter being defined by the relation $\tilde M\circ\iota_n\equiv M_n$ for all $n$. 
\end{lemma}
\begin{remark}
A similar statement was given in~\cite[Lemma 2.7]{EPW}. The formulation is a bit
different since we do not have a common ambient space and the trees we consider are 
spatial. One reason why we cannot directly reuse that result is that it is not clear a priori
that relative compactness in $\T^\alpha_\Sp$ implies relative compactness of the 
images in $\bigcup_n \ST_n / {\sim}$ with the natural equivalence relation induced
by the consistency maps $\iota_{m,n}$ {(and part of our proof consists of showing that this is indeed the case)}. 
This is because the optimal correspondence between
$\ST_n$ and $\ST_m$ may differ from the one given by $\iota_{m,n}$. Take for example
the trees $(\ST,\ast) = ([0,1],1/3)$ and $(\bar \ST, \bar\ast) = ([0,1/3], 1/3)$. Then, for the
natural correspondence $\CC$ suggested by our notation, one has $\dis\CC = 2/3$,
while the correspondence $\bar \CC$ mapping $x \in \bar \ST$ to $2/3-x \in \ST$
is also an isometric embedding but has $\dis \bar \CC = 1/3$. This shows that the condition in
\cite[Lemma 2.7]{EPW} assuming that the $\zeta_n$ are Cauchy as subsets of a common
space in the Hausdorff topology may a priori be stronger than the relative compactness
assumed here. (A posteriori it is not, as demonstrated by the fact that $\tilde \ST$ is
dense in $\ST$.)
\end{remark}

\begin{proof}
We will limit ourselves to the case of $\ST_n$ compact, the general case easily follows from 
the definition of the metric $\Delta_\Sp$. 

{We begin by constructing the limit space. 
Let $\tilde\ST = \big(\bigsqcup_n \ST_n\big)/{\sim}$, where $\sim$
is the smallest equivalence relation such that $\sz \sim \iota_{n,m}(\sz)$ for
every $n \le m$ and every $\sz\in \ST_n$,
and let $\ast\in\tilde \ST$ be the equivalence class containing all the $\ast_n$. 
We define $d$ on $\tilde\ST\times\tilde\ST$ by setting, for $\sz\in\ST_n$, $\sw\in\ST_m$ with $n\le m$, 
$d(\sz,\sw)\eqdef d_m(\iota_{n,m}(\sz),\sw)$, which is clearly a metric on $\tilde\ST$. 
For $n\in\N$, let $\iota_n\colon \ST_n\to\tilde\ST$ be the canonical embedding, which 
can be easily seen to be an isometry such $\iota_n = \iota_{m}\circ \iota_{n,m}$ for all $n<m$, 
and $\tilde\ST=\bigcup_n\iota_n(\ST_n)$. 
At last, let $\tilde M\colon \tilde \ST\to \R^2$ be the map defined as $\tilde M(\sz)\eqdef M_n(\iota_n^{-1}(\sz))$ 
for $\sz\in\iota_n(\ST_n)$, so that $\tilde M\circ\iota_n=M_n$. 

We will first show that $(\tilde\ST,\ast,d)$ is totally bounded and $\tilde M$ is little $\alpha$-H\"older continuous. 
For the first, recall that, since the sequence $\zeta_n$ is relatively compact in $\T^\alpha_\com$, 
by point 1. of Proposition~\ref{p:Compactness}, for every $\eps>0$ there exists $N(\eps)\in\N$ 
such that 
\begin{equ}[e:Compact]
\sup_n \cN_{d_n}(\ST_n,\eps)\leq N(\eps)\,.
\end{equ}
We now make the following claim.
\begin{enumerate}[label=(\Alph*)]
\setcounter{enumi}{2}
\item\label{i:Claim} for every $\eps>0$ there exists $n_\eps\in\N$ such that 
for all $n> n_\eps$ and $\sz\in\iota_n(\ST_n)$ there is $\sw\in\iota_{n_\eps}(\ST_{n_\eps})$ for which $d(\sz,\sw)<\eps$.
\end{enumerate}
To prove~\ref{i:Claim}, assume by contradiction that there exists $\bar\eps>0$ such that the claim fails, so that there
exists a sequence $\sz_n$ with $\sz_n \in\iota_n(\ST_n)$ and such that 
$d(\sz_i,\sz_j)>\bar\eps$ for all $i\neq j$. 
Setting $k=N(\bar\eps)+1$, this yields a $\bar\eps$-separated set for 
$\ST_{k}$ whose cardinality is greater than $N(\bar\eps)$, thus yielding a contradiction.

We now show that $\tilde\ST$ is relatively compact, i.e.\ that for every $\eps>0$ it has a finite $\eps$-net. 
Let $\eps>0$ be fixed, $n_{\eps}$ be as in~\ref{i:Claim}, and 
$S_{\eps}$ be an $\eps$-net for $\ST_{n_{\eps}}$, 
which can be chosen finite since $\ST_{n_{\eps}}$ is compact. 
It then follows from~\ref{i:Claim} that $\iota_{n_\eps}(S_{\eps})$ 
is a $2\eps$-net for $\tilde\ST$. 

To show the little H\"older continuity of $\tilde M$, let $\sz,\sw\in\tilde\ST$ be such that $d(\sz,\sw)\leq \delta$. 
Then, there exist $m,n$, with $n\leq m$, 
such that $\sz\in\iota_n(\ST_n)\subset \iota_m(\ST_m)$ and $\sw\in\iota_m(\ST_m)$, so that, in particular, 
$d_m(\iota_m^{-1}(\sz),\iota_m^{-1}(\sw))=d(\sz,\sw)\leq \delta$. Hence, 
\begin{equs}[e:EmbHolder]
\delta^{-\alpha}\|\tilde M(\sz)-\tilde M(\sw)\|&=\delta^{-\alpha}\|M_m(\iota_m^{-1}(\sz))-M_m(\iota_m^{-1}(\sw))\|\\
&\leq \delta^{-\alpha}\omega(M_m, \delta)\leq \delta^{-\alpha}\sup_m\omega(M_m, \delta)
\end{equs}
and, since $\{\zeta_n\}_n$ is relatively compact by assumption, 
the right hand side converges to $0$ as $\delta\to 0$ in view of point 2 of Proposition~\ref{p:Compactness}.

We are now ready to construct the limit point $\zeta$. 
Let $\ST$ be the completion of $\tilde \ST$ with respect to the metric $d$, $M$ the unique 
little H\"older continuous extension of $\tilde M$ to $\ST$, and set 
$\zeta\eqdef(\ST,\ast,d,M)$, which, by the above discussion, belongs to $\T^\com_\Sp$. 
It remains to show that the sequence $\{\zeta_n\}_n$ converges to $\zeta$. 
We apply Lemma~\ref{l:Holder}. Condition~\eqref{e:UnifModCont} is implied 
by~\eqref{e:EmbHolder}, so that we only need to find a correspondence for which~\eqref{e:dis+M} holds. 

For $n > 0$, we then set $\eps_n = \inf\{\eps>0\,:\, n_\eps < n\}$, with $n_\eps$ as
in~\ref{i:Claim}, and we define $\CC_n = \{(\sz,\sz')\in\ST_{n}\times \ST\,:\,
d(\iota_n(\sz),\sz') \le \eps_n\}$. This is a correspondence by the definition of $\eps_n$
and one has
 $\dis\CC_{n}\lesssim \eps_n$. 
The second part of~\eqref{e:dis+M} follows from \eqref{e:EmbHolder}, 
so that the proof is complete.}
\end{proof}

\subsection{Directed \TitleEquation{\R}{R}-trees and the radial map}\label{sec:CharTrees}

As mentioned in the introduction, we would like to view the backward Brownian Web as a flow. 
More specifically, at {\it any} time $t$ and position $x$, 
we want to be able to {\it follow} a backward Brownian trajectory starting at $x$ at time $t$. 
These trajectories will be encoded by the branches of our $\R$-tree and should not be allowed to cross. 

In the following definition we identify a subset of the space of $\alpha$-spatial $\R$-trees whose elements possess
a notion of direction in time and satisfy a {\it monotonicity} assumption, both imposed at the level of the evaluation map $M$. 
Henceforth we use the following shorthand notation. Given an $\R$-tree $\ST$, elements
$\sz_0, \sz_1 \in \ST$, and $s \in [0,1]$, we write $\sz_s$ for the unique element of $\llb\sz_0,\sz_1\rrb$ with
$d(\sz_0,\sz_s) = s\,d(\sz_0,\sz_1)$.  

\begin{definition}\label{def:CharTree}
For $\alpha\in(0,1)$, we define the space of {\it directed $\alpha$-spatial $\R$-trees}, $\Ch^\alpha_\Sp\subset \T^\alpha_\Sp$ consisting of those elements $\zeta=(\ST,\ast,d,M)$, 
whose evaluation map $M$ 
satisfies the following additional conditions\footnote{Recall Definition~\ref{def:SpTree} for the definitions of
$M_t$ and $M_x$.}.
\begin{enumerate}[noitemsep, label=(\arabic*)]
\item\label{i:Back} {\it Monotonicity in time.} For every $\sz_0,\sz_1\in\ST$ and $s \in [0,1]$ one has
\begin{equation}\label{e:Back}
M_t(\sz_s)=
\bigl(M_t(\sz_0)-d(\sz_0,\sz_s)\bigr) \vee \bigl(M_t(\sz_1)-d(\sz_s,\sz_1)\bigr)\;.
\end{equation}
\item\label{i:MonSpace} {\it Monotonicity in space.} Assuming (1) holds, for every $s< t$, interval $I = (a,b)$
and any four elements $\sz_0,\bar \sz_0, \sz_1, \bar \sz_1$ 
such that $M_t(\sz_0)=M_t(\bar \sz_0) = t$, $M_t(\sz_1)=M_t(\bar \sz_1) = s$, $M_x(\sz_0)< M_x(\bar \sz_0)$,
and $M(\llb\sz_0,\sz_1\rrb), M(\llb\bar \sz_0,\bar\sz_1\rrb) \subset [s,t] \times (a,b)$  
, we have
\begin{equation}\label{e:MonSpace}
M_x(\sz_r)\leq M_x(\bar \sz_{r})
\end{equation}
for every $r \in [0,1]$.
\item\label{i:Spread} There exists an isometry $\iota_* \colon \R_+ \to \ST$ such that $\iota_\ast(0) = \ast$ and
\begin{equ}[e:Spread]
M_t(\iota_\ast(s)) = M_t(\ast)-s\,.
\end{equ}
\end{enumerate}
Note that \ref{i:MonSpace} also makes sense in the periodic case if we restrict to 
intervals $(a,b)$ that do not wrap around the whole torus.
\end{definition}

\begin{remark}\label{rem:CharTree}
The first condition guarantees that geodesics are $\vee$-shaped and that the ``time'' coordinate moves at unit speed. 
Together with the first, the second condition enforces the statement that ``characteristics cannot cross''. 
They are still allowed (and forced, 
in our case) to coalesce but their spatial order must be preserved.
The last requirement says that the tree is oriented and has a direction 
which corresponds to the direction of time, i.e. the characteristics move indeed backward in time. 
\end{remark}

\begin{remark}\label{rem:CharTreeFor}
We denote by $\hat\Ch^\alpha_\Sp$ the subspace of $\T^\alpha_\Sp$ defined in exactly the same way but
with $\vee$ replaced by $\wedge$ in~\ref{i:Back} and in~\ref{i:Spread} 
the isometry $\iota_\ast$ such that $\iota_\ast(0) = \ast$ and $M_t(\iota_\ast(s)) = s-M_t(\ast)$ for $s\geq  M_t(\ast)$. 
Note that $\zeta=(\ST,\ast,d,M)\mapsto -\zeta\eqdef(\ST,\ast,d,-M)\in\hat\Ch^\alpha_\Sp$ is an isometric involution. 
\end{remark}


First notice that it is not difficult to show that the properties in the previous definition 
are consistent with the equivalence relation in Definition~\ref{def:SpTree}, 
i.e. if there exists a bijective isometry $\varphi$ such that $\varphi\circ\zeta=\zeta'$ and 
$\zeta$ satisfies the conditions above then so does $\zeta'$. 
In other words, the space $\Ch^\alpha_\Sp$ is a well-defined subset of $\T^\alpha_\Sp$. 
In the next proposition, we study important properties of $\Ch^\alpha_\Sp$. 

\begin{proposition}\label{p:EndProp}
Let $\alpha\in(0,1)$. Then, 
\begin{enumerate}[noitemsep]
\item the set $\Ch^\alpha_\Sp$ is closed in $\T^\alpha_\Sp$, 
\item any $\zeta=(\ST,\ast,d,M)\in\Ch^\alpha_\Sp$ is such that
$\ST$ has a unique open end $\dagger$ which satisfies 
\begin{equ}[e:Ray]
M_t(\sw)=M_t(\sz)-d(\sz,\sw)\;.
\end{equ}
for all $\sz\in\ST$ and $\sw\in\llb\sz,\dagger\rangle$. 
\end{enumerate}
\end{proposition}
\begin{proof}
We first prove that $\Ch^\alpha_\Sp$ is closed. 
Let $\{\zeta^n\}_n \subset \Ch^\alpha_\Sp$ be a sequence converging to $\zeta\in\T^\alpha_\Sp$. 
Since $\zeta$ is directed if and only if, for every $R>0$, $\zeta^{(R)}$ is monotone 
and~\ref{i:Spread} holds for every $s\in[0,R]$, and since 
$\uDelta^\com_\Sp(\zeta^{n,(R)},\zeta^{(R)})\to 0$ for every $R > 0$,
it suffices to restrict to the compact case. 

We start with monotonicity in time.
Take $\sz_0, \sz_1 \in \ST$, let $\CC_n$ be a sequence of correspondences such that 
$\lim_n \uDelta^{\com,\CC_n}_\Sp(\zeta^n,\zeta)\to 0$ and let $\sz_i^n$ be such that 
$(\sz_i^n, \sz_i)\in \CC_n$. 
For any $s \in [0,1]$, {let $\sz_s\in\ST$ and $\sz_s^n$ be defined as above.  
For every $n$, let $\bar\sz_s(n)\in\CC_n$ be a point for which $(\sz_s^n, \sz_s(n))\in\CC_n$. Then, 
by the triangle inequality and the definition of correspondence, $\sz_s(n)$ belongs to 
a compact ball centred at, say, $\sz_0$. Hence, 
the sequence $\{\sz_s(n)\}_n$ is precompact.  
Now, for any limit point $\bar\sz_s\in\ST$ we have
\begin{equs}
|d(\sz_i,\bar\sz_s)-d(\sz_i,\sz_s)|\leq& |d(\sz_i,\bar\sz_s)-d(\sz_i,\sz_s(n))|+\\
&|d(\sz_i,\sz_s(n))-d^n(\sz_i^n, \sz_s^n)|+|d^n(\sz_i^n, \sz_s^n)-d(\sz_i,\sz_s)|\\
=&|d(\sz_i,\bar\sz_s)-d(\sz_i,\sz_s(n))|+\\
&|d(\sz_i,\sz_s(n))-d^n(\sz_i^n, \sz_s^n)|+f_i(s)|d^n(\sz_i^n, \sz_1^n)-d(\sz_i,\sz_1)|
\end{equs}
where $f_0(s)=s$ while $f_1(s)=1-s$ and in the last step we used the definition of $\sz_s$ and $\sz_s^n$. 
Now, all the terms at the right hand side are converging to $0$, which implies that 
$\bar\sz_s$ must be such that $d(\sz_i,\bar\sz_s)=f_i(s)d(\sz_0,\sz_1)$. 
The point $\bar\sz_s$ is therefore unique and given by $\sz_s$, so that the sequence $\{\sz_s(n)\}_n$ 
converges to $\sz_s$. 
It then follows that} 
%
%
%
%
\begin{equ}
M_t(\sz_s) = \lim_{n \to \infty} M_t(\bar \sz_s(n)) = \lim_{n \to \infty} M_t^n(\sz_s^n)\;.
\end{equ}
Since furthermore $\lim_{n \to \infty} d^n(\sz_0^n,\sz_1^n) = d(\sz_0,\sz_1)$ and 
$\lim_{n \to \infty} M_t^n(\sz_i^n) = M_t(\sz_i)$
by the definition of $\uDelta^{\com,\CC_n}_\Sp$, the claim follows.

Regarding monotonicity in space, we perform the same construction, whence we get
\begin{equ}
M_x(\sz_s) = \lim_{n \to \infty} M_t^n(\sz_s^n) \le \lim_{n \to \infty} M_t^n(\sz_{s'}^n)
=  M_x(\sz_{s'})\;,
\end{equ}
as required.

For the last property, let $s\in[0,R]$ and $\sz_s^n\in\ST$ be such that 
$(\iota_\ast^n(s),\sz_s(n))\in\CC^n$. Then, arguing as above, there exists $\sz_s\in\ST$ 
such that {$d(\sz_s(n), \sz_s)$ converges to $0$ as $n\to\infty$}, which, as $(\ast^n,\ast)\in\CC^n$, satisfies 
\begin{equ}
d(\ast,\sz_s)=\lim_{n\to\infty}d(\ast,\sz_s(n))=\lim_{n\to\infty} d^n(\ast^n,\iota_\ast^n(s))=s\,.
\end{equ}
Further, by continuity of $M$, $M_t(\sz_s(n))$ and $M_t(\ast^n)$ respectively converge to 
$M_t(\sz_s)$ and $M_t(\ast)$, we also have 
\begin{equ}
M_t(\sz_s)=M_t(\ast)-s\,.
\end{equ}
Therefore, upon defining the map $\iota_\ast(s)\eqdef \sz_s$, we immediately have that $\iota_\ast$ 
is an isometry and satisfies~\eqref{e:Spread}.

We now move to the second part of the statement. The third property in Definition~\ref{def:CharTree} implies 
that any directed $\R$-tree $\zeta=(\ST,\ast,d,M)$ is unbounded,
since $M$ is continuous and $\ST$ is complete. Therefore, $\ST$ must have at least one unbounded open end. 
{By property~\ref{i:Spread}, the isometry $\iota_\ast$ is such that $L_\ast\eqdef\iota_\ast(\R_+)$ 
is an unbounded ray in $\ST$. As ends are equivalence classes of rays, let $\dagger$ be the unique (necessarily) 
open end such that $L_\ast=\llb\ast,\dagger\rangle$. Now,~\eqref{e:Spread} implies that}
~\eqref{e:Ray} holds for $\sz=\ast$ and any $\sw\in\llb\ast,\dagger\rangle$. 
It then suffices to apply the monotonicity in time, i.e. property~\ref{i:Back}, to see that it must hold for any $\sz\in\ST$.
%
%
%
%
\end{proof}

Thanks to the previous proposition, we can introduce, in the context of directed trees, the {\it radial map}. This is a map
on the $\R$-tree that allows to move along the rays. 

\begin{definition}\label{def:RadMap}
Let $\alpha\in(0,1]$, $\zeta=(\ST,\ast,d,M)\in\Ch^\alpha_\Sp$ and $\dagger$ the open end with unbounded 
rays such that~\eqref{e:Ray} holds. The {\it radial map} $\rho \colon \ST \times \R \to \ST$ associated to $\zeta$ is uniquely defined by postulating that 
\begin{equ}\label{e:RadMap}
\rho(\sz,s) \in \llb\sz,\dagger\rangle\;,\qquad
M_t(\rho(\sz,s)) = s \wedge M_t(\sz)\;.
\end{equ}
If instead $\zeta\in\hat\Ch^\alpha_\Sp$, the radial map $\hat\rho$ 
is defined in the same way but with $\vee$ instead of $\wedge$.
\end{definition}

%

\subsection{Alternative topologies}\label{s:Topo}

Before detailing our alternative construction of the Brownian Web, 
we show how the topology introduced above relates to that of~\cite{FINR}. 
To describe the latter, let first $\R^2_\com$ be the completion of $\R^2$ with respect to the metric {$\rho$ 
in~\cite[eq. (6.1)]{SSS} given by} 
\begin{equ}
\rho((t_1,x_1),(t_2,x_2))\eqdef |\tanh(t_1)-\tanh(t_2)|\vee\Big|\frac{\tanh(x_1)}{1+|t_1|}-\frac{\tanh(x_2)}{1+|t_2|}\Big|
\end{equ}
for all $(t_1,x_1),\,(t_2,x_2)\in\R^2$. (See \cite[Fig.~3]{BWKill} for a cartoon illustrating the geometry of the
resulting compactification of $\R^2$.)
A backward path $\pi$ in $\R^2_\com$ with starting time $\sigma_\pi\in[-\infty,\infty]$ is a continuous map 
$\R \ni t\mapsto(t,\pi(t))\in \R^2_\com$ with $\pi(t) = \pi(\sigma_\pi)$ for all $t \ge \sigma_\pi$. 
We define a metric $d$ on the space $\Pi$ of such paths by 
\begin{equ}[e:TopoCom]
d_\Pi(\pi_1,\pi_2)\eqdef |\tanh(\sigma_{\pi_1})-\tanh(\sigma_{\pi_2})|\vee\sup_{t\leq \sigma_{\pi_1}\wedge\sigma_{\pi_2}}\Big|\frac{\tanh(\pi_1(t))}{1+|t|}-\frac{\tanh(\pi_2(t))}{1+|t|}\Big|
\end{equ}
for all $\pi_1,\pi_2\in\Pi$. Since $(\Pi,d_\Pi)$ is a Polish space, so is the space $\cH$ of compact 
subsets of $\Pi$ endowed with the Hausdorff metric {(see~\cite[Lemma B.2]{SSS2010})}. 

Let $\alpha\in(0,1)$, $\zeta=(\ST,\ast,d,M)\in\Ch_\Sp^\alpha$, 
and $\rho$, $\zeta$'s radial map defined according to~\eqref{e:RadMap}. For $\sz\in\ST$, define
\begin{equ}[e:CompPath]
\pi_\sz(t)\eqdef M_x(\rho(\sz,t))\,,\qquad \text{for all $t\leq M_t(\sz)$.}
\end{equ}
Since $\pi_\sz\in\Pi$ by continuity of $M$, we have a map
\begin{equ}[e:CompZeta]
\Ch^\alpha_\Sp \ni \zeta \mapsto K(\zeta) \eqdef {\overline{\{\pi_\sz\,:\,\sz\in\ST\}}} \subset \Pi\;,
\end{equ}
where the bar denotes closure with respect to the metric in~\eqref{e:TopoCom}. 

\begin{proposition}\label{p:MapTopo}
Let $\alpha\in(0,1)$. For every $\zeta\in\Ch_\Sp^\alpha$, $K(\zeta)$ is compact and the map 
$\zeta\mapsto K(\zeta)$ is continuous from $\Ch_\Sp^\alpha$ to $\cH$. 
\end{proposition}

\begin{remark}\label{rem:MapTopoFor}
Defining $\hat \Pi$ and $\hat \cH$ in the same way, except that 
now $\pi(t) = \pi(\sigma_\pi)$ for all $t \le \sigma_\pi$ and $\le$ is replaced by $\ge$ in
the right-hand side of \eqref{e:TopoCom}, we also have a map
$\hat K \colon \hat\Ch_\Sp^\alpha \to \hat \cH$ given by $\hat K(\zeta) = -K(-\zeta)$.
\end{remark}

For the proof of the previous proposition we will need the following two lemmas. For the first, define
\begin{equ}
\Pi^R\eqdef \{\pi\in \Pi\,:\,\exists\,\,t\leq\sigma_\pi\text{ s.t. } (t,\pi(t))\in [-R,R]^2\}\;,
\end{equ}
and, for $\pi \in \Pi$, write $\pi^R \in \Pi$ for the stopped path such that
\begin{equ}
\sigma_{\pi^R} = \sigma_\pi\;,\qquad \pi^R(t) = 
\left\{\begin{array}{cl}
	\pi(R) & \text{if $t \ge R$,} \\
	\pi(-R) & \text{if $t \le -R$,} \\
	\pi(t) & \text{otherwise.}
\end{array}\right.
\end{equ}

\begin{lemma}\label{l:MapTopo}
Let $\sK$ be a subset of $\Pi$ and, for $R>0$, let $\sK^R \subset \Pi$ be defined as
\begin{equ}[e:CompRestr]
\sK^R\eqdef\{\pi^R\,:\, \pi\in \sK \cap \Pi^R\}\,.
\end{equ}
If for all $R>0$, the family of paths in $\sK^R$ is equicontinuous then $\sK$ is relatively compact. 
\end{lemma}
\begin{proof}
Our main ingredient then is the fact that, since $|1-\tanh R|\le e^{-R}$, one has the bounds
\begin{equs}[3]
x &\ge R &&\Rightarrow& \rho((t,x),(t,\infty)) &\le e^{-R} \quad\forall t\;,\\
x &\le -R &\quad&\Rightarrow&\quad \rho((t,x),(t,-\infty)) &\le e^{-R}\quad\forall t\;, \label{e:boundsCompact}\\
|t| &\ge R &&\Rightarrow& \rho((t,x),(t,y)) &\le {2\over R}\quad\forall x,y\;.
\end{equs}
Writing $\pi^\pm_t$ for the path with $\sigma_{\pi^\pm_t} = t$ and $\pi^\pm_t(s) = \pm \infty$, it follows that
for every $\pi \in \Pi$ and every $R\ge 1$
one has $d_\Pi(\pi,\pi^R) \le 2/R$. If furthermore
$\pi \not\in \Pi^R$, then $d_\Pi(\pi, \pi^+_{\sigma_\pi}) \wedge d_\Pi(\pi, \pi^-_{\sigma_\pi}) \le 2/R$.

It remains to note that, given $\eps > 0$, we can cover $\sK^{4/\eps}$ with finitely many balls of radius $\eps / 2$ 
by Arzelà--Ascoli, so that $\sK \cap \Pi^R$ is covered by the balls with same centres and radius $\eps$. 
The complement of $\Pi^R$ on the other hand can be covered by finitely many balls of radius $\eps$ centred at 
elements of type $\pi^\pm_t$ for $t \in \eps \Z \cap [-4\eps^{-1},4\eps^{-1}]$.
\end{proof}

The next lemma highlights the fact that if two directed trees 
are close then also the respective rays must be close in a suitable sense. 

\begin{lemma}\label{l:VicPath}
Let $\alpha\in(0,1)$ and $\zeta_1,\,\zeta_2\in\Ch^\alpha_\Sp$. Let $r>0$ and assume there exists a correspondence 
$\CC$ between $\ST_1^{(r)}$ and $\ST^{(r)}_2$ such that 
$\Delta_\Sp^{\com,\CC}(\zeta_1^{(r)},\zeta_2^{(r)})<\eps$
for some $\eps>0$. Let $(\sz_1,\sz_2)\in\CC$ and define a new correspondence $C_\CC$ as 
\begin{equ}[e:CorrP]
C_\CC\eqdef\CC\cup\{(\rho_1(\sz_1,s),\rho_2(\sz_2,s)\colon -r\leq s\leq M_{1,t}(\sz_1)\wedge M_{2,t}(\sz_2)\}
\end{equ}
Then, 
\begin{equ}
\frac12\dis C_\CC +\sup_{(\sz,\bar\sz)\in C_\CC}\|M_1(\sz)-M_2(\bar\sz)\|\lesssim \eps+\|M_1\|_\alpha^{(r)}\eps^\alpha
\end{equ}
\end{lemma}
\begin{proof}
Let $(\sz_1,\sz_2)\in\CC$ be as in the statement and $-r\leq s\leq M_{1,t}(\sz_1)\wedge M_{2,t}(\sz_2)$. 
Let $\sz_s\in\ST_1$ be such that $(\sz_s,\rho_2(\sz_2,s))\in\CC$. Notice that for any $(\sw_1,\sw_2)\in\CC$, 
by the triangle inequality and the assumption $\Delta_\Sp^{\com,\CC}(\zeta_1^{(r)},\zeta_2^{(r)})<\eps$, we have 
\begin{equs}
|d_1(\rho_1(\sz_1,s),\sw_1)-d_2(\rho_2(\sz_2,s),\sw_2)|\leq d_1(\sz_s,\rho_1(\sz_1,s))+\dis\CC\leq d_1(\sz_s,\rho_1(\sz_1,s)) +2\eps
\end{equs}
which means that we only need to focus on $d_1(\sz_s,\rho_1(\sz_1,s))$. 
Now, if $\rho(\sz_1,s)$ belongs to the ray starting at $\sz_s$, by~\eqref{e:Ray}, we have
\begin{equ}
d_1(\sz_s,\rho_1(\sz_1,s))=M_{1,t}(\sz_s)-M_{1,t}(\rho_1(\sz_1,s))=M_{1,t}(\sz_s)-s\leq M_{2,t}(\rho_2(\sz_2,s))+\eps-s\leq \eps \,.
\end{equ}
Otherwise,
\begin{equs}
d_1(\sz_s,\rho_1(\sz_1,s))&=d_1(\sz_s,\sz_1)-d_1(\sz_1,\rho_1(\sz_1,s))\leq d_2(\rho_2(\sz_2,s), \sz_2)+\eps-d_1(\sz_1,\rho_1(\sz_1,s))\\
&=M_{2,t}(\sz_2)-s+\eps-M_{1,t}(\sz_1)+s\leq 2\eps\,.
\end{equs}
Therefore, we immediately conclude that $\dis C_\CC <4\eps$. 
Concerning the bound on the evaluation maps, we have
\begin{equs}
\|M_1(\rho_1(\sz_1,s))-M_2(\rho_2(\sz_2,s))\|\leq& \|M_1(\rho_1(\sz_1,s))-M_1(\sz_s)\|\\
&+\|M_1(\sz_s)-M_2(\rho_2(\sz_2,s))\|\lesssim \|M_1\|_\alpha^{(r)}\eps^\alpha +\eps
\end{equs}
where we exploited the H\"older continuity of $M_1$, the bound on $d_1(\sz_s,\rho_1(\sz_1,s))$ 
and the fact that $(\sz_s,\rho_2(\sz_2,s))\in\CC$. The conclusion follows at once. 
\end{proof}

We are now ready for the proof of Proposition~\ref{p:MapTopo}. 

\begin{proof}[of Proposition~\ref{p:MapTopo}]
Let $\zeta=(\ST,\ast,d,M)\in\Ch^\alpha_\Sp$ and $K(\zeta)$ be as in~\eqref{e:CompZeta}. 
By definition, $M^{-1}(\Lambda_R)\subset B_d(\ast, b_\zeta(R)]$ and, since $\ST$ is a tree, 
if $\sz\in B_d(\ast, b_\zeta(R)]$ then $\rho(\sz,s)\in B_d(\ast, b_\zeta(R)]$ for all $s\in[-R,M_t(\sz)]$. 
Moreover, $M$ is $\alpha$-H\"older continuous on $B_d(\ast, b_\zeta(R)]$, therefore $K(\zeta)^R$ 
as defined in~\eqref{e:CompRestr} consists of equicontinuous paths and Lemma~\ref{l:MapTopo} 
implies that $K(\zeta)\in\cH$. 

Let now $\{\zeta^n=(\ST^n,\ast^n,d^n,M^n)\}_n\subset \Ch^\alpha_\Sp$ be a sequence converging to 
$\zeta\in \Ch^\alpha_\Sp$ with respect to $\Delta_\Sp$. 
In view of Proposition~\ref{p:Compactness}, the evaluation maps $M^n$ are uniformly proper and have uniformly bounded 
$\alpha$-H\"older norm when restricted to balls of fixed size. Hence, arguing as above, we see that 
$\cup_n K(\zeta^n)$ is relatively compact in $\Pi$ which, thanks to~\cite[Lemma B.3]{SSS2010}, 
implies that the sequence $\{K(\zeta^n)\}_n$ is relatively compact in $\cH$ with respect to the Hausdorff topology. 

It remains to show that $K(\zeta^n)$ converges to $K(\zeta)$ in $\cH$. By~\cite[Lemma B.1]{SSS2010}, 
we need to prove that for every $\pi_\sz\in K(\zeta)$ there exists a sequence 
$\pi_{\sz_n}\in K(\zeta^n)$ such that $d_\Pi(\pi_\sz,\pi_{\sz_n})\to 0$, {and that, if $\{\pi_{\sz_n}\}_n$ is a sequence 
such that $\pi_{\sz_n}\in K(\zeta_n)$, then any of its cluster points $\pi$ belongs to $K(\zeta)$. }
 
We begin with the first. Let $\sz\in\ST$ and $\eps>0$. Pick $C>0$ big enough so that $\sz\in B_d(\ast, C]$ 
and $\sup_n b_{\zeta^n}(\eps^{-1})\leq C$. 
Let $n$ be sufficiently large so that there exists a correspondence $\CC_n$  
between $B_d(\ast, C]$ and $B_{d^n}(\ast^n, C]$ with $\Delta^{\com,\CC_n}_\Sp(\zeta^{(C)}, \zeta^{n,\,(C)})<\eps$. 
Let $\sz_n\in B_{d^n}(\ast^n, C]$ with $(\sz,\sz_n)\in\CC_n$ and 
define $\pi_\sz$ and $\pi_{\sz_n}$ as in~\eqref{e:CompPath}. 
Since $|M_t(\sz)-M^n_t(\sz_n)|<\eps$, it follows that $|\tanh(\sigma_{\pi_z})-\tanh(\sigma_{\pi_{z_n}})|<\eps$. 

To estimate the distance between $\pi_\sz(s)$ and $\pi_{\sz_n}(s)$ for
$s \le \sigma_{\pi_z}\wedge \sigma_{\pi_{z_n}}$, we first consider the case $s \ge -\eps^{-1}$. 
Since $C$ is large enough so that $\rho^n(\sz_n,s) \in B_{d^n}(\ast^n, C]$, we can apply Lemma~\ref{l:VicPath} 
and get 
\begin{equs}[e:ConvRays]
|\pi_\sz(s)-\pi_{\sz_n}(s)|=|M_x(\rho(\sz,s))-M^n_{x}(\rho_n(\sz_n,s))|\lesssim \eps+\|M\|_\alpha^{(C)}\eps^\alpha\,.
\end{equs}
For $s<-\eps^{-1}$ we use again the last bound of~\eqref{e:boundsCompact}. Combining these bounds, 
we obtain $d_\Pi(\pi_\sz,\pi_{\sz_n})\lesssim \eps^\alpha$. 

{ We now turn to the second. Let $\{\pi_{\sz_n}\}_n$ be a sequence 
such that $\pi_{\sz_n}\in K(\zeta_n)$, $\pi$ be one of its cluster points. 
Passing at most to a subsequence (which we will still index by $n$) we can assume $d_\Pi(\pi_{\sz_n},\pi)$ 
converges to $0$. Our goal is to show that there exists $\sz\in\ST$ such that $\pi=\pi_\sz$. 

Let $\bar R$ be big enough so that for all $n$, $\sz_n\in B_{d_n}(\ast_n, \bar R]$, 
which must exist since the evaluation maps $M^n$ are uniformly proper and $\pi_{\sz_n}$ converges 
in $d_\Pi$. 
Moreover, since $\ST^n$ is a directed $\R$-tree, $\sz_n\in B_{d_n}(\ast_n, \bar R]$ implies that 
for every $t\in[ -\bar R, M_t^n(\sz_n)]$, $\rho_n(\sz_n, t)\in B_{d_n}(\ast_n, \bar R]$. 
Let $R>\bar R$ and $\CC_n$ be a correspondence such that 
$\lim_n\Delta^{\com,\CC_n}_\Sp(\zeta^{(R)}, \zeta^{n,\,(R)})=0$. 
Since $\sz_n\in B_{d_n}(\ast_n, R]$ and $\ST^n$ is a directed $\R$-tree, 
for every $t\in[ - R, M_t^n(\sz_n)]$, $\rho_n(\sz_n, t)\in B_{d_n}(\ast_n, R]$. 
Hence, for $n\in\N$ and $t\leq M_t^n(\sz_n)\wedge\sigma_\pi$, there is $\sz_t(n)\in\ST^{(R)}$ such that 
$(\rho_n(\sz_n,t),\sz_t(n))\in\CC_n$. 
As $\Delta^{\com,\CC_n}_\Sp(\zeta^{(R)}, \zeta^{n,\,(R)})$ and $d_\Pi(\pi_{\sz_z},\pi)$
converge to $0$, for any $s,t\geq -R$ we have 
\begin{equs}
&\lim_n|d(\sz_t(n),\sz_s(n))-d^n(\rho_n(\sz_n,t),\rho_n(\sz_n,s))|=0\label{e:DisTopo}\\
&\lim_n \|M(\sz_t(n))-M^n(\rho_n(\sz_n,t))\|=0=\lim_n |M_x(\sz_t(n))-\pi(t)|=0\label{e:MapTopo}\,.
\end{equs}
For every $t\in[-R,\sigma_\pi]$, the set $\{\sz_t(n)\}_n\subset B_d(\ast,R]$ is bounded, 
and therefore relatively compact. 
Hence, via a diagonal argument, we can find a subsequence in $n$ such that for every $t$ in a 
countable dense subset $D^R_\pi$ of $[-R,\sigma_\pi]$ containing $\sigma_\pi$, 
$\sz_t(n)$ converges to $\sz_t\in\ST^{(R)}$. 
Let $\sz\eqdef\sz_{\sigma_\pi}$ and note that by~\eqref{e:MapTopo}, for every $t\in D^R_\pi$ 
$M_x(\sz_t)=\pi(t)$ and $M_t(\sz_t)=t$, so that in particular $M_x(\sz)=\pi(0)$ and $M_t(\sz)=\sigma_\pi$. 
We now want to show that $\sz_t=\rho(\sz_t)$, which in turn, since $\zeta$ is a directed 
$\R$-tree and by the definition of the radial map $\rho$, follows if we prove that 
$d(\sz,\sz_t)=d(\sz,\rho(\sz_t))$. 
To see this latter point, note that by~\eqref{e:Ray} and~\eqref{e:RadMap}, we have
\begin{equ}
d(\sz,\rho(\sz,t))=M_t(\sz)-t=\lim_n M^n_t(\sz_n)-t=\lim_n d_n(\sz_n, \rho^n(\sz_n,t))=d(\sz,\sz_t)
\end{equ}
where the last step follows by~\eqref{e:DisTopo}. 
As a consequence, we have shown for every $t\in D^R_\pi$, $\pi(t)=M_x(\sz_t)=M_x(\rho(\sz,t))=\pi_\sz(t)$, 
so that the same equality holds for any $R$, as $R$ was arbitrary, and $t$, by continuity of $\pi$ and $M$. 
Therefore the proof is concluded. }
\end{proof}

In general, we cannot expect the map $K$ to be injective. Indeed, there is no mechanism that {\it a priori} 
prevents different branches of the tree to be mapped via the evaluation map to the same path.  

In the following definition, we introduce a (measurable) subset of $\Ch^\alpha_\Sp$ whose 
elements satisfy a condition, the {\it tree condition}, which allows to distinguish two rays in the tree based 
on their images under the evaluation map. 

\begin{definition}\label{d:TreeCond}
Let $\alpha\in(0,1)$. We say that $\zeta=(\ST,\ast,d,M)\in\Ch^\alpha_\Sp$ satisfies the {\it tree condition} 
if 
\begin{enumerate}[label=($\mathfrak{t}$)]
\item\label{i:TreeCond} for all $\sz_1,\sz_2\in\ST$, if $M(\sz_1)=M(\sz_2)=(t,x)$ and there exists $\eps>0$ such that 
$M(\rho(\sz_1,s))=M(\rho(\sz_2,s))$ for all $s\in[t-\eps,t]$, then $\sz_1=\sz_2$.
\end{enumerate}
We denote by $\Ch^\alpha_\Sp(\ft)$, the subset of $\Ch^\alpha_\Sp$ whose elements satisfy~\ref{i:TreeCond}. 
\end{definition}

Condition~\ref{i:TreeCond}, guarantees that different rays on the tree under study are mapped, via the evaluation map, 
to paths which cannot agree on any open interval up to the time they coalesce. 
{Alternatively said,~\ref{i:TreeCond} is equivalent to requiring that $M$ is {\it segment-injective}, i.e. 
that if two segments of positive length 
have the same image under $M$ then they coincide. More precisely, the {\it segment-injective} property is 
\begin{enumerate}[label=($\mathfrak{t}'$)]
\item if $\sz_1,\sz_2, \sz_1',\sz_2'\in\ST$ such that $d(\sz_1,\sz_2)\wedge d(\sz_1',\sz_2')>0$ satisfy
$M(\llb\sz_1,\sz_2\rrb)=M(\llb\sz_1',\sz_2'\rrb)$, then $\llb\sz_1,\sz_2\rrb=\llb\sz_1',\sz_2'\rrb$. 
\end{enumerate}}
It is not difficult to construct examples of directed trees for which~\ref{i:TreeCond} 
does not hold, while it clearly does if the evaluation map is injective. 
However, we cannot expect the evaluation map of the Brownian Web to be injective 
because of the presence of special points from which multiple trajectories depart (see Section~\ref{sec:DBW}). 
In the following lemma, the proof of which is immediate, 
we provide a less trivial example. 

\begin{lemma}\label{l:TreeCond}
Let $\alpha\in(0,1)$ and $\zeta=(\ST,\ast,d,M)\in\Ch^\alpha_\Sp$. If there exists a dense subtree $T$ of $\ST$ 
such that $(T,\ast,d,M\restr T)$ satisfies~\ref{i:TreeCond} then so does $\zeta$. 
Moreover, the subset of $\Ch^\alpha_\Sp$ whose elements satisfy~\ref{i:TreeCond} is measurable 
with respect to the Borel $\sigma$-algebra generated by $\Delta_\Sp$ in~\eqref{e:Metric}. 
\end{lemma}
\begin{proof}
The first part of the statement follows by Lemma~\ref{l:Rtrees} point 2. 
For the second, it suffices to observe that the set $\Ch^\alpha_\Sp(\ft_{n^{-1}})\subset \T^\alpha_{\Sp}$, $n\in\N$, 
whose elements are such that~\ref{i:TreeCond} holds for $\eps=n^{-1}$, is closed and clearly
$\Ch^\alpha_\Sp(\ft)=\cup_{n\in\N}\Ch^\alpha_\Sp(\ft_{n^{-1}})$.
\end{proof}

We conclude this section by showing that on $\Ch^\alpha_\Sp(\ft)$, $K$ is indeed injective.  

\begin{proposition}\label{p:MapInj}
Let $\alpha\in(0,1)$ and $\Ch^\alpha_\Sp(\ft)$ be given as in Definition~\ref{d:TreeCond}. 
Then, the map $K$ in~\eqref{e:CompZeta} is injective on $\Ch^\alpha_\Sp(\ft)$. 
\end{proposition}

\begin{remark}\label{rem:Inj}
Even though the map $K$ is injective on $\Ch^\alpha_\Sp(\ft)$, it is not on any open subset of $\Ch^\alpha_\Sp$ 
and the set $\Ch^\alpha_\Sp(\ft)$ is not closed. 
To see this, let $\zeta=(\ST,\ast,d,M)\in\Ch^\alpha_\Sp(\ft)$. Now, add to $\ST$ a branch of arbitrary finite length
and impose that the image of the new branch via the evaluation map is contained in $M(\ST)$. 
We can clearly do so in such a way that the new spatial-tree $\zeta'$ is again directed. 

Now, $\zeta'$ does not satisfy condition~\ref{i:TreeCond}, $K(\zeta)=K(\zeta')$ and 
upon tuning the length of the extra branch, we can make it arbitrarily close to $\zeta$. 
\end{remark}

\begin{proof}
Let $\zeta,\,\zeta'\in\Ch^\alpha_\Sp$ be such that~\ref{i:TreeCond} holds 
and $K(\zeta)\equiv K(\zeta')$. Then, for all $\sz\in\ST$ there exists a unique element $\phi(\sz)\in\ST'$ such that $\pi_\sz\equiv\pi_{\phi(\sz)}$ and therefore not only $M(\sz)=M'(\phi(\sz))$,
but $M(\rho(\sz,s))=M'(\rho'(\sz',s))$ for all $s$. To show that $\phi$ is the required isomorphism, 
assume by contradiction that there exist $\sz_1,\,\sz_2\in\ST$ such that 
$d(\sz_1,\sz_2)\neq d'(\varphi(\sz_1),\varphi(\sz_2))$ and let 
$\bar s,\,\bar s'\leq M_t(\sz_1)\wedge M_t(\sz_2)$ be  the 
first times at which $\rho(\sz_1,\bar s)=\rho(\sz_2,\bar s)$ and $\rho'(\varphi(\sz_1),\bar s')=\rho'(\varphi(\sz_2),\bar s')$ respectively. 
Since $d(\sz_1,\sz_2)\neq d'(\varphi(\sz_1),\varphi(\sz_2))$ we have $\bar s\neq \bar s'$ so that, 
without loss of generality, 
we can assume $\bar s>\bar s'$. Since $\ST$ is a tree, we must have 
\begin{equ}
M'(\rho'(\varphi(\sz_1),s)))=M(\rho(\sz_1,s))=M(\rho(\sz_2,s))=M'(\rho'(\varphi(\sz_2),s)))
\qquad \forall s\in[\bar s',\bar s]\;,
\end{equ}
which, by~\ref{i:TreeCond}, implies that $\rho'(\varphi(\sz_1),\bar s)=\rho'(\varphi(\sz_2),\bar s)$. Hence, 
$d(\sz_1,\sz_2)=d'(\varphi(\sz_1),\varphi(\sz_2))$ and we reach the required contradiction. 
\end{proof}

\begin{remark}\label{rem:MapTopoPeriodic}
In the periodic case, let $\Pi_\per$ be the set of backward periodic paths endowed with the metric $d_\Pi^\per$ 
whose definition is the same as in~\eqref{e:TopoCom} but in the second argument of the maximum 
the inner metric is replaced by the periodic one, 
i.e. for $\pi_1,\,\pi_2\in\Pi_\per$ and $t\leq \sigma_{\pi_1}\wedge\sigma_{\pi_2}$, 
we take $\inf_{k\in\Z}|\pi_1(t)-\pi_2(t)+k|$. 
Let $\cH_\per$ be the set of compact subsets of $\Pi_\per$ with the Hausdorff metric. 
Then, Propositions~\ref{p:MapTopo} and~\ref{p:MapInj} remain true, which means that
the map $K\,:\,\Ch^\alpha_{\Sp,\per}\to\cH_\per$ defined as in~\eqref{e:CompZeta} is continuous 
and its restriction to $\Ch^\alpha_{\Sp,\per}(\ft)$ is injective. 
\end{remark}

\section{The Brownian Web Tree and its dual}\label{sec:BW-BCM}

Here, we provide an alternative (and finer) characterisation of the Brownian Web 
so to be able to view it as a directed spatial $\R$-tree. 
%
%
%
%
%

\subsection{An alternative characterisation of the Brownian Web}\label{sec:BW}

In this section, we will build both the standard (or planar) backward Brownian Web
and its {\it periodic} (or cylindric) counterpart as given in~\cite{CMT}.  Since the two constructions are almost identical, 
we will mainly focus on the first and limit ourselves to indicate what 
needs to be modified in order to accommodate the second 
(see Remarks~\ref{rem:PeriodicBW},~\ref{rem:p:PeriodicBW},~\ref{rem:t:PeriodicBW}). 

Consider a standard probability space $(\Omega, \CA, \Prob)$ supporting countably many 
independent standard Brownian motions $\{W^\da_k\}_{k\in\N}$ starting at $0$ and running backward in time, 
i.e. from $0$ to $-\infty$. 
Fix a countable dense set $\cD\eqdef\{z_k=(t_k,x_k)\,:\,k\in\N\}$ of $\R^2$, with $z_0=(0,0)$. 
Then, build inductively a family of coalescing backward Brownian motions $\{\pid_{z_k}\}_{k\in\N}$ such that 
$\pid_{z_k}$ starts at $x_k$ at time $t_k$. As in~\cite[Section 3]{FINR}, one way to do so is to set 
$\pid_{z_0}(t)=W^\da_0(t)$ and then define $\pid_{z_k}(t) = x_k + W^\da_0(t-t_k)$ for all $\tau_k\leq t\leq t_k$, 
where $\tau_k$ is the largest value such that $x_k + W^\da_k(\tau_k-t_k) = \pid_{z_\ell}(\tau_k)$ for some $\ell < k$, and 
for $t\leq\tau_k$, $\pid_{z_k}(t) = \pid_{z_\ell}(t)$. 
The construction guarantees that even though $\ell$ may not be unique, 
the definition of $\pid_k$ is. 

For every $n\in\N$, let $\tilde\ST^\da_n(\cD)\eqdef \{(t,\pid_{z_k})\,:\,t\leq t_k\,,\, k\leq n\}$ and $\tilde\ST^\da_\infty(\cD)$ be 
the space defined as before but in which $k$ is free to range over all of $\N$. 
Now, for $n\in\bar\N\eqdef\N\cup\{\infty\}$, consider  the equivalence relation $\sim$ on $\tilde\ST^\da_n(\cD)$, given by 
\begin{equation}\label{e:Equiv}
\text{$(t,\pid_{z_i})\sim(t,\pid_{z_j})$ if and only if $\pid_{z_i}(s)=\pid_{z_j}(s)$ $\forall\,s\leq t$}
\end{equation}
for $t\leq t_i\wedge t_j$ and $i,j\leq n$. We now introduce 
$\zeta^\da_n(\cD)\eqdef(\ST^\da_n(\cD),\ast^\da,d^\da,M_n^{\da,\cD})$, as
\begin{equs}[e:Skeleton]
\ST^\da_n(\cD)&\eqdef\tilde\ST_n(\cD)/\sim\,,\\
\ast^\da&\eqdef (0,\pid_0)\,,\\
d^\da((t,\pid_{z_i}),(s,\pid_{z_j}))&\eqdef (t+s)-2\tau^\da_{t,s}(\pid_{z_i},\pid_{z_j})\,,\\
M_n^{\da,\cD}((s,\pid_{z_i}))&= (M^{\da,\cD}_{n,t}((s,\pid_{z_i})),M^{\da,\cD}_{n,x}((s,\pid_{z_i})))\eqdef(s,\pid_{z_i}(s))\,,
\end{equs}
where $i,j\leq n$ and, in the definition of the {\it ancestor metric} $d^\da$, 
$\tau^\da_{t,s}(\pid_{z_i},\pid_{z_j})\eqdef\sup\{r<t\wedge s\,:\, \pid_{z_i}(r)=\pid_{z_j}(r)\}$. 

\begin{remark}\label{rem:PeriodicBW}
The construction in the periodic setting is analogous. Indeed, it suffices to replace the family of backward Brownian motions
$\{B^\da_k\}_k$ with a family of periodic ones defined via $B_k^{\da,\per}\eqdef B^\da_k \mod 1$, 
take a countable dense set $\cD^\per\eqdef \{w_k=(s_k,y_k)\,:\,k\in\N\}$ of $\R\times\T$, 
build $\{\pi^{\per,\da}_{w_k}\}_{k\in\N}$ as before and 
define $\zeta^{\per,\da}_n(\cD^\per)=(\ST^{\per,\da}_n(\cD),\ast^\da,d^\da,M_n^{\per,\cD^\per, \da})$ as in~\eqref{e:Skeleton}. 
\end{remark}

%

The construction above readily implies a number of properties each of the $\zeta^\da_n(\cD)$'s enjoys. 
Indeed, for every $n\in\N$ finite, $\zeta^\da_n(\cD)$ is a spatial $\R$-tree which is 
monotone in both space and time, 
{and it further satisfies property~\ref{i:Spread} in Definition~\ref{def:CharTree} as can be 
readily seen by setting 
$\iota_\ast\colon\R_+\to\ST^\da_n(\cD)$ as $\iota_\ast(s)\eqdef (s,\pid_0)$. Moreover,
as a consequence of the fact that Brownian trajectories are $\alpha'$-H\"older continuous 
for any $\alpha'<\frac{1}{2}$, $M^{\da,\cD}_n$ 
is little $\alpha$-H\"older continuous for any $\alpha\in(\alpha',1/2)$. In other words, 
for every $n\in\N$, $\zeta^\da_n(\cD)\in\Ch^\alpha_\Sp$. }
In the next proposition, we will show that the sequence $\{\zeta^\da_n(\cD)\}_n$ is not only tight in $\T^\alpha_\Sp$ 
for any $\alpha<1/2$, but it actually converges to a unique limit in $\Ch^\alpha_\Sp$ which 
can be explicitly characterised starting from $\zeta^\da_\infty(\cD)$. 

\begin{proposition}\label{p:BW}
Let $\cD$ be a countable dense of $\R^2$ containing $(0,0)$ and, for $n\in\bar\N$, let 
$\zeta^\da_n(\cD)\eqdef(\ST^\da_n(\cD),\ast^\da,d^\da,M_n^{\da,\cD})$ be defined according to~\eqref{e:Skeleton}. Then, 
for every $\alpha<\frac{1}{2}$ the sequence $\{\zeta^\da_n(\cD)\}_{n\in\N}$ converges in $\T^\alpha_\Sp$ 
to a unique limit $\zeta^\da(\cD)\eqdef(\ST^\da(\cD),\ast^\da,d^\da, M^{\da,\cD}){\in\Ch^\alpha_\Sp}$, 
where $\ST^\da(\cD)$ is the completion of 
$\ST^\da_\infty(\cD)$ and $M^{\da,\cD}$ is the unique continuous
extension of $M_\infty^{\da,\cD}$ to all of $\ST^\da(\cD)$.

Moreover, for any fixed $\theta>\tfrac32$ and all $r>0$ there exists a random variable $c=c(r) \in \R_+$ 
depending only on $r$ such that for all $\eps>0$ 
\begin{equ}[e:nMEC]
\cN_{d^\da}(\ST^{\da,\,(r)}(\cD), \eps)\leq c \eps^{-\theta}\;,\quad \text{almost surely,}
\end{equ}
where $\cN_{d^\da}(\ST^{\da,\,(r)}(\cD), \eps)$ is defined as in~\eqref{e:EpsNet}, i.e. it is the 
cardinality of the minimal $\eps$-net in $\ST^{\da,\,(r)}(\cD)$ with respect to $d^\da$. 

At last, almost surely $M^{\da,\cD}$ is surjective and~\ref{i:TreeCond}{, 
given in Definition~\ref{d:TreeCond},} holds. 
\end{proposition}

\begin{proof}
We fix $\cD$ once and for all for the duration of this proof and therefore suppress its dependence in the notation.
By construction, the sequence $\{\zeta^\da_n\}_n$ of $\alpha$-spatial $\R$-trees is such that for every $n\in\N$, 
$\zeta^\da_n$ is embedded into $\zeta^\da_{n+1}$, { and, as argued above}
$\zeta^\da_n\in\Ch^\alpha_\Sp$. 
Hence, Lemma~\ref{l:Emb} and { part 1. of }Proposition~\ref{p:EndProp} guarantee that, 
provided that the sequence is tight in $\T^\alpha_\Sp$, 
it converges to a unique 
$\zeta^\da = (\ST^\da,\ast^\da,d^\da, M^{\da})\in\Ch^\alpha_\Sp$ 
which {further satisfies~\eqref{e:nMEC}, $M^{\da}$ is surjective and~\ref{i:TreeCond} holds}. 

Since every $\zeta^\da_n$ is canonically embedded in $\zeta^\da_\infty=(\ST^\da_\infty,\ast^\da,d^\da,M_\infty^{\da})$, if we show that, almost surely, $\ST^\da_\infty$ 
(which is an $\R$-tree and hence, by Point 2 in Theorem~\ref{thm:Rtrees} so is its completion) is locally compact and 
$M^{\da}_\infty$ is proper and uniformly little $\alpha$-H\"older continuous on bounded balls, 
then we have a bound uniform in $n$ on both the size of 
the $\eps$-nets of balls in $\ST^\da_n$ and the local modulus of continuity of $M^{\da}_n$, so that tightness of 
the sequence 
follows readily from Proposition~\ref{p:Compactness}.

Let $r\ge 1$. We start by introducing an event on which $\ST^{\da,\,(r)}_\infty$ is enclosed between two paths. 
Let $R>r$, $Q_R^\pm$ be two squares of side $1$ centred at $(r+1,\pm(2R+1))$ and 
$z^\pm=(t^\pm,x^\pm)$ be two points in $\cD\cap Q_R^\pm$, respectively. 
By the non-crossing property of our coalescing paths, on the event
\begin{equ}\label{e:ER}
E_R\eqdef\{\sup_{0\geq s\geq -r}|\pid_0(s)|\leq R\,,\,\sup_{t^\pm\geq s\geq -r}|\pid_{z^\pm}(s)-x^\pm|\leq R\}
\end{equ}
any element $(s,\pid_z)\in\ST^{\da,\,(r)}_\infty$ with $z=(t,x)\in\cD$ is necessarily such that 
$s\in[-r, t\wedge r]$ and $\pid_{z^-}(s)<\pid_z(s)<\pid_{z^+}(s)$. 
Moreover, by the reflection principle, we have
\begin{equation}\label{e:RP}
\Prob(E_R^c)\leq C_1 \frac{\sqrt{r}}{R} e^{-\frac{R^2}{2r}}
\end{equation}
where $E^c_R$ is the complement of $E_R$ in $\Omega$, and $C_1$ is a positive constant independent of $r$ and $R$.

Now, in order to show that, almost surely, $\ST^{\da,\,(r)}_\infty$ is relatively compact, note that { 
we can brutally bound}
\begin{equ}
\Prob \bigl(\exists \eps \in (0,1]\,:\,\cN_{d}(\ST^{\da,\,(r)}_\infty,\eps) \ge K\eps^{-\theta}\bigr)
\le \sum_{n \ge 1} \Prob \bigl(\cN_{d}(\ST^{\da,\,(r)}_\infty,2^{-n}) \ge K2^{\theta (n-1)}\bigr)\;.
\end{equ}
The following lemma implies~\eqref{e:nMEC} (and in fact that $\P(c > K)\lesssim 1/\sqrt K$) and consequently 
relative compactness. 

\begin{lemma}\label{l:TB}
There exists a constant $C=C(r)>0$ such that
\begin{equation}\label{e:TB}
\P \big(\cN_{d}(\ST^{\da,\,(r)}_\infty,\eps) > K \eps^{-3/2}\big) \le {C\over \sqrt K}
\end{equation}
uniformly over $\eps\in(0,1]$ and $K\geq 1$.
\end{lemma}
\begin{proof}
Let $R>r$ and set $\tilde R\eqdef 3R+1$. For $t_0,t_1\in\R$, $t_0>t_1$, we define
\begin{equation}\label{e:CPS}
\Xi_R(t_0,t_1) \eqdef \left\{\rho(\sz,t_1)\,:\, M_{\infty,t}^{\da}(\sz)>t_0
\text{ and }M_{\infty,x}^{\da}(\rho(\sz,t_1))\in[-\tilde R,\tilde R]\right\}
\end{equation}
where $\rho$ is the radial map of $\ST^\da_\infty$ defined as in~\eqref{e:RadMap}, and set 
$\eta_R(t_0,t_1)$ to be the cardinality of $\Xi_R(t_0,t_1)$. 
By the definition of $\ST^\da_\infty$, $\eta^R(t_0,t_1)$ has the same distribution as 
the quantity $\hat\eta(t_0,t_1;-\tilde R,\tilde R)$ of~\cite[Definition 2.1]{FINR}, 
which is almost surely finite by~\cite[Proposition 4.1]{FINR}. 

Consider the numbers $L_\eps$ and the times $t_k^\eps$ given by 
\begin{equation}\label{def:partition}
L_\eps \eqdef \Big\lceil \frac{8r}{\eps}\Big\rceil + 1,\,\qquad t_k^\eps\eqdef r-k\frac{\eps}{4}\,,\qquad k=0,\dots,L_\eps-1,
\end{equation}  
where, for $x\in\R$, $\lceil x\rceil$ is the least integer greater than $x$. 
We now claim that, on the event $E_R$,
\begin{equ}[e:claimBoundNeps]
\cN_{d}(\ST^{\da,\,(r)}_\infty,\eps) \le \sum_{k=0}^{L_\eps-1}\eta^R(t_k^\eps,t_{k+1}^\eps)\;.
\end{equ}
Indeed, if $(t,\pid_z)\in\ST^{\da,\,(r)}_\infty$ for some $t \in \R$ and $z\in\cD$, then by definition of the metric $t\in[-r,r]$ and 
$\pid_{z^-}(t)<\pid_z(t)<\pid_{z^+}(t)$, since we are on $E_R$. 
Then, there exists $k\in\{0,\dots,L_\eps-1\}$ such that $t\in[t_{k+1}^\eps,t_{k}^\eps]$ and, consequently, 
a unique element $\sz\in\Xi_R(t_{k+1},t_{k+2})$, necessarily belonging to $\ST^{\da,\,(r)}_\infty$, 
such that, by the coalescing property, $\rho((t,\pid_z),t_{k+2}^\eps) = \sz$. Since $d^\da((t,\pid_z),\sz)\leq\eps/2<\eps$,
\eqref{e:claimBoundNeps} follows.
Therefore, we obtain 
\begin{equs}\label{e:CardNet}
\Prob\big(\cN_{d}&(\ST^{\da,\,(r)}_\infty,\eps) \ge N\big) \le 
\Prob\Big(E_R^c\cup \Big\{\sum_{k=1}^{L_\eps} \eta^R(t_k^\eps,t_{k+1}^\eps)> N\Big\}\Big)\\
&\leq \Prob(E_R^c)+ N^{-1} \,\sum_{k=1}^{L_\eps} \E [\eta^R(t_k^\eps,t_{k+1}^\eps)] \leq C_1 \frac{\sqrt{r}}{R} e^{-\frac{R^2}{2r}} + C_2 \frac{L_\eps R}{\sqrt{\eps}N}\;,
\end{equs}
for some constant $C_2 > 0$,
where the last inequality follows from~\cite[Proposition 6.2.7]{SSS}. 
Setting $N = K\eps^{-3/2}$, it suffices to choose  $R = \sqrt K$ to obtain~\eqref{e:TB}.
\end{proof}

We now focus on the H\"older continuity of the map $M_\infty^{\da}\restr \ST^{\da,\,(r)}_\infty$. 
In this case, it suffices to show that 
\begin{equation}\label{e:EC}
\limsup_{\eps\to0}\Prob\big(\sup \{\|M_\infty^{\da}(\sz)-M_\infty^{\da}(\sz')\|\,:\,\sz, \sz'\in\ST^{\da,\,(r)}_\infty
\text{ s.t. }d^\da(\sz, \sz')\leq\eps\}\leq\eps^{\alpha}\big)=1\;,
\end{equation}
for any fixed $\alpha<1/2$ (then taking at most an even smaller $\alpha$ one deduces the {\it little} H\"older property). 
We claim that, on the event $E_R$, $M_\infty^{\da}\restr \ST^{\da,\,(r)}_\infty$ is $\alpha$-H\"older continuous 
provided that 
the paths $\pid_z$, $z\in\cD$, restricted to the box $\Lambda_{r,R}\eqdef [-r,r]\times[-\tilde R,\tilde R]$ 
satisfy a suitable equi-H\"older continuity condition. 
The latter can be stated in terms of a modulus of continuity of the form (see also the proof of~\cite[Theorem 6.2.3]{SSS})
\begin{equ}
\Psi_{\ST^\da_\infty,R,r}(\eps)\eqdef \sup\{|\pid_z(s)-\pid_z(t)|\,:\,z\in\cD\,,M_\infty^{\da}(s,\pid_z)\in\Lambda_{r,R}\,,\,t\in[s,s+\eps]\}
\end{equ} 
for $\eps\in(0,1)$. Indeed, on $E_R$, assume $\Psi_{\ST^\da_\infty,R,r}(\eps)\leq \eps^\alpha/2$ and let 
$(s,\pid_z),\,(t,\pid_{z'})\in\ST^{\da,\,(r)}_\infty$ be such that $d^\da((s,\pid_z),(t,\pid_{z'}))\leq\eps$. 
Then, necessarily, $M_\infty^{\da}(s,\pid_z)$, $M_\infty^{\da}(s,\pid_{z'})$ $\in \Lambda_{r,R}$ and both 
$s-\tau^\da_{s,t}(\pid_z,\pid_{z'})$ and $t-\tau^\da_{s,t}(\pid_z,\pid_{z'})\leq \eps$. 
Therefore, by the coalescing property,
\begin{equs}
\|M_\infty^{\da}&(s,\pid_z)-M_\infty^{\da}(s,\pid_{z'})\|= |\pid_z(s)-\pid_{z'}(t)|\vee|t-s|\\
&\leq\Big(|\pid_z(\tau^\da_{s,t}(\pid_z,\pid_{z'}))-\pid_z(s)|+|\pid_{z'}(\tau^\da_{s,t}(\pid_z,\pid_{z'}))-\pid_{z'}(t)|\Big)\vee|t-s|\leq\eps^\alpha\,.
\end{equs} 
The following lemma concludes the proof of~\eqref{e:EC}. 

\begin{lemma}\label{l:ModCont}
There exists a constant $C=C(r)>0$ such that 
\begin{equ}[e:ProbMC]
\P\Big(\Psi_{\ST^\da_\infty,R,r}(\eps)>\frac{\eps^\alpha}{2}\Big)\leq \frac{C}{\eps^{2\alpha+\frac{1}{2}}}e^{-\eps^{2\alpha-1}}
\end{equ}
uniformly over $\eps\in(0,1]$. 
\end{lemma}
\begin{proof}
We proceed similarly to what was done in the proofs of~\cite[Proposition B.1 and B.3]{FINR} and 
in~\cite[Theorem 6.2.3]{SSS}. 
We introduce the grid $G^{\eps}_{r,R}\eqdef\{(n\,\eps,m\,\eps^\alpha/4)\,:\,m,n\in\Zint\}\cap\Lambda_{r,R}$. 
For any $z_0=(t_0,x_0)\in G^{\eps}_{r,R}$, we define the rectangles 
$R_{z_0}^-=[t_0+\eps/4,t_0+\eps/2]\times[x_0-7\eps^\alpha/32,x_0-5\eps^\alpha/32]$ and 
$R_{z_0}^+=[t_0+\eps/4,t_0-\eps/2]\times[x_0+5\eps^\alpha/32,x_0+7\eps^\alpha/32]$ and 
consider two points $z_0^\pm\in\cD\cap R_{z_0}^\pm$. 
Let $\pid_{z_0^\pm}$ be the backward Brownian motions starting from $z_0^\pm$ respectively. 

Assume now that $\Psi_{\ST^\da_\infty,R,r}(\eps)>\eps^\alpha/2$, 
then there exists a path $\pid_z$, $z\in\cD$ such that $|\pid_z(s)-\pid_z(t)|>\eps^\alpha/2$, 
for some $s$ for which $(\pid_z(s),s)\in\Lambda_{r,R}$ and $t\in[s-\eps,s]$. 
Then, pick the closest point $z_0=(t_0,x_0)\in G_{R,r}^{\eps}$, 
for which $|\pid_z(s)-x_0|\leq\eps^\alpha/8$ and $|s-t_0|\leq\eps$. 
By the coalescing property of our paths, it follows that necessarily one between $\pid_{z_0^\pm}$ must be such that
\begin{equ}
\sup_{h\in[0,2\eps]}|\pid_{z_0^\pm}(t_0-h)-x_0|\geq\eps^\alpha/32\,.
\end{equ}
Let $E_{R,r}^{\eps}(z_0)\eqdef \{\sup_{h\in[0,2\eps]}|\pid_{z_0^\pm}(t_0-h)-x_0|\leq\eps^\alpha/32\}$, then, again by the reflection principle we have
\begin{equs}
\Prob(E_R^c\cup(E_{R,r}^{\eps})^c)&\leq \Prob(E_R^c)+\sum_{z_0\in G_{R,r}^{\eps}}\Prob((E_{R,r}^{\eps}(z_0))^c)\\
&\leq C_1 \frac{\sqrt{2r +1}}{R} e^{-\frac{R^2}{2r+1}}+ C_3 \frac{Rr}{\eps^{2\alpha-1/2}} e^{-\eps^{2\alpha-1}}
\end{equs}
and upon taking $R=\eps^{-1}$,~\eqref{e:ProbMC} follows. 
\end{proof}

We now want to show properness of $M_\infty^{\da}$, which is a direct consequence of the following lemma. 

\begin{lemma}\label{l:Prop}
There exists a constant $c>0$ independent of $r$ such that for any $K>0$ sufficiently large
\begin{equ}[e:Prop]
\Prob\Big(b_{\zeta^\da_\infty}(r)\geq K\Big)\leq c\frac{r}{\sqrt[4]{K}}
\end{equ}
where $b_{\zeta^\da_\infty}$ is the properness map given in~\eqref{def:PropMap}.  
\end{lemma}
\begin{proof}
Let $R>1$ and consider two squares $\tilde Q_R^\pm$ of side $1$ centred at $(r+1,\pm(r+R+1))$. 
Let $\tilde z^\pm=(\tilde t^\pm, \tilde x^\pm)$ be two points respectively belonging to $\tilde Q_R^\pm\cap\cD$, 
and without loss of generality, assume $\tilde t^+=\tilde t^-=\tilde t$.
Let $\pi^\da_{\tilde z^\pm}$ be the two paths starting from $\tilde z^\pm$. 
For $K>4r$, we introduce the event
\begin{equ}\label{e:EvProp}
\tilde E_{R,r}^K\eqdef\Big\{\sup_{-r\leq s\leq \tilde t} \pi^\da_{\tilde z^-}(s)<-r\,,\,\inf_{-r\leq s\leq \tilde t} \pi^\da_{\tilde z^+}(s)>r\,,\,\tau^\da(\pi^\da_{\tilde z^+},\pi^\da_{\tilde z^-})>r-\frac{K}{2}\Big\}\,,
\end{equ}
where in $\tau^\da$ we omitted the subscript since we imposed $\tilde t^+=\tilde t^-$. 
Notice that on $\tilde E_{R,r}^K$, for any point  $(s,\pi^\da_z)\in\ST_\infty^\da$ such that 
$M_\infty^{\da}(s,\pi^\da_z)\in \Lambda_r$, by the coalescing property, 
the trajectories of both $\pi^\da_0$ and $\pi^\da_z$, after time $s$ must be confined between those of 
$\pi^\da_{\tilde z^+}$ and $\pi^\da_{\tilde z^-}$. 
Therefore, $d^\da((\pi^\da_z,s),(\pi^\da_0,0))\leq 2r-2\tau^\da(\pi^\da_{\tilde z^+},\pi^\da_{\tilde z^-})<K$, 
so that one has
\begin{equ}
\Prob\Big(b_{\zeta^\da_\infty}(r)\geq K\Big) \le \P((\tilde E_{R,r}^K)^c)\;,
\end{equ}
independently of the choice of $R$. The reflection principle 
yields a bound of the type
\begin{equ}[e:PropQ]
\P((\tilde E_{R,r}^K)^c)
\le C \frac{\sqrt{r}}{R}e^{-\frac{R^2}{r}} + C \frac{R+r+1}{\sqrt{K}}\;,
\end{equ}
for some constant $C>0$, and \eqref{e:Prop} follows at once, upon choosing $R\eqdef \sqrt[4]{K}$. 
\end{proof}

{Since $\cD$, which is dense in $\R^2$, is contained in $M^{\da,\cD}(\ST^\da)$, and $M^{\da}$ is proper, }
surjectivity follows. At last,~\ref{i:TreeCond} {in Definition~\ref{d:TreeCond}} 
is a direct consequence of the fact that, almost surely, it holds for 
$\zeta^\da_\infty(\cD)$ by construction and Lemma~\ref{l:TreeCond}. 
\end{proof}

\begin{remark}\label{r:Hom}
Almost surely, the map $M^{\da,\cD}$ is continuous and proper on $\ST^\da_\infty(\cD)$. 
Moreover, it is bijective on its image (endowed with the usual Euclidean topology) by construction, 
hence $M^{\da,\cD}:\ST^\da_\infty(\cD)\to M^{\da,\cD}(\ST^\da_\infty(\cD))$ is a homeomorphism. 
\end{remark}

\begin{remark}\label{rem:p:PeriodicBW}
The previous proposition remains true if instead of the sequence $\zeta^\da_n(\cD)$
we take $\zeta^{\per,\da}_n(\cD^\per)$, $\cD^\per$ being a countable dense set of $\R\times\T$. 
The proof is actually simpler since it is not necessary to introduce the event in~\eqref{e:ER}. 
In the periodic setting, the convergence happens in $\T^\alpha_{\Sp,\per}$, the limit
$\zeta^{\per,\da}(\cD^\per)=(\ST^{\per,\da}_n(\cD),\ast^\da,d^\da,M^{\per,\da,\cD^\per})$ belongs to $\Ch^\alpha_{\Sp,\per}$ and 
$M^{\per,\da,\cD^\per}(\ST^{\per,\da}_n(\cD))=\R\times\T$. 
\end{remark}

The next theorem introduces and uniquely characterises the law on the space of directed trees 
of the random variable which in the sequel we will refer to as the {\it Brownian Web tree}. 
%
%

\begin{theorem}\label{thm:BW}
Let $\alpha<\tfrac12$. There exists a $\Ch^\alpha_\Sp$-valued random variable 
$\zeta^\da_\bw=(\ST^\da_\bw,\ast^\da_\bw,d^\da_\bw,M^\da_\bw)$ with radial map $\rho^\da$, such that
\begin{enumerate}
\item for any deterministic point $w=(s,y)\in\R^2$ there exists almost surely a unique point 
$\sw_w\in\ST^\da_\bw$ such that $M^{\da}_\bw(\sw_w)=w$, 
\item for any deterministic $n\in\N$ and $w_1=(s_1,y_1),\dots,w_n=(s_n,y_n)\in\R^2$, the joint distribution 
of $( M^{\da}_{\bw,x}(\rho^\da(\sw_{w_i},\cdot)))_{i=1,\dots,n}$ 
is that of $n$ coalescing backward Brownian motions starting at $w_1,\dots, w_n$, 
\item for any deterministic countable dense set $\cD$ such that $0\in\cD$, 
let $\tilde\ast^\da$ be the point determined in 1. associated to $0$. 
Define $\tilde\zeta_\infty^\da(\cD)=(\tilde\ST_\infty^\da(\cD),\tilde\ast^\da,d^\da,\tilde M_\infty^{\,\da,\cD})$ as
\begin{equation}\label{e:Coupling}
\begin{split}
\tilde\ST_\infty^\da(\cD)&\eqdef\{\rho^\da(\sw_w,t)\,:\,w=(s,y)\in\cD\,,\,t\leq s\}\\
\tilde M_\infty^{\,\da,\cD}(\rho^\da(\sw_w,t))&\eqdef M_\bw(\rho^\da(\sw_w,t))
\end{split}
\end{equation}
and $d^\da$ to be the ancestral metric in~\eqref{e:Skeleton}. 
Let $\tilde\ST^\da(\cD)$ be the completion of $\tilde\ST_\infty^\da(\cD)$ under $d^\da$,  
$\tilde M^{\,\da,\cD}$ be the unique little $\alpha$-H\"older continuous extension of $\tilde M_\infty^{\,\da,\cD}$ and 
$\tilde\zeta^\da(\cD)\eqdef(\tilde\ST^\da(\cD),\ast^\da,d^\da, \tilde M^{\,\da,\cD})$. Then,  
$\tilde \zeta^\da(\cD)\eqlaw\zeta^\da_\bw$.
\end{enumerate}
{The law of $\zeta^\da_\bw$ is uniquely characterised by the properties above} and 
almost surely, $\zeta^\da_\bw$ satisfies~\eqref{e:nMEC} for all $\theta>3/2$, 
$M^\da_\bw$ is surjective and~\ref{i:TreeCond} in Definition~\ref{d:TreeCond} holds. 
\end{theorem}

\begin{proof}
Let $\cD$ be a countable dense set of $\R^2$ containing $0$. 
Thanks to Proposition~\ref{p:BW}, for $\alpha<\tfrac12$, $\zeta^\da(\cD)$ almost surely belongs to $\Ch^\alpha_\Sp$ 
so, if we show that it satisfies properties 1.-3. above, then the existence part of the statement 
follows. 

In order to prove 1., let $w=(s,y)\notin\cD$, and consider two sequences of points $z_n^\pm=(t_n^\pm,x_n^\pm)\in\cD$ 
for which there exist two constants $c^\pm>0$ such that 
\begin{equs}
y-\frac{c_1}{n^2}<x_n^-<\,&y<x_n^+<y-\frac{c_2}{n^2}\\
s<t_n^-<s+|x_n^-|^3\quad&\text{and }\,\, s<t_n^+<s+|x_n^+|^3\,.
\end{equs}
For every $n\in\N$, let $\pid_{z_n^\pm}$ be the two backward Brownian motions starting at $z_n^\pm$ respectively. 
Denote by $\tau_n=\tau^\da_{t_n^-,t_n^+}(\pid_{z_n^-},\pid_{z_n^+})$ and $X_n=\pid_{z_n^-}(\tau_n)=\pid_{z_n^+}(\tau_n)$ 
the time and spatial point at which they coalesce. 
Define $\Delta_n$ as the triangular region in $\R^2$ with vertices $z_n^\pm$ and $(\tau_n,X_n)$, 
the base being given by the segment joining $z_n^-$ and $z_n^+$, while the sides by the paths
$(r,\pid_{z_n^-}(r))_{t_n^-\geq r\geq \tau_n}$, $(r,\pid_{z_n^+}(r))_{t_n^+\geq r\geq \tau_n}$.

In the proof of~\cite[Proposition 3.1]{FINRa} the authors show that the event 
\begin{equ}
E_n\eqdef\Big\{\pid_{z_n^-}(s)<y<\pid_{z_n^+}(s)\,,\,\tau_n\geq s-1/n\,,\,|X_n-y|<n^{-1/4}\Big\}
\end{equ}
occurs infinitely often. Hence, for any sequence $z_m=(t_m,x_m)\in\cD$ converging to $w$, 
for all $n\in\N$ large enough there exists $m_n\in\N$ such that, for all $m\geq m_n$, $z_m\in \Delta_n$.  
The coalescing property then implies that for every $m_1,\,m_2\geq m_n$, 
\begin{equ}
d^\da((t_{m_1},\pid_{z_{m_1}}), (t_{m_2},\pid_{z_{m_2}}))\leq (t_{m_1}+t_{m_2})-2\tau_n\leq (t_{m_1}-s)+(t_{m_2}-s)+2/n\,.
\end{equ} 
In other words, for any $z_m=(t_m,x_m)\in\cD$ converging to $w$, $(t_m,\pid_{z_m})_{m\in\N}$ is Cauchy in 
$\ST_\infty^\da(\cD)$ therefore it converges in $\ST^\da(\cD)$ to a unique point $\sw_w$ which, by continuity of 
$M^{\da,\cD}$, is necessarily such that $M^\cD(\sw_w)=w$. 

Moreover, by construction we know that $\rho^\da((t_m,\pid_{z_m}),t)=(\pid_{z_m},t)$ for all $t\leq t_m$ and, 
since at $\tau_n$ the ray starting at $\sw_w$ must have coalesced with that starting at $(\pid_{z_n},t_n)$, 
we must have $\rho^\da(\sw,t)=\rho^\da((\pid_{z_n},t_n),t)$ for any $t\leq\tau_n$. 
Hence, the sequence of paths 
$(-\infty,t_m]\ni t\mapsto M^{\da,\cD}_x(\rho^\da((t_m,\pid_{z_m}),t))=M^{\da,\cD}_x(t,\pid_{z_m})$ converges to 
$(-\infty,s]\ni t\mapsto M^{\da,\cD}_x(\rho^\da(\sw_w,t))$ in $\Pi$, where $\Pi$ is given as in Section~\ref{s:Topo}. 
Since $( M^{\da,\cD}_x(\rho^\da((t_m,\pid_{z_m}),t)))_{t\leq t_m}$ is distributed according to a backward Brownian motion 
starting at $z_m$, $(M^{\da,\cD}_x(\rho^\da(\sw_w,t))_{t\leq s}$ is itself distributed according 
to a backward Brownian motion, but starting at $w$. 

For 2., let $w_1,\dots,w_n$ be $n$ deterministic points in $\R^2$ and $\sw_{w_1},\dots,\sw_{w_n}$ 
be the points in $\ST^\da(\cD)$ determined by applying 1.
Thanks to the last part of the proof of 1., if $z_{m_i}=(t_{m_i},x_{m_i})$ is a sequence in $\cD$ converging to 
$w_i$, $i\in[n]$, then the paths 
$(M^{\da,\cD}_x(\cdot,\pid_{z_{m_i}}))_{i\in[n]}$ 
converge to $(M^{\da,\cD}_x(\rho^\da(\sw_{w_i},\cdot))_{i\in[n]}$ in $\Pi^n$. 
Since the first are distributed as coalescing backward Brownian motions starting from $(z_{m_1},\dots,z_{m_n})$, 
it is easy to see that the limit will be also distributed according to coalescing Brownian motions starting 
from $(w_1,\dots,w_n)$.

We now prove 3., for which we proceed as follows. Let $\cD'$ be another countable dense set in 
$\R^2$ containing $(0,0)$. We want to determine a suitable coupling of $\zeta^\da(\cD)$ and $\tilde\zeta^\da(\cD')$ 
under which they are almost surely equal. 
We first construct $\zeta^\da(\cD')$ as in~\eqref{e:Skeleton} and Proposition~\ref{p:BW}.  Then, we build 
$\tilde\zeta_\infty^\da(\cD')=(\tilde\ST_\infty^\da(\cD'),\tilde\ast^\da,d^\da,\tilde M_\infty^{\,\da,\cD'})$ inside $\zeta^\da(\cD)$ 
according to~\eqref{e:Coupling}. 
Obviously $\zeta^\da(\cD')$ and $\tilde\zeta^\da(\cD')$ are equal in distribution, 
and the latter is such that $\tilde\ST^\da(\cD')\subseteq \ST^\da(\cD)$, $\tilde\ast^\da=\ast^\da$ and 
$M^{\da,\cD}\restr \tilde\ST^\da(\cD')= \tilde M^{\,\da,\cD'}$. 
Therefore, if we are able to show that $\tilde\ST^\da(\cD')$ coincides with $\ST^\da(\cD)$, we are done. 
We claim that if $z\in\cD$ and {$\sz_z\in\ST^\da(\cD)$ is the unique point such that $M^{\da,\cD}(\sz_z)=z$ 
(which holds by 1.) then $\sz_z$} also belongs to $\tilde\ST^\da(\cD')$. 
Notice that if this is the case then 
for {\it all} $z\in\cD$, $\sz_z\in\tilde\ST^\da(\cD')$. 
It follows that all the rays starting from these $\sz_z$'s are contained in $\tilde\ST^\da(\cD')$ and 
hence also the closure of their union, 
which by construction is $\ST^\da(\cD)$. {Hence, the proof of 3. is complete 
once we show the claim}. 

Let $z\in\cD$ and $\sz_z\in\ST^\da(\cD)$ be as above. 
Let $w_n=(s_n,y_n)$ be a sequence in $\cD'$ converging to $z$ in $\R^2$. 
By 1., we know that for all $n$ there exists a unique point $\sz_{w_n}$ in $\ST^\da(\cD)$ such that 
$M^{\da,\cD}(\sw_n)=w_n$ and since $\tilde\ST_\infty^\da(\cD')\subseteq\ST^\da(\cD)$ and, by construction, there is 
a unique point in $\tilde\ST_\infty^\da(\cD')$ whose image is $w_n$, it follows that $\sz_{w_n}\in\tilde\ST_\infty^\da(\cD')$. 
Now, the map $M^{\da,\cD}$ is proper and the sequence $\{w_n\}_n$ is bounded, 
therefore the sequence $\{\sz_{w_n}\}_n$ is also bounded and it converges along subsequences. 
Fix one of these subsequences (that, with a slight abuse of notation, will still be indexed by $n$) and 
notice that by continuity of $M^{\da,\cD}$ and uniqueness of $\sz_z$, 
we necessarily have that $(\sz_{w_n})_n$ converges to $\sz_z$ in $\ST^\da(\cD)$. 
Now, since $\{\sz_{w_n}\}_n$ converges, it is Cauchy and since it is contained in $\tilde\ST_\infty^\da(\cD')$,
the limit must belong to $\tilde\ST^\da(\cD')$.

It remains to argue uniqueness and the properties of the limit. 
{Uniqueness immediately follows since conditions 1-3 above imply conditions 1-4 in 
Theorem~\ref{thm:Main}. On the other hand, we have just shown that $\zeta^\da(\cD)$ 
satisfies 1-3 and, by Proposition~\ref{p:BW} 
also the other claimed properties, so that the proof of the statement is concluded. }
\end{proof}

\begin{remark}\label{rem:t:PeriodicBW}
The theorem above remains true upon replacing conditions $1.$-$3.$ with $1_\per.$, $2_\per.$ and $3_\per.$, 
obtained from the former by adding the word ``periodic'' before any instance of ``Brownian motion'', and taking 
the periodic version of all objects and spaces in the statement. 
\end{remark}

\begin{definition}\label{def:BW}
Let $\alpha<\tfrac12$. We define  {\it backward Brownian Web 
Tree} and {\it periodic backward Brownian Web tree}, the $\Ch^\alpha_\Sp$ and $\Ch^\alpha_{\Sp,\per}$ 
random variables  $\zeta^\da_\bw=(\ST^\da_\bw,\ast^\da_\bw,d^\da_\bw,M^\da_\bw)$ and 
$\zeta^{\per,\da}_\bw=(\ST_\bw^{\per,\da},\ast_\bw^{\per,\da}, d_\bw^{\per,\da},M_\bw^{\per,\da})$ 
whose distributions is uniquely characterised by properties $1.$-$3.$ in Theorem~\ref{thm:BW} and 
$1_\per.$, $2_\per.$ and $3_\per.$ in Remark~\ref{rem:t:PeriodicBW}. 
We will denote their respective laws by $\Theta^\da_\bw(\dd \zeta)$ 
and $\Theta_\bw^{\per,\da}(\dd\zeta)$.
\end{definition}

As a first property of the Brownian Web tree, which can be deduced by Theorem~\ref{thm:BW} 
and the results stated therein, we determine its Minkowski, also known as box-covering, dimension. 
Recall that the box-covering dimension of a (compact) metric space $(T,d)$ is given by 
\begin{equ}[e:BoxDim]
\dim_{\mathrm{box}} (T)\eqdef \lim_{\eps\to 0} \frac{\log\cN_d(T,\eps)}{\log \eps^{-1}}
\end{equ}
when this limit exists. 

\begin{corollary}\label{cor:BCdim}
Let $\zeta^\da_\bw=(\ST^\da_\bw,\ast^\da_\bw,d^\da_\bw,M^\da_\bw)$ be 
the Brownian Web tree of Definition~\ref{def:BW}. Then, almost surely $\ST^\da_\bw$ 
has box-covering dimension $\tfrac32$. 
\end{corollary}

\begin{remark}
In~\cite{DEFKZ}, the authors consider a family of coalescing particles starting from every point on the circle 
and, in Theorem 11.2 of that article, show that, as a metric space, it has Hausdorff dimension equal to $1/2$. 
Translated to our setting this says 
that the periodic Brownian Web at a fixed time slice has Hausdorff dimension equal to $1/2$. 
\end{remark}

\begin{proof}
According to~\eqref{e:BoxDim}, it suffices to determine almost sure upper and lower bounds for 
$\cN_{d^\da_\bw}(\ST^{\da,\,(r)}_\bw,\eps)$ of the same order, for all $r>0$. 
Now, the upper bound follows by the fact that, by Theorem~\ref{thm:BW}, 
almost surely $\ST^{\da}_\bw$ satisfes~\eqref{e:nMEC} {for all $\theta>3/2$}. 
{For the lower bound, we need to show that almost surely for all $r>0$, $\kappa>0$ 
there exists a random constant $C=C(r,\kappa)>0$ such that $\cN_{d^\da_\bw}(\ST^{\da,\,(r)}_\bw,\eps)\geq C\eps^{\kappa-\frac32}$. 
This in turn follows at once, provided we prove that for all $\delta>0$, $r>0$ and $\kappa>0$, 
there exists $K>0$ such that 
\begin{equ}[e:LB]
\P(\cN_{d^\da_\bw}(\ST^{\da,\,(r)}_\bw,\eps)\leq K\eps^{\kappa-\frac32})\leq \delta\,,\qquad\text{for all $\eps\in(0,1]$.}
\end{equ}
We now fix $\delta,r,\kappa$. 
To control the probability on the left-hand side of~\eqref{e:LB}, by the reflection principle, we know that 
we can find $t<0$ and $x>0$ such that the event 
\begin{equ}
E_{t,x}\eqdef \Big\{\sup_{s\in[t,0]}|\pi^\da_{(0,0)}(s)|<\tfrac{x}{4}\,,\,\sup_{s\in[t,0]}|\pi^\da_{(0,x)}(s)-x|<\tfrac{x}{4}\,,\,\tau_{0,0}^\da(\pi^\da_{(0,0)}, \pi^\da_{(0,x)})\leq \tfrac{r}{2}\Big\}
\end{equ}
has probability bigger than $1-\delta/2$. Notice that, on $E_{t,x}$, 
$(M^\da_\bw)^{-1}([t,0]\times [ x/4,  3x/4])\subset \ST^{\da,(r)}_\bw$. 
Define $L_\eps=\lceil t/\eps\rceil + 1$ and $t_k^\eps\eqdef -k\eps$, $k=0,\dots, L_\eps-1$. 
Therefore, arguing as in the proof of~\eqref{e:claimBoundNeps} we have 
\begin{equ}
\cN_{d^\da_\bw}(\ST^{\da,\,(r)}_\bw,\eps)\geq \sum_{k=0}^{L_\eps-1}\eta^x(t_k^\eps,t_{k+1}^\eps)
\end{equ}
where, for $b<a$, $\eta^x(a,b)$ is the cardinality of $\Xi_{x}(t_0,t_1)$  defined 
as in~\eqref{e:CPS}, but with the interval $[-\tilde R,\tilde R]$ replaced by 
$[ x/4,  3x/4]$. 
Then, we have 
\begin{equs}
\P(\cN_{d^\da_\bw}(\ST^{\da,\,(r)}_\bw,\eps)\leq C\eps^{\kappa-\frac32})&\leq \frac\delta2+\P(E_{t,x}\cap\{\cN_{d^\da_\bw}(\ST^{\da,\,(r)}_\bw,\eps)\leq C\eps^{\kappa-\frac32}\})\\
&\leq \frac\delta2+\P\Big(\sum_{k=0}^{L_\eps-1}\eta^x(t_k^\eps,t_{k+1}^\eps)\leq C\eps^{\kappa-\frac32}\Big)\,.
\end{equs}
For the right-hand side, we notice that for any $\tilde K>0$ we have
\begin{equs}\label{e:Variance}
\P&\Big(\sum_{k=0}^{L_\eps-1}\eta^x(t_k^\eps,t_{k+1}^\eps)\leq C\eps^{\kappa-\frac32}\Big)\\
&=\P\Big(\sum_{k=0}^{L_\eps-1}\eta^x(t_k^\eps,t_{k+1}^\eps)\leq C\eps^{\kappa-\frac32}\,,\,\Big|\sum_{k=0}^{L_\eps-1}(\eta^x(t_k^\eps,t_{k+1}^\eps)-\E[\eta^x(t_k^\eps,t_{k+1}^\eps)])\Big|>\tilde K\Big)\\
&\quad+\P\Big(\sum_{k=0}^{L_\eps-1}\eta^x(t_k^\eps,t_{k+1}^\eps)\leq C\eps^{\kappa-\frac32}\,,\,\Big|\sum_{k=0}^{L_\eps-1}(\eta^x(t_k^\eps,t_{k+1}^\eps)-\E[\eta^x(t_k^\eps,t_{k+1}^\eps)])\Big|\leq\tilde K\Big)\,.
\end{equs}
Now, the first summand is bounded above by 
\begin{equs}[e:1]
\P\Big(\Big|\sum_{k=0}^{L_\eps-1}&(\eta^x(t_k^\eps,t_{k+1}^\eps)-\E[\eta^x(t_k^\eps,t_{k+1}^\eps)])\Big|>\tilde K\Big)\leq \frac{{\rm Var}\Big(\sum_{k=0}^{L_\eps-1}\eta^x(t_k^\eps,t_{k+1}^\eps)\Big)}{\tilde K^2}\\
&\lesssim \frac{\eps^{-1}}{\tilde K^2}\sum_{k=0}^{L_\eps-1}{\rm Var}(\eta^x(t_k^\eps,t_{k+1}^\eps))\leq 2\frac{\eps^{-1}}{\tilde K^2}\sum_{k=0}^{L_\eps-1}\E[\eta^x(t_k^\eps,t_{k+1}^\eps)]\lesssim  \frac{\eps^{-5/2}}{\tilde K^2}
\end{equs}
where in the penultimate step we used that $\eta_x$ is negatively 
correlated \cite[Lemma~C.4]{GSW}, 
so that~\cite[Lemma~C.5]{GSW} implies that its variance is bounded above by twice its mean, 
and in the last step we exploited~\cite[Proposition 6.2.7]{SSS}. 
At this point, suitably choosing $\tilde K =O(\delta^{-1/2}\eps^{-5/4})$ 
we see that the right-hand side of~\eqref{e:1} is bounded above by $\delta/2$ while the second summand 
in~\eqref{e:Variance} vanishes. 

Collecting the estimates obtained so far,~\eqref{e:LB}, and consequently the lower bound, follow at once. }
\end{proof}

\begin{remark}\label{rem:BCdim}
The previous corollary shows in particular that the law of the Brownian Web trees 
on the space of $\R$-trees is, as expected, singular with respect to that of 
the scaling limit of the Uniform Spanning Tree in two dimensions. 
Indeed, the latter has Hausdorff dimension $5/8$~\cite{BCK} and 
the Hausdorff dimension is always greater than or equal to the box-counting one (see e.g.~\cite[Chapter 1]{E}). 
\end{remark}

In the following Corollary, we establish the relation between the Brownian Web Tree of 
Definition~\ref{thm:BW} and the Brownian Web constructed in~\cite{FINR}, which 
is a simple consequence of Theorem~\ref{thm:BW} and the results in Section~\ref{s:Topo}. 

\begin{corollary}\label{cor:Topologies}
Let $\zeta^\da_\bw$ and $\zeta^{\per,\da}_\bw$ be the backward and backward periodic Brownian Web trees 
of Theorem~\ref{thm:BW} and Remark~\ref{rem:t:PeriodicBW}, and $K$ be the map defined in~\eqref{e:CompZeta}. 
Then, $K(\zeta^\da_\bw)$ is a backward Brownian Web 
according to~\cite[Theorem 2.1]{FINR} and $K(\zeta^\da_{\bw,\per})$ is a backward cylindric Brownian Web 
according to~\cite[Theorem 2.3]{CMT}. 
\end{corollary}
\begin{proof}
To prove the statement it suffices to verify that $K(\zeta^\da_\bw)$ and $K(\zeta^\da_{\bw,\per})$ satisfy 
$(o)$, $(i)$ and $(ii)$ in~\cite[Theorem 2.1]{FINR} and~\cite[Theorem 2.3]{CMT}, respectively. 
This is in turn an immediate consequence of the definition of  $K$ and properties 
1.-3. in Theorem~\ref{thm:BW} and $1_\per.$-$3_\per.$ in Remark~\ref{rem:t:PeriodicBW}. 
\end{proof}

\subsection{A convergence criterion to the Brownian Web tree}\label{sec:ConvBW}

In this section, we want to derive a criterion that allows to conclude that the 
limit law for tight sequences of directed spatial $\R$-trees is $\Theta^\da_\bw$.

\begin{theorem}\label{thm:ConvBW}
Let $\alpha\in(0,1)$ and $\{\zeta_n\}_n$ be a tight sequence of  random 
variables in $\Ch^\alpha_\Sp$ with laws  $\Theta_n$ and assume that the following holds.
\begin{enumerate}[label=(\Roman*)]
\item\label{i:Conv1} For any $k\in\N$ and (deterministic) $z^1,\dots,z^k\in\R^2$ there exist sequences 
$\sz^i_n\in\ST_n$, $i=1,\dots, k$ such that $\lim_{n \to \infty} M_{n}(\sz_n^i)=z^i$ almost surely and such that  
$(M_n(\rho_n(\sz_n^i,\cdot)))_i$ converges in law to $k$ coalescing backward Brownian motions.
\item\label{i:Conv2} For every $h>0$ \begin{equation}\label{e:ConvCond}
\frac{1}{\eps}\limsup_{n\to\infty}\sup_{(t,x)\in\R^2} \Theta_n\left(\#\{\rho(\sw,t-h)\,:\,\sw\in M^{-1}(I_{t,x,\eps})\}\geq 3\right)\overset{\eps\to 0}{\longrightarrow} 0
\end{equation}
where $I_{t,x,\eps}\eqdef \{t\}\times(x-\eps,x+\eps)$. 
\end{enumerate}
Then $\Theta_n$ converges weakly to $\Theta^\da_\bw$. 
\end{theorem}

\begin{remark}
In view of Corollary~\ref{cor:Topologies}, the Brownian Web tree and the Brownian Web are strictly connected 
so that it is not surprising that the convergence criterion stated above is extremely similar to~\cite[Theorem 6.6.5]{SSS}. 
As a matter of fact, requiring the sequence $\zeta_n$ to be made of directed trees allows us to talk about paths, 
while the fact that we are dealing with monotone trees enforces 
the non-crossing condition. 
That said, even though Proposition~\ref{p:MapTopo} guarantees continuity of the map $K$ assigning to 
any directed tree 
a compact subset of $\Pi$, as highlighted in Remark~\ref{rem:Inj}, $K$ is not injective in any open subset of 
$\Ch^\alpha_\Sp$ but only on $\Ch^\alpha_\Sp(\ft)$ and 
the inverse map is not continuous, even when restricted to $\Ch^\alpha_\Sp(\ft)$. 
This means that we cannot infer convergence in $\Ch_\Sp^\alpha$ from \cite{SSS} 
(see Remark~\ref{rem:exam} for an example illustrating this). 
\end{remark}

\begin{proof}
Let $K$ be the map defined in~\eqref{e:CompZeta}. { At first, we want to show that the sequence
$\{K(\zeta_n)\}_n$ converges in law to the backward Brownian Web. 
To do so, notice that Proposition~\ref{p:MapTopo} implies that,
since $\{\zeta_n\}_n$ is tight, so is $\{K(\zeta_n)\}_n$. 
Further, as $\zeta_n\in\Ch^\alpha_\Sp$ for every $n$, $K(\zeta_n)$ is supported on compact subsets of $\Pi$ 
formed of non-crossing paths. 
Hence, by~\cite[Theorem 6.6.5]{SSS}, the convergence of $\{K(\zeta_n)\}_n$ to the (backward) Brownian Web 
is guaranteed, provided we show that $(I)$ and $(B2)$ therein hold. 
The former is a direct consequence of~\ref{i:Conv1} in the present statement. 
The latter instead follows by~\ref{i:Conv2} since
\begin{equ}[e:Cardrho]
\#\{\rho_n(\sw,t-h)\,:\,\sw\in M_n^{-1}(I_{t,x,\eps})\} \ge \#\{M_n(\rho_n(\sw,t-h))\,:\,\sw\in M_n^{-1}(I_{t,x,\eps})\}\,,
\end{equ}
and the right hand side above equals $\eta_{K(\zeta_n)}(t,\eps;x-\eps,x+\eps)\geq \eta_{K(\zeta_n)}(t,\eps;x,x+\eps)$, 
where $\eta$ is defined in~\cite[eq. (6.52)]{SSS}. 

Now,} since the sequence $\zeta_n$ is tight by assumption,  it converges along some subsequence. 
Let $\zeta=(\ST,\ast,d,M)$ be a limit point, $\rho$ its radial map and denote by $\Theta$ its law on $\Ch^\alpha_\Sp$. 
$K$ is continuous by Proposition~\ref{p:MapTopo}, hence, $K(\zeta_n)$ converges 
to $K(\zeta)$, which by the above is a backward Brownian Web. 
Further, by Proposition~\ref{p:MapInj} $K$ is injective on $\Ch^\alpha_\Sp(\ft)$, so that 
by Proposition~\ref{p:MapTopo}, it remains to show that $\zeta$ satisfies~\ref{i:TreeCond}. 

{Since $K(\zeta)$ is a Brownian Web, almost surely for every point $z\in\Q^2$ 
there exists a unique path $\pi_z\in K(\zeta)$ starting at $z$ 
and therefore, on an event of probability $1$, 
$K(\zeta)$ can be taken to be the closure of $\{\pi_z\colon z\in\Q^2\}$ in $\Pi$. 
For every $z\in\Q^2$, we choose $\sz_z\in\ST$ such that $M(\sz_z)=z$. This point 
clearly exists but {\it a priori} might not be unique - if this is not the case we only pick one. 
Then, we define $\tilde \ST$ as the closure of $T=\{\rho(\sz_z, s)\colon z=(x,t)\in\Q^2,\,s\leq t\}\subset \ST$ 
with respect to the metric $d$ on $\ST$ and $\tilde M$ as the restriction of $M$ to $\tilde \ST$. 
Note that, by construction $(T,\ast,d,M\restr T)$ satisfies~\ref{i:TreeCond}, 
hence so does $\tilde \zeta\eqdef (\tilde \ST,\ast,d,\tilde M)$ in view of Lemma~\ref{l:TreeCond}. 
Furthermore, the definition of $K$ in~\eqref{e:CompZeta} together with the fact that 
$K(\zeta)$ is the closure of $\{\pi_z\colon z\in\Q^2\}$ in $\Pi$, implies that $K(\tilde\zeta)=K(\zeta)$. 
In particular, the conclusion then follows if we show that $\tilde\ST=\ST$. 

We claim that if $\tilde \ST\subsetneq\ST$, then there exist $a,b,t,h\in\Q$ such that 
\begin{equ}[e:Cardalpha]
\#\{\rho(\sw,t-h)\,:\,\sw\in M^{-1}(I_{t;a,b})\}>\#\{\rho(\sw,t-h)\,:\,\sw\in (\tilde M)^{-1}(I_{t;a,b})\}
\end{equ}
where $I_{t;a,b}\eqdef\{t\}\times[a,b]$. 
Indeed, let $\sz\in\ST$, $z=(s,y)=M(\sz)$ and $r>0$. Since $\zeta\in\Ch^\alpha_\Sp$,  
$M$ is locally $\alpha$-H\"older continuous which implies that there exists $C>0$ such that 
\begin{equ}
|M(\rho(\sz,s-h))-M(\sz)|\leq C h^\alpha\,,\qquad\text{for all $h\leq \bar h$}
\end{equ}
where $\bar h$ is chosen in such a way that $r\geq C\bar h^\alpha$. 
Let $y^+_n,\,y^-_n,\,s_n$ and $h_n$ be sequences in $\Q$ such that $y^\pm_n$ converges to $y\pm r$, 
$s_n$ converges to $s$, 
$h_n$ converges to $0$ and, for all $n$, $y_n^-\leq y-C\bar h^\alpha$, $y_n^+\geq y+C\bar h^\alpha$ and 
$s-\bar h/2\leq s_n-h_n<s$. 
Then, $M(\rho(\sz,s_n))\in\{s_n\}\times[y^-_n,y_n^+]$. 
Now, if for all $a,b,t,h\in\Q$,~\eqref{e:Cardalpha} were an equality, then for all $n$, 
$\rho(\sz,s_n-h_n)\in \{\rho(\sw,s_n-h_n)\,:\,\sw\in (\tilde M)^{-1}(I_{s_n;y_n^-,y_n^+})\}\subset \tilde\ST$. 
But the sequence $\{\rho(\sz,s_n-h_n)\}_n$ is Cauchy in $\tilde\ST$ and since the latter is complete, 
$\rho(\sz,s)\in\tilde\ST$ for every $s$ sufficiently small, so that taking $s$ to $0$, we get that $\sz\in\tilde\ST$. 

The previous claim implies that the probability that $\ST\setminus\tilde\ST\neq\emptyset$ is bounded above by the 
probability that there exist $a,b,t,h\in\Q$ such that~\eqref{e:Cardalpha} holds.
Hence, if we show that for every $a,b,t,h\in\Q$ fixed, the probability of~\eqref{e:Cardalpha} is $0$ we are done. 
Fix $a,b,t,h\in\Q$, $a<b$.
For $N\in\N$, let 
\begin{equ}
x_j^N\eqdef a+j\eps,\,\quad \text{for $j=0,\dots,N$ and } \eps\eqdef\frac{b-a}{N}\,,
\end{equ}
and $z^N_j=(t,x_j)$. By construction, there exist unique points $\sz^N_j\in\tilde\ST$ 
such that $\tilde M(\sz^N_j)=M(\sz^N_j)=z^N_j$ for all $j=0,\dots,N$. 
Hence, 
\begin{equs}
\Theta&\left(\#\{\rho(\sw,t-h)\,:\,\sw\in M^{-1}(I_{t;a,b})\}>\#\{\rho(\sw,t-h)\,:\,\sw\in (\tilde M)^{-1}(I_{t;a,b})\} \right)\\
&\leq \lim_{N\to\infty}\Theta\left(\#\{\rho(\sw,t-h)\,:\,\sw\in M^{-1}(I_{t;a,b})\}>\#\{\rho(\sz^N_j,t-h)\,:\,j=0,\dots,N\}\right)\,.
\end{equs}
Moreover, since as soon as two rays in an $\R$-tree touch, they coalesce (otherwise one could form a cycle), 
$\#\{\rho(\sw,t-h)\,:\,\sw\in M^{-1}(I_{t;a,b})\}>\#\{\rho(\sz^N_j,t-h)\,:\,j=0,\dots,N\}$ 
if and only if there exists $i=1,\dots,N$ such that $\#\{\rho(\sw,t-h)\,:\,\sw\in M^{-1}(I_{t,y_i^N,\eps})\}\geq 3$, 
where, for $i=1,\dots,N$, $y^N_i$ denotes the mid-point of the interval $(x^N_{i-1},x^N_i)$. 
In other words, 
\begin{equs}
\Theta\big(\#\{&\rho(\sw,t-h)\,:\,\sw\in M^{-1}(I_{t,x,\eps})\}>\#\{\rho(\sz^N_j,t-h)\,:\,|j|\le N\}\big)\\
&\leq \sum_{i=1-N}^N\Theta\left(\#\{\rho(\sw,t-h)\,:\,\sw\in M^{-1}(I_{t,y_i^N,{\eps \over N}})\}\geq 3\right)\\
&\lesssim N \sup_{(t,y)\in\R^2}\Theta\left(\#\{\rho(\sw,t-h)\,:\,\sw\in M^{-1}(I_{t,y,{\eps \over N}})\}\geq 3\right)\\
&\lesssim N\limsup_{n\to\infty}\sup_{(t,y)\in\R^2} \Theta_n\left(\#\{\rho(\sw,t-h)\,:\,\sw\in M_n^{-1}(I_{t,y,{\eps \over N}})\}\geq 3\right)\;,
\end{equs}
which converges to $0$ as $N \to \infty$ by~\eqref{e:ConvCond}, and the conclusion follows at once.} 
\end{proof}

\begin{remark}\label{rem:exam}
The first part of the proof above shows that, given that $\{\zeta_n\}_n$ is tight and satisfies
conditions~\ref{i:Conv1} and~\ref{i:Conv2}, the sequence $\{K(\zeta_n)\}_n$ converges to the 
Brownian Web. In light of Propositions~\ref{p:MapTopo} and~\ref{p:MapInj}, 
one might wonder whether tightness in $\Ch^\alpha_\Sp$ of a sequence 
$\{\zeta_n\}_n\subset \Ch^\alpha_\Sp(\ft)$ 
together with convergence of $\{K(\zeta_n)\}_n$ in $\Pi$ 
can directly imply convergence of $\{\zeta_n\}_n$ in $\Ch^\alpha_\Sp$. 

The answer is no as the following example shows. 
For all $n$ odd, let $\zeta_n$ be the directed tree given by one infinite branch $e$ embedded into
 $\R^2$ as $\{0\}\times(-\infty,0]$, while for $n$ even let $\zeta_n$ be the directed tree 
 given by the same $e$ together with a branch $e_n$ embedded as $\{(\frac{1-t}{n},-t)\colon t\in[0,1]\}$. 
Clearly, the sequence  $\{\zeta_n=(\ST_n,\ast_n,d_n,M_n)\}_n\subset\Ch^\alpha_\Sp(\ft)$ and is tight 
in $\Ch^\alpha_\Sp$ but it does not converge \dash the odd subsequence is constant while 
the even one converges to the directed tree formed by two branches $e\cup e_\infty$ where 
$e_\infty$ is embedded as $\{0\}\times[-1,0]$ and the two branches meet at $(0,-1)$. At the same time,  
$\{K(\zeta_n)\}_n$ converges in $\Pi$ to the set $\{\pi_{(0,t)}\colon t\leq 0\}$ for $\pi_{(0,t)}$ identically equal to $0$ on 
$(-\infty,t]$. 
\end{remark}

\subsection{The double Brownian Web tree and special points}\label{sec:DBW}

A crucial aspect of the backward Brownian Web is that it comes naturally associated with a dual (see e.g.~\cite{TW,FINRb}), 
which is given by a family of forward coalescing Brownian motions starting from every point in $\R^2$ or $\R\times\T$, 
in the periodic case. 
In the next theorem we will see how it is possible to devise such a duality in the present context 
and characterise the joint law of the Brownian Web Tree in Definition~\ref{def:BW} and its dual. 

\begin{theorem}\label{thm:DBW}
Let $\alpha<1/2$. There exists a $\Ch^\alpha_\Sp\times\hat\Ch^\alpha_\Sp$-valued random variable 
$\zeta^{\uda}_\bw\eqdef(\zeta^\da_\bw,\zeta^\ua_\bw)$, 
$\zeta^{\dotp}_\bw=(\ST^{\dotp}_\bw,\ast^{\dotp}_\bw,d^{\dotp}_\bw,M^{\dotp}_\bw)$, $\dotp\in\{\da,\ua\}$, whose 
law is uniquely characterised by the following properties
\begin{enumerate}[label=(\roman*)]
\item\label{i:Dist} Both $-\zeta^\ua_\bw\eqdef (\ST^\ua_\bw,\ast^\ua_\bw,d^\ua_\bw,-M^\ua_\bw)$ and $\zeta^\da_\bw$ are 
distributed as the backward Brownian Web Tree in Definition~\ref{def:BW}.
\item\label{i:Cross} Almost surely, for any $\sz^\da\in\ST^\da_\bw$ and $\sz^\ua\in\ST^\ua_\bw$, the paths 
$M^\da_{\bw}(\rho^\da(\sz^\da,\cdot))$ and $M^\ua_{\bw}(\rho^\ua(\sz^\ua,\cdot))$ do not cross, i.e. 
for all $M^\ua_{\bw,t}(\sz^\ua)\leq s_1<s_2\leq M^\da_{\bw,t}(\sz^\da)$ 
	\begin{equ}[e:NonCrossing]
	\prod_{i=1,2}(M^\ua_{\bw,x}(\rho^\ua(\sz^\ua, s_i))-M^\da_{\bw, x}(\rho^\da(\sz^\da, s_i))) \geq 0\,,
	\end{equ}
where $\rho^\da$ (resp. $\rho^\ua$) is the radial map of $\zeta^\da_\bw$ (resp. $\zeta^\ua_\bw$). 
\end{enumerate}
Moreover, almost surely $\zeta^{\uda}_\bw\in\Ch^\alpha_\Sp(\ft)\times\hat\Ch^\alpha_\Sp(\ft)$ and 
$\zeta^\ua_\bw$ is determined by $\zeta^\da_\bw$ and vice-versa.
Finally, $(K(\zeta^{\da}_\bw),\hat K(\zeta^{\ua}_\bw)) $ is distributed according to the double Brownian Web of~\cite[Theorem 6.2.4]{SSS}. 
\end{theorem}

\begin{remark}
Here, given a random variable $(X,Y)$ on some product Polish space 
$\CX \times \CY$, we say that $X$ is determined by $Y$ 
if the conditional law of $X$ given $Y$ is almost surely given by a Dirac mass.
\end{remark}

\begin{proof}
Throughout the proof, we will adopt the notation and conventions of Section~\ref{s:Topo}. 

Notice at first that, by Theorem~\ref{thm:BW}, any $\Ch^\alpha_\Sp\times\hat\Ch^\alpha_\Sp$-valued 
random variable for which~\ref{i:Dist} holds, almost surely belongs to $\Ch^\alpha_\Sp(\ft)\times\hat\Ch^\alpha_\Sp(\ft)$. 

Now, let $(W^\da,W^\ua)$ be the $\cH\times\hat\cH$-valued random variable constructed in~\cite[Theorem 6.2.4]{SSS} and 
$K$ the map in~\eqref{e:CompZeta}. 
Since $W^\da$ is distributed as the backward Brownian Web, 
by Corollary~\ref{cor:Topologies}, $W^\da\eqlaw  K(\zeta^\da_\bw)$ and
$W^\ua\eqlaw -W^\da\eqlaw -K(\zeta^\da_\bw)=\hat K(-\zeta^\da_\bw)$, 
where the first equality is  due to~\cite[Theorem 6.2.4(a)]{SSS} and the last is a consequence of Remark~\ref{rem:MapTopoFor}. 
Therefore, 
$(W^\da, W^\ua)\in K(\Ch^\alpha_\Sp(\ft))\times \hat K(\hat\Ch^\alpha_\Sp(\ft))$ almost surely
so that, by Proposition~\ref{p:MapTopo} and Remark~\ref{rem:MapTopoFor}, there exists a unique 
couple $(\zeta_{W^\da},\zeta_{W^\ua})\in \Ch^\alpha_\Sp(\ft)\times \hat\Ch^\alpha_\Sp(\ft)$ such that 
$(K(\zeta_{W^\da}),\hat K(\zeta_{W^\ua}))=(W^\da, W^\ua)$. 
By Proposition~\ref{p:BW} and Theorem~\ref{thm:BW} we also have  $\zeta^\da_\bw\in\Ch^\alpha_\Sp(\ft)$ almost 
surely so that, since $K(\zeta_{W^\da})\eqlaw  K(\zeta^\da_\bw)$ and 
$K(-\zeta_{W^\ua})\eqlaw K(\zeta^\da_\bw)$,
$(\zeta_{W^\da},\zeta_{W^\ua})$ satisfies~\ref{i:Dist}. 
The definition of the map $K$ in~\eqref{e:CompPath} and~\eqref{e:CompZeta} combined with~\cite[Theorem 6.2.4(b)]{SSS}
ensures that~\ref{i:Cross} holds for $(\zeta_{W^\da},\zeta_{W^\ua})$. The fact that $\zeta_{W^\ua}$ 
is determined by $\zeta_{W^\da}$ is a direct consequence of the fact that this is known to be true for 
$W^\da$ and $W^\ua$ and that $K$ is invertible on $\Ch^\alpha(\ft)$. 

We argue uniqueness. Let $(\zeta,\zeta')$ be another  random variable in $\Ch^\alpha_\Sp\times\hat\Ch^\alpha_\Sp$ 
which satisfies~\ref{i:Dist} and~\ref{i:Cross}. Now,~\ref{i:TreeCond} holds for 
both $\zeta$ and $\zeta'$, while~\ref{i:Dist},~\ref{i:Cross} and~\eqref{e:CompZeta} 
ensure that $(K(\zeta),\hat K(\zeta'))$ satisfies~\cite[Theorem 6.2.4 (a)-(b)]{SSS}. Hence, the conclusion follows 
by the uniqueness part of~\cite[Theorem 6.2.4]{SSS} and Proposition~\ref{p:MapTopo}. 
\end{proof}

\begin{remark}\label{rem:PeriodicDBW}
In the periodic setting Theorem~\ref{thm:DBW} remains true 
upon replacing all the objects and spaces appearing in the statement 
with their periodic counterparts. The proof follows the exact same lines but uses 
Remarks~\ref{rem:t:PeriodicBW} and~\ref{rem:MapTopoPeriodic} 
instead of Theorem~\ref{thm:BW} and Proposition~\ref{p:MapTopo}. 
%
%
%
\end{remark}

\begin{definition}\label{def:DBW}
Let $\alpha<\tfrac12$. We define the {\it double Brownian Web tree} and {\it double periodic Brownian Web tree} 
as the $\Ch^\alpha_\Sp\times\hat\Ch^\alpha_\Sp$ and $\Ch^\alpha_{\Sp,\per}\times\hat\Ch^\alpha_{\Sp,\per}$-valued 
random variables $\zeta^{\uda}_\bw\eqdef(\zeta^{\da}_\bw,\zeta^{\ua}_\bw)$ and 
$\zeta^{\per,\uda}_\bw\eqdef(\zeta^{\per,\da}_\bw, \zeta^{\per,\ua}_\bw)$ given by 
Theorem~\ref{thm:DBW} and Remark~\ref{rem:PeriodicDBW}. 
We will refer to $\zeta^\ua_\bw$ and $ \zeta^{\per,\ua}_\bw$ as the {\it forward} (or dual) and {\it forward periodic 
Brownian Web trees}.

We  denote their laws
by $\Theta^{\uda}_\bw(\dd(\zeta^\da\times\zeta^\ua))$ and $\Theta^{\per,\uda}_\bw(\dd(\zeta^\da\times\zeta^\ua))$, 
 with marginals $\Theta^{\da}_\bw(\dd\zeta)$, $\Theta^{\ua}_\bw(\dd\zeta)$ and 
$\Theta^{\per,\da}_\bw(\dd\zeta)$, $\Theta^{\per,\ua}_\bw(\dd\zeta)$ respectively. 
\end{definition}

\begin{remark}
The proof of Theorem~\ref{thm:DBW} heavily relies on the results of~\cite{FINRb} 
(summarised in~\cite{SSS}). Clearly, it would have been possible to construct the double Brownian Web tree 
directly starting from a countable family of (independent) forward and backward standard Brownian motion, 
turning it into a perfectly coalescing\slash reflecting system (see~\cite[Section 3.1.1]{STW}) and follow the 
same procedure as in~\eqref{e:Skeleton}, Proposition~\ref{p:BW} and Theorem~\ref{thm:BW}. 
\end{remark}

As a first consequence of the duality the Brownian Web tree enjoys we show that each of the $\R$-trees $\ST^\ua_\bw$ and
$\ST^\da_\bw$ has a {\it unique} open end with unbounded rays. 
This end should be thought of as the point at ($\pm$)$\infty$ 
where all the Brownian motions coalesce. We will see in Proposition~\ref{p:TwoEnds} below that 
the periodic Brownian Web tree, instead, has (exactly) {\it two} open ends 
with unbounded rays which are connected by a unique bi-infinite edge.  

\begin{proposition}\label{prop:OneEnd}
Let $\zeta^\ua_\bw$ 
and $\zeta^\da_\bw$ 
be respectively the forward and backward Brownian Web trees. 
Then, almost surely, the $\R$-trees $\ST^\da_\bw$ and $\ST^\ua_\bw$ have precisely one open end 
with unbounded rays, 
which we denote by $\dagger^\ua$ and $\dagger^\da$ respectively. These are precisely the ends of 
Proposition~\ref{p:EndProp}, so that in particular
\begin{equ}
\lim_{\sz\to\dagger^\da} M_t^\da(\sz)=-\infty \qquad\text{and}\qquad \lim_{\sz\to\dagger^\ua} M_t^\ua(\sz)=+\infty\,.
\end{equ}

\end{proposition}
\begin{proof}
We prove the result for $\ST^\ua_\bw$, the other being analogous by duality. 
Notice that the statement follows if we show that for every $r>0$ almost surely there exists a compact 
$\sK\subset \ST^\ua_\bw$ with
$\ST_\bw^{\ua,\,(r)}\subset \sK$, such that for all $\sz,\sz'\in \sK^c$ the path connecting 
$\sz$ and $\sz'$ does not intersect $\ST_\bw^{\ua,\,(r)}$.
Thanks to the double Brownian Web tree we are able to exhibit an explicit compact set for which the latter claim holds. 
Let $r>0$ be fixed, $\cD$ be a countable dense set in $\R^2$ containing $0$ and 
recall that, with probability one, $\zeta^{\da}_\bw=\zeta^{\da}(\cD)$. 

Using the same notation and conventions as in the proof of Proposition~\ref{p:BW}, 
let $\tilde E^N_{R,r}$ be defined according to~\eqref{e:EvProp}. 
Set $\tau\eqdef\tau^\da(\pi^\da_{\tilde z^+},\pi^\da_{\tilde z^-})$, 
$X\eqdef \pid_{\tilde z^+}(\tau_{n})= \pid_{\tilde z^-}(\tau_{n})$ 
and let $\Delta_N$ be the triangular region of $\R^2$ 
with vertices $\tilde z^\pm$ and $(\tau,X)$, base given by the segment joining $\tilde z^+$ and $\tilde z^-$, 
and sides formed by the paths $(s,\pid_{\tilde z^-}(s))_{\tilde t^-\geq s\geq \tau}$, $(s,\pid_{z_n^+}(s))_{t_n^+\geq s\geq \tau}$.
On $\tilde E^N_{R,r}$, $\Delta_N$ is compact and the properness of $M^\ua_\bw$ guarantees that so is 
$\sK_N\eqdef (M^\ua_\bw)^{-1}(\Delta_N)$. 
By point~\ref{i:Cross} in Theorem~\ref{thm:DBW} paths in the forward and backward Web trees do not cross, therefore 
$\ST_\bw^{\ua,\,(r)}\subset \sK_N$ and the path connecting any two points in $\sK_N^c$ cannot intersect 
$\ST_\bw^{\ua,\,(r)}$. Hence, it remains to argue that there is an almost surely finite $N$ for which the 
realisation of $\zeta^{\da}_\bw$ belongs to $\tilde E^N_{R,r}$. This in turn is a direct consequence of~\eqref{e:PropQ} 
and a standard application of Borel--Cantelli. 
\end{proof}

We are now interested in deriving properties of the inverse maps $(M^{\dotp}_\bw)^{-1}$ and $(M^{\per,\dotp}_\bw)^{-1}$, 
for $\dotp\in\{\ua,\da\}$, and how these are related to the degrees of points in the $\R$-trees $\ST^{\dotp}_\bw$ and 
$\ST^{\per,\dotp}_\bw$. 
 We begin with the following proposition, which is a translation in the language of the present paper 
 of~\cite[Proposition 3.10]{FINRb}. 

\begin{proposition}\label{p:CardDeg}
Let $\zeta_\bw^{\uda}=(\zeta^\ua_\bw,\zeta^\da_\bw)$ and $\zeta_\bw^{\uda,\per}=(\zeta^{\ua,\per}_\bw,\zeta^{\da,\per}_\bw)$  
be the double and double periodic Brownian Web trees. 
Then, almost surely for every point $z=(t,x)\in\R^2$
\begin{equ}\label{e:CardDeg}
|(M^\ua_\bw)^{-1}(z)|-1=\sum_{i=1}^{|(M^\da_\bw)^{-1}(z)|} (\deg(\sz^\da_i)-1)
\end{equ}
where $\{\sz^\da_i\}_i$ are the points in $(M^\da_\bw)^{-1}(z)$ and 
$|(M_\bw^{\dotp})^{-1}(z)|$ denotes the cardinality of $(M_\bw^{\dotp})^{-1}(z)$. 
The relation~\eqref{e:CardDeg} holds as well with the arrows $\ua\,,\da$ reversed and for their periodic counterpart. 
\end{proposition}

\begin{proof}
As usual we will focus on the non-periodic case, the other being analogous.  

We claim that for all $z=(t,z)\in\R^2$, $|(M^\da_\bw)^{-1}(z)|=m^b_{\mathrm{out}}(z)$ and 
the right-hand side of~\eqref{e:CardDeg} coincides with $m^b_{\mathrm{in}}(z)$, 
where $m^b_{\mathrm{out}}(z)$ and $m^b_{\mathrm{in}}(z)$ are defined according to~\cite[(3.11) and (3.10)]{FINRb} 
and respectively represent the number of distinct paths ``leaving'' and ``entering'' the point $z$ for the 
backward Brownian Web (by removing the superscript $b$ and reverting the arrows 
the same holds for the forward by duality).

Indeed, for every $\sz^\da\in(M^\da_\bw)^{-1}(z)$, denoting by $\rho^\da$ the radial map associated to $\zeta^\da_\bw$, 
we have that $(-\infty,t]\ni s\mapsto M^\da_{\bw,x}(\rho^\da(\sz^\da,s))$ is a path from $z$.
On the other hand, $\deg(\sz^\da)-1$ corresponds to the number of rays in the tree which coalesce at or reach $\sz$. 
Notice that, since almost surely $\zeta^\da_\bw$ satisfies~\ref{i:TreeCond}, 
the image of the rays coalescing or reaching $\sz$ as well as that of the rays from points in $(M^\da_\bw)^{-1}(z)$ are distinct 
so that the claim follows. 

Now, by Theorem~\ref{thm:DBW} $(K(\zeta^{\da}_\bw),\hat K(\zeta^{\ua}_\bw))$ is distributed as the double Brownian Web 
and almost surely $\zeta_\bw^{\uda}\in\Ch^\alpha_\Sp(\ft)\times\hat\Ch^\alpha_\Sp(\ft)$. 
Since moreover the
restriction of $K$ to $\Ch^\alpha_\Sp(\ft)$ is bijective on its image 
thanks to Proposition~\ref{p:MapTopo},~\eqref{e:CardDeg} is a direct consequence of~\cite[Proposition 3.10]{FINRb}. 
\end{proof}

We are now ready to classify the different points in $\R^2$ or in $\R\times\T$ based on the meaning they have 
for the (periodic) Brownian Web tree (and its dual) as we constructed it. 

\begin{definition}\label{def:Type}
Let $\zeta_\bw^{\uda}=(\zeta^\ua_\bw,\zeta^\da_\bw)$ be the double Brownian Web tree. 
For $\dotp\in\{\ua,\da\}$, the type of a point $z\in\R^2$ for $\zeta^{\dotp}_\bw$ is $(i,j)\in\N^2$, where 
\begin{equ}
i=\sum_{i=1}^{|(M^{\dotp}_\bw)^{-1}(z)|}(\deg(\sz_i^{\dotp})-1)\quad\text{and}\quad j=|(M^{\dotp}_\bw)^{-1}(z)|\,.
\end{equ} 
Above, $\{\sz^{\dotp}_i\,:\,i\in\{1,\dots, |(M^{\dotp}_\bw)^{-1}(z)|\}\}=(M^{\dotp}_\bw)^{-1}(z)$. 
We define $S^\ua_{i,j}$ (resp. $S^\da_{i,j}$) as the subset of $\R^2$ containing all points of type $(i,j)$
for the forward (resp. backward) Brownian Web tree. 
For the periodic Brownian Web $\zeta_\bw^{\per,\uda}=(\zeta^{\per,\ua}_\bw,\zeta^{\per,\da}_\bw)$, 
the definition is the same as above and the set of all of points in $\R\times\T$ of type $(i,j)$ for the backward (resp. forward) 
periodic Brownian Web tree, will be denoted by 
$S^{\per,\da}_{i,j}$ (resp. $S^{\per,\ua}_{i,j}$).
\end{definition}

\begin{theorem}\label{thm:Types}
For the backward and backward periodic Brownian Web trees $\zeta^\da_\bw$ and $\zeta^{\da,\per}_\bw$, 
almost surely, every $z\in\R^2$ (resp. $\R\times\T$) is of one of the following types, 
all of which occur: $(0,1),\,(1,1),\,(2,1),\,(0,2),\,(1,2)$ and $(0,3)$. 
Moreover, almost surely, for every $t\in\R$
\begin{itemize}[noitemsep,label=-]
\item $S^\da_{0,1}$ has full Lebesgue measure on $\R^2$ and 
$S^\da_{0,1}\cap\{t\}\times\R$ has full Lebesgue measure in $\{t\}\times\R$,
\item $S^\da_{1,1}$ and $S^\da_{0,2}$ have Hausdorff dimension $3/2$ while 
$S^\da_{1,1}\cap\{t\}\times\R$ and $S^\da_{0,2}\cap\{t\}\times\R$ are both countable 
and dense in $\{t\}\times\R$,
\item $S^\da_{1,2}$ has Hausdorff dimension $1$, $S^\ua_{2,1}$ and $S^\ua_{0,3}$ 
are countable and dense while $S^\da_{2,1}\cap\{t\}\times\R$, $S^\da_{1,2}\cap\{t\}\times\R$ 
and $S^\da_{0,3}\cap\{t\}\times\R$ have each cardinality at most 1. 
\end{itemize}
For deterministic times $t$, $S^\da_{2,1}\cap\{t\}\times\R$, $S^\da_{1,2}\cap\{t\}\times\R$ 
and $S^\da_{0,3}\cap\{t\}\times\R$ are almost surely empty. 
Upon reversing all arrows, the properties above hold for the forward and forward periodic Brownian Web trees.
\end{theorem}
\begin{proof}
Arguing as in the proof of Proposition~\ref{p:CardDeg}, 
the statement follows immediately by~\cite[Theorems 3.11, 3.13 and 3.14]{FINRb}. 
\end{proof}

Thanks to the classification above, we can now prove one of the features 
that distinguishes the Brownian Web tree and its periodic version. 
In the next proposition, whose conclusion was first noted in~\cite{CMT}, 
we show that the periodic Brownian Web tree possesses a unique bi-infinite path connecting its 
{\it two} open ends with unbounded rays. 

\begin{proposition}\label{p:TwoEnds}
For $\dotp\in\{\da,\ua\}$, let $\zeta^{\per,\dotp}_\bw = (\ST^{\per,\dotp}_\bw, \ast_\bw^{\per,\dotp}, d^{\per,\dotp}_\bw, 
M^{\per, \dotp}_\bw)$ be the periodic backward and forward Brownian Web trees of Definition~\ref{def:DBW}. 
Then, almost surely, each $\ST^{\per,\da}_\bw$ and $\ST^{\per,\ua}_\bw$ has exactly two open ends with unbounded rays and 
a unique bi-infinite edge connecting them. 
\end{proposition}
\begin{proof}
Since $\ST^{\per,\da}_\bw$ and $\ST^{\per,\ua}_\bw$ are periodic directed trees, we already know they have one 
open end with unbounded rays, and this is the one for which~\eqref{e:Back} holds 
(for the forward periodic Web see Remark~\ref{rem:CharTreeFor}). Denote them by $\dagger^\da$ and $\dagger^\ua$ and 
let $\rho_\per^\da$ and $\rho_\per^\ua$ be the radial maps introduced in Definition~\eqref{def:RadMap}. 
Similarly to~\eqref{e:CPS}, for $t_0,t_1\in\R$, $t_0<t_1$, we introduce 
\begin{equs}
\Xi^\ua_\T(t_0,t_1)&\eqdef \{\rho_\per^\ua(\sz,t_1)\,:\, \sz\in\ST^{\per,\ua}_\bw\text{ and } M^{\per,\ua}_{t,\bw}(\sz)\leq t_0\}\\
\Xi^\da_\T(t_1,t_0)&\eqdef \{\rho_\per^\da(\sz,t_0)\,:\, \sz\in\ST^{\per,\da}_\bw\text{ and }M^{\per,\da}_{t,\bw}(\sz)\geq t_1\}
\end{equs}
and set $\eta_\T^\ua(t_0,t_1)$ and $\eta_\T^\da(t_1,t_0)$ to be the cardinality of $\Xi^\ua_\T(t_0,t_1)$ and 
$\Xi^\da_\T(t_1,t_0)$ respectively. We inductively define the sequence of stopping times
\begin{equs}
\tau_1&\eqdef \inf\{t>0\,:\,\eta_\T^\ua(0,t)=1\}\\
\tau_k&\eqdef \inf\{t>\tau_{k-1}\,:\,\eta_\T^\ua(\tau_{k-1},t)=1\}\,.
\end{equs}
These stopping times coincide (in distribution) with those in the proof of~\cite[Theorem 3.1]{CMT}, where it is 
further showed that almost surely $\lim_{k\to\infty}\tau_k=+\infty$. 

Now, by definition, for every $k\geq 1$, there must exist a point $z_{k-1}\in\T\times\{\tau_{k-1}\}$ such that 
$|(M^{\per,\ua}_\bw)^{-1}(z_{k-1})|\geq 2$
and the distance of (at least) two elements in $(M^{\per,\ua}_\bw)^{-1}(z_{k-1})$ is $2(\tau_k-\tau_{k-1})$. 
By~\eqref{e:CardDeg} and Theorem~\ref{thm:Types}, it follows that there exists 
exactly one point $(M^{\per,\da}_\bw)^{-1}(z_{k-1})$ whose degree is greater or equal to $2$. Denote it by $\sz_k$.
Then the map $\beta^\da:\R\to\ST^{\per,\da}_\bw$ given by 
\begin{equ}
\beta^\da(s)\eqdef 
\begin{cases}
\rho^\da_\per(\sz_k,s)\,,&\text{for $s\in(\tau_{k-1},\tau_k]$}\\
\rho^\da_\per(\sz_0,s)\,&\text{for $s<0$.}
\end{cases}
\end{equ}
is not only well-defined by Theorem~\ref{thm:DBW}\ref{i:Cross} but also {\it uniquely} defined since so is the choice 
of the point $\sz_k$. The map $\beta^\da$ shows that there are exactly two open ends with unbounded rays, 
and $\beta^\da(\R)$ is the unique linear subtree of $\ST^{\per,\da}_\bw$ satisfying the properties 
in~\cite[Lemma 3.7(i)]{Ch}.
\end{proof}

\section{The Discrete Web Tree and convergence}\label{sec:DWT}

In this section, we introduce the discrete web and its dual, and show that, as a couple, they converge to 
the Double Brownian Web Tree of Definition~\ref{def:DBW}. 

\subsection{The Double Discrete Web Tree}
\label{sec:graphical}

 We begin our analysis with the spatial tree representation of 
a family of coalescing backward random walks and its dual. 
The construction below will directly provide a coupling between forward and backwards paths under which 
one is determined by the other and the two satisfy the non-crossing property of Theorem~\ref{thm:DBW}\ref{i:Cross}. 

Let $\delta\in(0,1]$ and $(\Omega,\cA,\P_\delta)$ be a standard probability space supporting 
four Poisson random measures, $\mu_{\gamma}^L$, $\mu_{\gamma}^R$, $\hat\mu_{\gamma}^L$ and $\hat\mu_{\gamma}^R$.  
The first two, $\mu_{\gamma}^L$ and $\mu_{\gamma}^R$, live on $\D^\da_\delta\eqdef\R\times\delta\Z$, are independent and 
have both intensity $\gamma\lambda$, 
where, for every $k\in\delta\Z$, $\lambda(\dd t, \{k\})$ is a copy of the Lebesgue measure on $\R$ and
throughout the section 
\begin{equ}[e:gamma]
\gamma=\gamma(\delta)\eqdef\frac{1}{2\delta^2}\,.
\end{equ} 
The others live on $\D^\ua_\delta\eqdef\R\times\delta(\Z+1/2)$, and are obtained from the formers by setting, 
for every measurable $A\subset\D^\ua_\delta$
\begin{equ}[e:DualPoisson]
\hat\mu^L_\gamma(A)\eqdef\mu^R_\gamma(A-\delta/2)\qquad\text{and}\qquad 
\hat\mu^R_\gamma(A)\eqdef\mu^L_\gamma(A+\delta/2)\,.
\end{equ}
Here, $A\pm\delta/2$ is the translate of $A$ in the spatial direction, i.e. $A\pm\delta/2\eqdef\{z\pm(0,\delta/2)\,:\,z\in A\}$. 

From now on, we will adopt the convention of writing $z\in\mu_{\gamma}^{\dotp}$, $\dotp\in\{R,L\}$, 
if {$\mu_{\gamma}^{\dotp}(\{z\})=1$}. 
We represent the Poisson points of $\mu_{\gamma}^L$, $\mu_{\gamma}^R$, $\hat\mu_{\gamma}^L$ and 
$\hat\mu_{\gamma}^R$ with arrows as follows. 
If $z\in\mu_{\gamma}^L$ (resp. $\mu_{\gamma}^R$) then we draw an arrow 
from $z$ to $z-\delta$ (resp. $z+\delta$), and similarly for  $\hat\mu^L_\gamma$ and $\hat\mu^R_\gamma$, 
as shown in Figure~\ref{f:0BDandDual}. We also define
\begin{equ}[e:defmuT]
\mu_\gamma^T = \{z -\delta\,:\, z \in \mu_\gamma^L\} \cup \{z +\delta\,:\, z \in \mu_\gamma^R\}\;,
\end{equ}
and similarly for $\hat\mu_\gamma^T$.
(Here, $T$ stands for ``tip'' since $\mu_\gamma^T$ denotes the collection of all tips of arrows.)

\begin{figure}[t]
\setlength{\unitlength}{0.27cm}

\begin{center}
\begin{tikzpicture}[scale=0.7,baseline=0.85cm]
\draw[->] (0,0) -- (8,0);
\draw[black!25] (0,4) -- (8,4);
\foreach \x in {1,...,7}
{
	\draw (\x,0) -- (\x,4.5);
}
\foreach \x in {1,...,8}
{
	\draw[blue!25] (\x-0.5,0) -- (\x-0.5,4.5);
}
\foreach \x/\y in {1/2,1/2.8,2/1.1,3/2,3/3.2,4/2.5,5/0.5,6/1.8,6/2.2,6/3.7}
{
	\draw[black!50,->] (\x,\y) -- (\x+1,\y);
}
\foreach \x/\y in {1/2,1/2.8,2/1.1,3/2,3/3.2,4/2.5,5/0.5,6/1.8,6/2.2,6/3.7}
{
	\draw[blue!25,->] (\x+0.5,\y-0.02) -- (\x-0.5,\y-0.02);
}
\foreach \x/\y in {2/0.2,3/3.1,4/1.5,5/1,5/3.5,6/0.8,6/2.7,7/0.3}
{
	\draw[black!50,->] (\x,\y) -- (\x-1,\y);
}
\foreach \x/\y in {2/0.2,3/3.1,4/1.5,5/1,5/3.5,6/0.8,6/2.7,7/0.3}
{
	\draw[blue!25,->] (\x-0.5,\y-0.02) -- (\x+0.5,\y-0.02);
}
\draw[very thick,dr!80] (4,0) -- (4,1) -- (5,1) -- (5,2.5) -- (4,2.5) -- (4,3.2) -- (3,3.2) -- (3,4);
\draw[very thick,blue!80] (3.5,0) -- (3.5,1.48) -- (4.5,1.48) -- (4.5,2.48) -- (3.5,2.48) -- (3.5,3.18) -- (2.5,3.18) -- (2.5,4);

\node[fill=white,above] at (3,4) {$y$};
\node[fill=white,below] at (3.5,0) {$\hat y$};
\node[left] at (0,0) {$0$};
\node[left] at (0,4) {$t$};
\node[inner sep=0pt,minimum size=1mm] at (3,4) [circle,fill=dr!80] {};
\node[inner sep=0pt,minimum size=1mm] at (3.5,0) [circle,fill=blue!80] {};
\end{tikzpicture}\qquad 
\begin{tikzpicture}[scale=0.7,baseline=0.85cm]
\draw[->] (0,0) -- (8,0);
\draw[black!25] (0,4) -- (8,4);
\foreach \x in {1,...,7}
{
	\draw (\x,0) -- (\x,4.5);
}
\foreach \x/\y in {1/2,1/2.8,2/1.1,3/2,3/3.2,4/2.5,5/0.5,6/1.8,6/2.2,6/3.7}
{
	\draw[black!50,->] (\x,\y) -- (\x+1,\y);
}
\foreach \x/\y in {2/0.2,3/3.1,4/1.5,5/1,5/3.5,6/0.8,6/2.7,7/0.3}
{
	\draw[black!50,->] (\x,\y) -- (\x-1,\y);
}
\draw[very thick,dr!80] (4,0) -- (4,1) -- (5,2.5) -- (4,3.2) -- (3.4,4);
\draw[very thick,dr!80] (4,0) -- (4,1) -- (5,2.5) -- (4,3.2) -- (4,3.5) -- (4.7,4);

\node[fill=white,above] at (3.4,4) {$x$};
\node[fill=white,above] at (4.7,4) {$z$};
\node[left] at (0,0) {$0$};
\node[left] at (0,4) {$t$};
\node[inner sep=0pt,minimum size=1mm] at (3.4,4) [circle,fill=dr!80] {};
\node[inner sep=0pt,minimum size=1mm] at (4.7,4) [circle,fill=dr!80] {};
\end{tikzpicture}

\end{center}
\vspace{-1em}\caption{On the left: graphical representation of the realisation of the 
Poisson processes $\mu^L$ and $\mu^R$, 
and their dual $\hat\mu^L$ and $\hat\mu^R$ which respectively live on 
$\D^\da_\delta$ and $\D_\delta^\ua$. 
The red and blue lines illustrate the restrictions of the backward and forward paths $\pidd_{(t,y)}$ and 
$\piud_{(0,\hat y)}$ to the interval $[0,t]$. 
On the right: the paths starting from $x$ and $z$ in the interpolated tree.}\label{f:0BDandDual}
\end{figure}

Let us now introduce two families of random walks. 
We define $\{\pidd_z(s)\}_{s\leq t}$, for $z=(t,y)\in\D^\da_\delta$, as the random walk going backwards in time, 
``following the arrows'' determined by $\mu_{\gamma}^L$ and $\mu_{\gamma}^R$, 
and, for $z=(t,y)\in\D^\ua_\delta$, $\{\piud_z(s)\}_{s\geq t}$ 
as the forward random walk which follows those of $\hat\mu^L$ and $\hat\mu^R$, 
as shown in Figure~\ref{f:0BDandDual}. (By convention, if $z$ is the start of an arrow, then
$\pidd_z$ and $\piud_z$ start by going downwards / upwards.)
These are almost surely well-defined  $\mu_{\gamma}^L$ and $\mu_{\gamma}^R$ are disjoint with probability one
and, for all $z\in\D^\da_\delta$ and $\hat z\in\D^\ua_\delta$,  
$\pidd_z$ is c\`agl\`ad (or c\`adl\`ag if we run time backwards from $+\infty$ to $-\infty$), 
while $\piud_{\hat z}$ is c\`adl\`ag. 
Moreover, $\{\pidd_z\}_z$ and $\{\piud_{\hat z}\}_{\hat z} $ are coalescing families of paths 
starting from every point in $\D^\da_\delta$ and $\D^\ua_\delta$ respectively, which do not cross. 

\begin{definition}\label{def:DWT}
Let $\delta\in(0,1]$, $\gamma$ as in~\eqref{e:gamma}, 
$\mu_{\gamma}^L$ and $\mu_{\gamma}^R$ be two independent Poisson random measures on
$\D^\da_\delta$ of intensity $\gamma\lambda$, $\hat\mu^L$ and $\hat\mu^R$ be given as in~\eqref{e:DualPoisson} 
and $\{\pidd_z\}_{z\in\D^\da_\delta}$ and $\{\piud_{\hat z}\}_{\hat z\in \D^\ua_\delta} $ be the 
families of coalescing random walks introduced above. We define the {\it Double Discrete Web Tree} as 
the couple $\zeta^{\uda}_\delta\eqdef(\zeta^\da_\delta, \zeta^\ua_\delta)$, in which
\begin{itemize}[noitemsep,label=-]
\item $\zeta^{\da}_\delta\eqdef(\ST^{\da}_\delta, \ast^{\da}_\delta, d^{\da}_\delta, M^{\da}_\delta)$ is given by setting $\ST^{\da}_\delta = \D^\da_\delta$, $\ast^{\da}_\delta = (0,0)$, $M^{\da}_\delta$ the canonical inclusion, and 
\begin{equ}[e:defdd]
d^{\da}_\delta(z,\bar z) = t + t' - 2 \sup \{ s \le t \wedge t' \,:\, \pidd_z(s) = \pidd_{\bar z}(s)\}\;.
\end{equ}
\item $\zeta^{\ua}_\delta\eqdef(\ST^{\ua}_\delta, \ast^{\ua}_\delta, d^{\ua}_\delta, M^{\ua}_\delta)$ is built 
similarly, but with $\ast^\ua_{\delta} = (0,\delta/2)$ 
and the supremum in \eqref{e:defdd} replaced by $\inf \{ s \ge t \vee t' \,:\, \piud_z(s) = \piud_{\bar z}(s)\}$.
\end{itemize}
\end{definition}

Notice that neither the Discrete Web Tree $\zeta^\da_\delta$ nor its dual are directed spatial $\R$-trees. 
Indeed, even though they satisfy the conditions of Definition~\ref{def:CharTree} and Remark~\ref{rem:CharTreeFor}
the evaluation maps are discontinuous {(but $(\ST^{\dotp}_\delta, \ast^{\dotp}_\delta, d^{\dotp}_\delta)$ is still 
a complete random $\R$-tree as $\pidd$ and $\piud$ are c\`agl\`ad and c\`adl\`ag, respectively)}. 

To circumvent this technical issue, we introduce two connected subsets of $\R^2$, $\cS_\delta^\da$ and $\cS^\ua_\delta$, 
obtained by interpolating the Poisson points of $\mu^{\dotp}_\gamma$ and 
$\hat\mu^{\dotp}_\gamma$, $\dotp\in\{L,R\}$, 
and which will represent the image of modified evaluation maps. 
Fix a realisation of $\mu^{\dotp}_\gamma$, $\dotp\in\{L,R\}$, and 
consider $\mu^{T}_\gamma$ as in \eqref{e:defmuT}. 
Given $z = (t,x) \in \D^\da_\delta$, we then define $z^\da$ as follows. Let $t^\da = \sup\{s < t\,:\, (s,x) \in \mu^{R}_\gamma \cup \mu^L_\gamma \cup \mu^T_\gamma\}$ and set { $z^\da = (t^\da, \pidd_z(t^\da))$}. 
We then define $\cS_\delta^\da$ as the union of all closed line segments joining $z$ to $z^\da$
with $z \in \mu^{R}_\gamma \cup \mu^L_\gamma \cup \mu^T_\gamma$. Given $z = (t,x) \in \D^\da_\delta$ and
setting $z^\ua = (t^\ua,x)$ with $t^\ua = \inf\{s \ge t\,:\, (s,x) \in \mu^{R}_\gamma \cup \mu^L_\gamma \cup \mu^T_\gamma\}$, we write $\tilde M_\delta^\da(z) \in \cS^\da_\delta$ for the unique element on the line segment
joining $z^\ua$ to $z^\da$ with the same time coordinate as $z$ (see Figure~\ref{f:0BDandDual} on the left).
The set $\cS^\ua_\delta$ is defined
similarly, but with time reversed.
It is immediate to see that, almost surely, the sets $\cS^\da_\delta$ and $\cS^\ua_\delta$ are well-defined and connected. 
With the previous construction at hand we are ready for the following definition. 

\begin{definition}\label{def:IDWT}
In the same setting as Definition~\ref{def:DWT}, we define the {\it Interpolated Double Discrete Web Tree} 
as the couple $\tilde\zeta^{\uda}_\delta\eqdef(\tilde\zeta^\da_\delta, \tilde\zeta^\ua_\delta)$ in which 
$\tilde\zeta^{\dotp}_\delta\eqdef(\ST^{\dotp}_\delta, \ast^{\dotp}_\delta, d^{\dotp}_\delta, \tilde M^{\dotp}_\delta)$, 
$\dotp\in\{\ua,\da\}$, and $(\ST^{\dotp}_\delta,\ast^{\dotp}_\delta,d^{\dotp}_\delta)$ coincides with that of $\zeta^{\dotp}_\delta$, 
while the evaluation map $\tilde M^{\dotp}_\delta$ is defined as just described. 
\end{definition}

\begin{proposition}\label{p:DWTisChar}
For any $\delta\in(0,1]$ and $\alpha\in(0,1)$, almost surely the interpolated double Discrete Web tree 
$\tilde\zeta^{\uda}_\delta$ in Definition~\ref{def:IDWT} belongs to $\Ch^\alpha_\Sp\times\hat\Ch^\alpha_\Sp$ 
and the evaluation maps $\tilde M^{\dotp}_\delta$, $\dotp\in\{\ua,\da\}$ are bijective on $\cS^{\dotp}_\delta$. 
Moreover, it satisfies the following two properties
\begin{enumerate}[label=(\roman*$_\delta$)]
\item\label{i:Distd} $-\tilde\zeta^\ua_\delta+\delta/2\eqlaw\tilde\zeta^\da_\delta$ where 
	$-\tilde\zeta^\ua_\delta+\delta/2\eqdef (\ST^\ua_\delta, \ast^\ua_\delta, d^\ua_\delta, -\tilde M^\ua_\delta+\delta/2)$
\item\label{i:Crossd} almost surely, for every $\sz^\da\in\ST^\da_\delta$ and $\sz^\ua\in\ST^\ua_\delta$ 
	there exists $c\in\{+1,-1\}$ such that for all 
	$\tilde M^\ua_{\delta,t}(\sz^\ua)\leq s_1<s_2\leq \tilde M^\da_{\delta,t}(\sz^\da)$ 
	\begin{equ}[e:AlmostNonCrossing]
	\prod_{i=1,2}(\tilde M^\ua_{\delta,x}(\rho^\ua(\sz^\ua, s_i))-\tilde M^\da_{\delta,x}(\rho^\da(\sz^\da, s_i))+{2}c\delta) \geq 0
	\end{equ}
\end{enumerate}
At last, almost surely, for $\dotp\in\{\ua,\da\}$
\begin{equ}[e:Dist_p_Mp]
\sup_{\sz\in\ST^{\dotp}_\delta}\|\tilde M^{\dotp}_\delta(\sz)-M^{\dotp}_\delta(\sz)\|\leq \delta
\end{equ}
where $M^{\dotp}_\delta$ are the evaluation maps of the double Discrete Web Tree in Definition~\ref{def:DWT}. \end{proposition}

\begin{proof}
The proof of the statement is an immediate consequence of basic properties of Poisson random measures 
and the definition of the sets $\cS^\da_\delta$ and $\cS^\ua_\delta$. 
{ We only notice that~\eqref{e:AlmostNonCrossing} would be the same as the non-crossing condition 
in~\eqref{e:NonCrossing} if the summand $2c\delta$ were not there. 
Since the families of random walks $\{\pidd_z\}_z$ and $\{\piud_{\hat z}\}_{\hat z} $ are non-crossing, 
the Double Discrete Web Tree of 
Definition~\ref{def:DWT} satisfies~\eqref{e:NonCrossing}, which, together with~\eqref{e:Dist_p_Mp}, 
immediately implies~\eqref{e:AlmostNonCrossing} for the interpolated discrete web. }
\end{proof}

\subsection{Tightness and convergence}

We are now ready to show that the family $\{\tilde\zeta^{\uda}_\delta\}_\delta$ is tight. 

\begin{proposition}\label{p:DWTTight}
Let $\alpha\in(0,1)$ and, for $\delta\in(0,1]$, let $\Theta^{\uda}_\delta$ be the law on 
$\Ch^\alpha_\Sp\times\hat\Ch^\alpha_\Sp$ 
of the Interpolated Double Discrete Web Tree $\tilde\zeta^{\uda}_\delta=(\tilde\zeta^\da_\delta,\tilde\zeta^\ua_\delta)$
of Definition~\ref{def:IDWT} and denote by $\Theta^{\dotp}_\delta$ with $\dotp\in\{\ua,\da\}$ the law of 
$\tilde\zeta^{\dotp}_\delta$. 
Then, for any $\alpha\in(0,\tfrac12)$ the family $\Theta^{\uda}_\delta$ is tight in $\Ch^\alpha_\Sp\times\hat\Ch^\alpha_\Sp$. 

Furthermore, for any $\theta>\tfrac32$ and $r>0$, the following holds
\begin{equ}\label{e:deltaNet}
\lim_{K\ua\infty}\liminf_{\delta\da 0}\Theta^\da_\delta\Big(\forall\,\eps\in(0,1]\,,\,\cN_{d}(\ST^{(r)},\eps)\leq K\eps^{-\theta}\Big)=1\,.
\end{equ}
\end{proposition}

\begin{proof}
Let us point out that since by Proposition~\ref{p:DWTisChar}\ref{i:Distd}, 
$-\tilde\zeta^\ua_\delta+\delta/2\eqlaw\tilde\zeta^\da_\delta$,  
it suffices to show that the family $\{\Theta^\da_\delta\}_\delta$ is tight in $\Ch^\alpha_\Sp$. 
{Again by Proposition~\ref{p:DWTisChar}, for every $\delta>0$ 
the probability measure $\Theta^\da_\delta$ is supported on $\Ch^\alpha_\Sp$, hence 
point 1 of Proposition~\ref{p:EndProp} guarantees that we only need to prove 
tightness of $\{\Theta^\da_\delta\}_\delta$ in $\T^\alpha_\Sp$, for which in turn 
we invoke Prokhorov's theorem and the characterisation of compact subsets of $\T^\alpha_\Sp$ 
in Proposition~\ref{p:Compactness}. 
Notice that point 1. therein is implied by~\eqref{e:deltaNet}, while points 2. and 3. 
can be easily seen to hold provided that for all $r>0$, }
%
\begin{equs}
&\lim_{\eps\da 0}\liminf_{\delta\da 0}\Theta^\da_\delta\Big(\sup\{\|M(\sz)-M(\sw)\|\,:\,\sz,\sw\in\ST^{(r)}\,,\,d(\sz,\sw)\leq\eps\}\leq \eps^\alpha\Big)=1\,,\qquad\label{e:deltaModCont}\\
&\lim_{K\ua\infty}\liminf_{\delta\da 0}\Theta^\da_\delta(b_\zeta(r)\leq K)=1\,.\label{e:deltaProp}
\end{equs}
These can be shown by following the same strategy and estimates as in the 
proof of Proposition~\ref{p:BW}, so that below we will adopt the notation and conventions therein. 

Notice at first that, for any $z=(t,x)$ in a countable dense set $\cD$ of $\R^2$, 
if $\{z_\delta\}_\delta$ is such that for all $\delta\in(0,1]$, $z_\delta\in\D_\delta^\da$ and 
$\{z_\delta\}_\delta$ converges to $z$, then, by Donsker's invariance principle, 
the backward random walk $\pidd_{z_\delta}$ defined above 
converges in law to a backward Brownian motion $\pi^\da_z$ started at $z$.  

Let $\{z^{\pm}_\delta\}_\delta\subset Q^\pm_R\cap (\D_\delta)$ be sequences converging to $z^\pm$. 
Denoting by $E_R^\delta$ the event $E_R$ in~\eqref{e:ER}, but in which $z^\pm$ is replaced by $z^\pm_\delta$, 
we see that the previous observation implies
\begin{equ}[e:deltaER]
\liminf_{\delta\da 0} \P_\delta(E_R^\delta)=\P(E_R)
\end{equ}
so that~\eqref{e:RP} holds. Moreover, the analog of~\cite[Proposition 4.1]{FINR} (see also~\cite[pg 326]{SSS}) 
for random walks ensures that for all $R,\,r>0$ and $a<b$
\begin{equ}[e:deltaNumber]
\limsup_{\delta\da 0}\E_\delta[\eta^R(a,b)]\leq\E[\eta^R(a,b)]
\end{equ}
where $\eta^R(a,b)$ is the cardinality of $\Xi_R(a,b)$ given in~\eqref{e:CPS} and $\E_\delta$ is the 
expectation with respect to $\P_\delta$. 
Thanks to~\eqref{e:deltaER} and~\eqref{e:deltaNumber}, we can argue as in Lemma~\ref{l:TB} and 
obtain that there exists a constant $C=C(r)>0$ independent of $\delta$ such that for all $K>0$
\begin{equ}
\limsup_{\delta\da 0}\P_\delta(\cN_{d}(\ST^{(r)},\eps)> K\eps^{-\theta})\leq \frac{C}{\sqrt{K}}
\end{equ}
so that by Borel-Cantelli~\eqref{e:deltaNet} follows.

As in Proposition~\ref{p:BW}, the uniform local H\"older continuity of the evaluation maps $M^\da_\delta$ 
can be reduced to properties of the paths $\pidd$. 
For fixed $R$ and $r$, let 
\begin{equ}
\Psi^\delta(\eps)\eqdef \sup\{|\pidd_z(s)-\pidd_z(t)|\,:\,z\in\D_\delta,\,M^\da_\delta(s,\pidd_z)\in\Lambda_{r,R},\,t\in[s-\eps,s]\}
\end{equ}
If $\Psi^\delta(\eps)\leq \eps^\alpha/4$ for every $\eps \ge 4\delta$, 
then for every $(s,\pidd_z),(t,\pidd_{z'})\in\ST^{\da,\,(r)}_\delta$ such that 
$d_\delta^\da((s,\pidd_z),(t,\pidd_{z'}))\leq\eps$, we have 
\begin{equ}
|M^\da_{\delta,x}(s,\pidd_z)-M^\da_{\delta,x}(t,\pidd_{z'})|\leq 2\delta+\frac{\eps^\alpha}{2}\leq \eps^\alpha\,.
\end{equ}
where we exploited the triangle inequality and~\eqref{e:Dist_p_Mp}. %
%
Therefore,~\eqref{e:deltaModCont} follows at once if 
\begin{equ}[e:ModContDis]
\limsup_{\eps\to 0}\liminf_{\delta\to 0}\Theta^\da_\delta\left(\Psi^\delta(\eps)\leq\eps^\alpha/4\right)=1\,.
\end{equ}
This in turn follows from the same arguments as in the proof of Lemma~\ref{l:ModCont}, 
together with the fact that if $\{z_0^{+,\delta}\}_\delta$ and $\{z_0^{-,\delta}\}_\delta$ are sequences of points in 
$R^\pm_{z_0}\cap(\D_\delta^\da)$ converging to $z_0^+$ and $z_0^-\in\cD$ respectively, then 
\begin{equ}
\liminf_{\delta\da0}\P_\delta\Big(\sup_{h\in[0,2\eps]}|\pidd_{z_0^{\pm,\delta}}(t_0-h)-x_0|\leq\eps^\alpha/32\Big)=\P(E^\eps_{R,r}(z_0))\,.
\end{equ}

Finally,~\eqref{e:deltaProp} can be proved by proceeding as in Lemma~\ref{l:Prop} and adapting the definition of the 
event $\tilde E^K_{R,r}$ in~\eqref{e:EvProp} as done for $E_R$ above. 
\end{proof}

In the following theorem we show that the Interpolated Double Discrete Web tree converges in law 
to the Double Brownian Web Tree. 

\begin{theorem}\label{thm:DWTConv}
Let $\alpha\in(0,1/2)$ and, for $\delta\in(0,1]$, $\Theta^{\uda}_\delta$ be the law on $\Ch^\alpha_\Sp\times \hat\Ch^\alpha_\Sp$ 
of the Interpolated Double Discrete Web Tree $\tilde\zeta^{\uda}_\delta$ in Definition~\ref{def:IDWT}. 
Then, as $\delta\da 0$, $\Theta^{\uda}_\delta$ converges to $\Theta^{\uda}_\bw$ weakly on 
$\Ch^\alpha_\Sp\times \hat\Ch^\alpha_\Sp$. 
\end{theorem}

\begin{proof}
Thanks to Proposition~\ref{p:DWTTight}, the sequence 
$\{\tilde\zeta^{\uda}_\delta=(\tilde\zeta^\da_\delta,\tilde\zeta^\ua_\delta)\}_\delta$
is tight in $\Ch^\alpha_\Sp\times \hat\Ch^\alpha_\Sp$. 
Moreover, Proposition~\ref{p:DWTisChar}~\ref{i:Distd} and~\ref{i:Crossd} imply that any limit point 
$\zeta^{\uda}=(\zeta^\da,\zeta^\da)$ must be such that $-\zeta^\ua\eqlaw\zeta^\da$ and 
the non-crossing property holds. 
In view of Theorem~\ref{thm:DBW}, the statement then follows once we show that 
$\tilde\zeta^\da_\delta \to \zeta^\da_\bw$ in law as $\delta \to 0$. 
To do so, we will apply Theorem~\ref{thm:ConvBW}, 
for which we need to verify the validity of~\ref{i:Conv1} and~\ref{i:Conv2}. 

Clearly, for any $z^1,\dots,z^k\in\R^2$, if $\{z^i_\delta\}_\delta$ is such that $z^i_\delta\in\D_\delta$ and 
$z_\delta^i\to z^i$ as $\delta\to 0$, then $(\pidd_{z^i_\delta}(\cdot))_i$ converges in law to a family of coalescing 
Brownian motions starting at $z^1,\dots,z^k$. Since furthermore~\eqref{e:Dist_p_Mp} holds,~\ref{i:Conv1} follows. 

For~\ref{i:Conv2}, our construction implies that, for any $t,x\in\R$, $h,\eps>0$, 
$\#\{\rho^\da_\delta(\sw,t-h)\,\,\sw\in (\tilde M^\da_\delta)^{-1}(I_{t,x,\eps})\}\eqlaw \hat\eta_\delta(t,t+h;x-\eps,x+\eps)$, 
where $\rho^\da_\delta$ is the radial map of $\tilde\zeta^\da_\delta$ and $\hat\eta_\delta$ 
was defined in~\cite[Definition 2.1]{FINR}\footnote{The subscript $\delta$ stands for fact that in~\cite[Definition 2.1]{FINR}, 
$\hat\eta$ was defined for families of Brownian motions, that in $\hat\eta_\delta$ are replaced by random walks}. 
For the latter, the statement was shown in the proof of~\cite[Theorem 6.1]{FINR}. 
\end{proof}

\begin{appendix}

{
\section{Some proofs}\label{a:Proofs}

In this appendix, we provide the proof of some of the statement of Section~\ref{sec:SpTrees}. 
We begin with a useful lemma. 

\begin{lemma}\label{l:Holder}
Let $\alpha\in(0,1)$ and $\{\zeta_n=(\ST_n,\ast_n,d_n,M_n)\}_{n\in\N\cup\{\infty\}}\subset\Mm^\com_\Sp$. 
Assume that 
\begin{equ}[e:UnifModCont]
\lim_{\delta\to0} \delta^{-\alpha}\sup_{n\in\N\cup\{\infty\}}\omega(M_n,\delta)=0
\end{equ}
and that there exists a sequence of correspondences $\CC^n$ between $\ST_n$ and $\ST_\infty$ 
such that $(\ast_n,\ast_\infty)\in\CC^n$, 
\begin{equ}[e:dis+M]
\lim_{n\to\infty} \dis\CC^n=0\qquad\text{and}\qquad\lim_{n\to\infty}\sup_{(\sz_n,\sz)\in\CC^n}\|M_n(\sz_n)-M_\infty(\sz)\|=0\,.
\end{equ}
Then, $\Delta^\com_\Sp(\zeta_n,\zeta_\infty)$ converges to $0$ in the limit $n\to\infty$. 
\end{lemma}
\begin{proof}
Notice that by the definition the metric $\uDelta^\com_\Sp$ in~\eqref{e:MetricC} 
and Lemma~\ref{l:ContProp}, to conclude that $\lim_n\Delta^\com_\Sp(\zeta_n,\zeta_\infty)=0$, 
it suffices to show that $\lim_n\uDelta^{\com, \CC^n}_\Sp(\zeta_n,\zeta_\infty)=0$. 
For this in turn,~\eqref{e:dis+M} directly implies that the first two summands in~\eqref{e:MetC} 
converge to $0$ so that we only need to focus on the latter.  

Let $1>\eps>0$, $\bar m\in\N$ be such that $2^{\bar m\alpha}\omega(M,2^{-\bar m})\leq \eps$.  
By~\eqref{e:dis+M} there exists $n_{\eps,m}\in\N$ such that for all $n\geq n_{\eps,m}$ 
\begin{equ}[e:boundDisMap1]
\dis \CC^n<\eps 2^{-\bar m}\quad\text{and}\quad\sup_{(\sz_n,\sz)\in\CC^n}\|M_n(\sz_n)-M_\infty(\sz)\|<\eps2^{- \bar m\alpha}\,.
\end{equ}
Let $(\sz_n,\sz),(\sw_n,\sw)\in\CC^n$ and consider first the case of $m> \bar m+2$. 
Then, since $\psi_m$ is bounded above by $1$, we deduce 
\begin{equs}[e:Hol1]
\|\psi_m(d_n(\sz_n,\sw_n))&\delta_{\sz_n,\sw_n}M_n-\psi_m(d_\infty(\sz,\sw))\delta_{\sz,\sw}M_\infty\|\\
&\leq \omega (M_n,2^{-m+2})+\omega(M_\infty,2^{-m+2})\\
&\leq 2^{-(m-2)\alpha} \big(2^{\bar m \alpha}\sup_{n\in\N\cup\{\infty\}}\omega(M_n,2^{-\bar m})\big)\lesssim \eps 2^{-m\alpha}\,.
\end{equs}
where in the last step we used~\eqref{e:UnifModCont}. 

If instead $m\leq \bar m+2$, there are two cases to be considered - 
either both $d_\infty(\sz,\sw)$ and $d_n(\sz_n,\sw_n)$ lie in the support of $\psi_m$ or 
one of them does not, say $d_n(\sz_n,\sw_n)$. 
In this latter scenario, by the first bound in~\eqref{e:boundDisMap1}, 
$d_n(\sz_n,\sw_n)\leq \eps 2^{-\bar m} + d_\infty(\sz,\sw)\leq \eps 2^{-\bar m} + 2^{-m+2}\leq 2^{-m+3}$ 
as $\eps<1$. Hence, in either case $d_n(\sz_n,\sw_n)\leq 2^{-m+3}$. Then, 
invoking the second bound in~\eqref{e:boundDisMap1} and using $\psi_m\leq 1$, we get 
\begin{equs}[e:Hol2]
\|&\psi_m(d_n(\sz_n,\sw_n))\delta_{\sz_n,\sw_n}M_n-\psi_m(d_\infty(\sz,\sw))\delta_{\sz,\sw}M_\infty\|\\
&\leq \sum_{\ell\in\{\sz,\sw\}}\|M_n(\ell_n)-M_\infty(\ell)\|+|\psi_m(d_\infty(\sz,\sw))-\psi_m(d_n(\sz_n,\sw_n))|\|\delta_{\sz_n,\sw_n}M_n\|\\
&\lesssim \eps 2^{-\alpha \bar m}+\|\partial\psi_m\|_\infty |d_\infty(\sz,\sw)-d_n(\sz_n,\sw_n)|\sup_n\omega(M_n,2^{-m+3})\\
&\lesssim \eps 2^{-\alpha \bar m} + \eps 2^{m-\bar m} 2^{-(m-3)\alpha } \Big(2^{(m-3)\alpha} \sup_n\omega(M_n,2^{-m+3})\Big)\lesssim \eps 2^{-m\alpha}
\end{equs}
where, in the third step, we used $\|\partial\psi_m\|_\infty\leq 2^{m+1}$ and the first bound in~\eqref{e:boundDisMap1}, 
while in the last step we exploited~\eqref{e:UnifModCont} and the fact that $\|\partial\psi_m\|_\infty\lesssim 2^{m}$. 

In conclusion, we have shown that for any $n\geq n_{\eps,m}$
\begin{equ}
\sup_{m\in\N}\,\,2^{m\alpha}\sup_{(\sz_n,\sz),(\sw_n,\sw)\in\CC^n}
\|\psi_n(d_n(\sz_n,\sw_n))\,\delta_{\sz_n,\sw_n}M_n-\psi_n(d_\infty(\sz,\sw))\,\delta_{\sz,\sw}M_\infty\|<\eps
\end{equ}
from which the result follows at once. 
\end{proof}

\begin{proof}[of Lemma~\ref{l:Approx}]
The proof follows closely that of the above lemma. 
Let $\CC^\delta$ be the correspondence given by $\{(\sz,\sz')\in\ST\times T\,:\,d(\sz,\sz')\leq \delta\}$. Then, for every 
$(\sz,\sz'),\,(\sw,\sw')\in\CC^\delta$, we have
\begin{equ}
|d(\sz,\sw)-d(\sz',\sw')|\leq 2\delta\,,\qquad\|M(\sz)-M(\sz')\|\leq\|M\|_\alpha d(\sz,\sz')^\alpha\leq \|M\|_\alpha \delta^\alpha
\end{equ}
so that the first two summands in~\eqref{e:MetC} are controlled. 
For the other, let $(\sz,\sz'),(\sw,\sw')\in\CC^\delta$ and $\bar m\in\N$ be the biggest integer for which $2^{-\bar m}<\sqrt{\delta}$. 
Then, for $m>\bar m+2$, we bound $\psi_m$ by $1$, so that
\begin{equs}
\|\psi_m(d(\sz,\sw))\delta_{\sz,\sw}M-\psi_m(d(\sz',\sw'))\delta_{\sz',\sw'}M\|&\leq 2\omega (M, 2^{-m+2}) \\
&\lesssim 2^{-m\alpha} (\delta^{-\alpha/2}\omega(M,\sqrt{\delta}))\,.
\end{equs}
 For $m\leq\bar m+2$ we proceed as in~\eqref{e:Hol2} and get
 \begin{equs}
 \|\psi_m&(d(\sz,\sw))\delta_{\sz,\sw}M-\psi_m(d(\sz',\sw'))\delta_{\sz',\sw'}M\|\\
 &\leq \|\delta_{\sz,\sz'}M\|+\|\delta_{\sw,\sw'}M\|+|\psi_m(d(\sz,\sw))-\psi_m(d(\sz',\sw'))|\|\delta_{\sz',\sw'}M\|\\
 &\lesssim \omega (M,\delta)+\|\partial\psi_m\||d(\sz,\sw)-d(\sz',\sw')|\omega(M,2^{-m+3})\\
 &\lesssim 2^{-m\alpha} (\delta^{-\alpha}\omega (M,\delta))+\delta 2^m 2^{-(m-3)\alpha}(2^{(m-3)\alpha}\omega(M,2^{-m+3}))\\
 &\lesssim 2^{-m\alpha}(\delta^{-\alpha}\omega (M,\delta)+\sqrt{\delta})\,,
 \end{equs}
from which the result follows at once. 
\end{proof}

}

\end{appendix}

\bibliographystyle{Martin}

\bibliography{refs}

\end{document}